\documentclass[11pt,reqno]{amsart}
\usepackage{amssymb,latexsym}
\usepackage{amsmath,color}
\usepackage{amsthm}
\usepackage{epsfig}
\usepackage{graphicx}
\usepackage{hyperref}
\usepackage{titletoc}
\usepackage[numbers,sort&compress]{natbib}
\numberwithin{equation}{section}
\newtheorem{theorem}{Theorem}[section]
\newtheorem{proposition}[theorem]{Proposition}
\newtheorem{lemma}[theorem]{Lemma}
\newtheorem{corollary}[theorem]{Corollary}

\newtheorem{definition}{Definition}[section]
\theoremstyle{definition}

\newtheorem{remark}{Remark}[section]

\def\XXint#1#2#3{{\setbox0=\hbox{$#1{#2#3}{\int}$}
\vcenter{\hbox{$#2#3$}}\kern-.5\wd0}}

\def\R{\mathbb{R}_+^{n+1}}

\def\s{\lambda}
\def\chi{p}

\def\e{\varepsilon}

\def\R{\mathbb{R}}

\def\C{C}
\def\te{\sigma}

\title[Boundary spike-layer solutions of the singular Keller-Segel system]{Boundary spike-layer solutions of the singular Keller-Segel system: existence, profiles and stability}

\author[J. Carrillo, J. Li, Z.-A. Wang, W. Yang]{Jose A. Carrillo$^{\star}$, Jingyu Li$^{\dagger}$, Zhi-An Wang$^{\ddagger}$, Wen Yang$^{\sharp}$}


\begin{document}
\footnotetext[1]{Mathematical Institute, University of Oxford, Oxford OX2 6GG, UK; carrillo@maths.ox.ac.uk}
\footnotetext[2]{School of Mathematics and Statistics, Northeast Normal University, Changchun, 130024, P.R. China; lijy645@nenu.edu.cn}
\footnotetext[3]{Department of Applied Mathematics, The Hong Kong Polytechnic University, Hung Hom,  Hong Kong; mawza@polyu.edu.hk}
\footnotetext[4]{Department of Mathematics, Faculty of Science and Technology, University of Macau, Taipa, Macau; wenyang@um.edu.mo}

\begin{abstract}
This paper is concerned with the boundary-layer solutions of the singular Keller-Segel model proposed in \cite{KS} in a multi-dimensional domain describing the chemotactic movement of cells up to the concentration gradient of the nutrient consumed by cells, where the zero-flux boundary condition is imposed to the cell while inhomogeneous Dirichlet boundary condition to the nutrient. The steady-state problem of the Keller-Segel system is reduced to a scalar Dirichlet nonlocal elliptic problem with singularity. Studying this nonlocal problem, we obtain the unique steady-state solution which possesses a boundary spike-layer profile as nutrient diffusion coefficient $\e>0$ tends to zero. When the domain is radially symmetric, we find the explicit expansion for the slope of boundary-layer profiles at the boundary and boundary-layer thickness in terms of the radius as $\e>0$ is small, which pinpoints how the boundary curvature affects the boundary-layer profile and thickness. Furthermore, we establish the nonlinear exponential stability of the boundary-layer steady-state solution for the radially symmetric domain. The main challenge encountered in the analysis is that the singularity will arise when the nutrient diffusion coefficient $\e>0$ is small for both stationary and time-dependent problems. By relegating the nonlocal steady-state problem to local problems and performing a delicate analysis using the barrier method and Fermi coordinates, we can obtain refined estimates for the solution of local steady-state problem near the boundary. This strategy finally helps us to find the asymptotic profile of the solution to the nonlocal problem as $\e \to 0$ so that the singularity is accurately captured and hence properly handled to achieve our results. For the time-dependent problem in the radially symmetric domain, we perform a change of variable to remove the singularity allowing us to finally establish the nonlinear stability of the boundary-layer steady-state solution. We make use of the equation satisfied by the relative radial mass distribution function to the steady state and sophisticated time-weighted energy estimates.
\end{abstract}

\maketitle
\section{introduction}
Proposed in \cite{KS}, the well-known Keller-Segel model with logarithmic chemotactic sensitivity in a smooth bounded domain $\Omega\subset \R^n$ reads as
 \begin{equation}\label{KS0}
	\begin{cases}
		u_t=\Delta u-\nabla\cdot(\chi u\nabla\log w)\quad &\mbox{in}~\Omega,\\
		w_t=\varepsilon\Delta w - uw^\beta&\mbox{in}~\Omega,
 	\end{cases}
\end{equation}
where $u(x,t)$ denotes the bacterial density and $w(x,t)$ the chemical (oxygen or nutrient) concentration at position $x\in \Omega$ and time $t\geq 0$. $\varepsilon\geq 0$ is the chemical diffusion coefficient, $\chi> 0$ denotes the chemotactic coefficient and $\beta\geq0$ the chemical consumption rate. The most prominent feature of the Keller-Segel model \eqref{KS0} lies in the logarithmic sensitivity $\log w$ which leads to a singular at $w=0$. The logarithmic sensitivity was used by Keller and Segel in \cite{KS} based on the Weber-Fechner law to explain the propagation of traveling bands driven by the {E}scherichia {C}oli bacterial chemotaxis observed in the celebrated experiment of Adler \cite{Adler66}, but later was employed to describe many other important biological processes such as the initiation of angiogenesis \cite{LSN, levine2001mathematical}, boundary movement of chemotactic bacteria \cite{Nossal72}, reinforced random walks \cite{LeSl, OS}, boundary layer formation of bacterial chemotaxis \cite{Carrillo,Tuval}, and so on. The mathematical derivation of logarithmic sensitivity was given in \cite{LeSl, OS} based on the random-walk framework. It was experimentally confirmed in \cite{Kalinin} that {E}scherichia {C}oli bacteria use logarithmic sensing for chemotactic movement.

Mathematically the singular logarithmic sensitivity was necessary generate traveling wave solutions from the system \eqref{KS0} (cf. \cite{KS,LuiWang}).
This singularity brings various challenges to analysis but attracts immense attention due to the interest in its own right.   Up to now, most studies are limited to the case $\beta=1$ except the existence of traveling wave solutions (cf. \cite{Wang-DCDSB-review}) and a recent work \cite{Carrillo} on the boundary-layer solutions in one dimension. When $\beta=1$, a clever Cole-Hopf type transformation \cite{LeSl} can be used to remove the singularity. This stimulated massive interesting works, for example the stability of traveling waves, see \cite{Davis, JLW, LW09,LLW, Chae, Choi2} for instance,   global well-posedness of solutions, see \cite{Choi1,Li-Pan-Zhao,peng-ruan-zhu2012global,MWZ,zhang-zhu2007,LPZ,LZ,Wang-Zhao13,TWW} in one-dimensional bounded domain with various boundary conditions or $\R$, and \cite{LLZ,Hao,DL,PWZ,Rebholz,WXY,LPZ,WWZ} in multidimensional spaces, just to mention a few. Among other things, this paper will be focused on the boundary-layer solutions of \eqref{KS0} and hence will only review the relevant results in this direction.

The observation of boundary-layer formation driven by chemotaxis was first reported in \cite{Tuval} where the chemotaxis model was coupled to fluid dynamics, with numerical studies followed in \cite{chertock2012sinking}. The analytical result of boundary-layer solutions of \eqref{KS0} was not available until in \cite{HWZ,HLWW,HWJMPA} where the Neumann boundary condition was imposed. It was shown therein that the (spatial) gradient of the solution instead of the solution itself possessed boundary-layer profile near the boundary. This is not consistent with the experimental observation of \cite{Tuval} where the model was imposed by the zero-flux boundary condition for $u$ while Dirichlet boundary condition for $w$. Therefore the authors \cite{Carrillo} later considered the Keller-Segel system \eqref{KS0} with $\beta\geq 0$ in the half-space $\R^+=[0,\infty)$ endowed with the zero-flux boundary condition for $u$ while Dirichlet boundary condition for $w$ at the boundary $x=0$. The unique steady-state boundary-layer solution was explicitly obtained and shown to be locally asymptotically stable. The work \cite{Carrillo} took advantage of the fact that the steady-state problem of \eqref{KS0} in $\Omega=[0,\infty)$ can be explicitly solved and hence the vanishing limit of the solution as $\varepsilon \to 0$ can be determined. In addition, the technique of taking anti-derivative or working at the level of the mass distribution function can be used in one dimension to establish the stability of the boundary-layer solution. All these advantages can only be used for one-dimensional space, and therefore the multi-dimensional problem still remains open. The main goal of the present work is to fill this gap and consider the singular Keller-Segel system with physical mixed zero-flux and Dirichlet boundary conditions
\begin{equation}\label{KS}
	\begin{cases}
		u_t=\Delta u-\nabla\cdot(\chi u\nabla\log w)\quad &\mbox{in}~\Omega,\\
		w_t=\varepsilon\Delta w - uw&\mbox{in}~\Omega,\\
		(\nabla u-\chi u\nabla\log w)\cdot \nu=0,~w=b\quad &\mbox{on}~\partial\Omega,\\
(u,w)(x,t)=(u_0,w_0)(x)&\mbox{in}~\Omega,
	\end{cases}
\end{equation}
where $b>0$ is a positive constant denoting the boundary value of $w$, and $\nu$ is the unit outer normal vector of $\partial\Omega$. We note that in this paper we consider the case $\beta=1$ only to avoid excessive technicalities.
We start with the steady-state (stationary) problem of \eqref{KS}. First we remark that the integration of the first equation of \eqref{KS} immediately gives
$$\int_\Omega u(x,t)dx=\int_\Omega u_0(x)dx:=m$$
which entails that the mass of $u$ is preserved, denoted by $m>0$, where $u_0\geq (\not \equiv 0)$ denotes the initial value of $u$.

Then the steady-state solutions of \eqref{KS}, denoted by $({U},~{W})$, satisfy
\begin{equation}
	\label{2}
	\begin{cases}
		\Delta {U}-\nabla\cdot(\chi {U}\nabla\log {W})=0\quad &\mbox{in}~\Omega,\\
		\varepsilon\Delta {W} - {U}{W}=0 &\mbox{in}~\Omega,\\
		(\nabla {U}-\chi {U}\nabla\log {W})\cdot \nu=0,~{W}=b\quad &\mbox{on}~\partial\Omega,\\
        \int_\Omega U(x)dx=m.
	\end{cases}
\end{equation}
 Multiplying the first equation of \eqref{2} by $\ln U-\chi \ln W$, and integrating the equation on $\Omega$, we have
\begin{equation}\label{new2.2}
\int_{\Omega} U|\nabla(\ln U-\chi\ln W)|^2dx=0.
\end{equation}
Since we are interested in the nonnegative solutions, we have $U(x)\geq0$ and $W(x)\geq0$ for any $x\in\overline{\Omega}$. Applying the strong maximum principle to the second equation of  \eqref{2}, we have $W(x)>0$ for any $x\in\overline{\Omega}$. We next write the first equation of  \eqref{2} as
\[-\Delta U+\chi\nabla U\frac{\nabla W}{W}+\chi U^2=\chi U\frac{|\nabla W|^2}{W^2}\geq0.\]
Then by the strong maximum principle and the Hopf's boundary point lemma along with the fact $\int_{\Omega} Udx=m$, one has $U(x)>0$ for all $x\in \overline{\Omega}$. Thus, it follows from \eqref{new2.2} that
\[\ln U-\chi\ln W=c_0\]
for an arbitrary constant $c_0$. Therefore, we get a constant $\lambda=e^{c_0}>0$ such that
\begin{equation}
	\label{3}
	{U}=\s{W}^\chi, \ \ \s=\frac{m}{\int_\Omega W^\chi dx},
\end{equation}
where the constant $\s=\frac{m}{\int_\Omega W^\chi dx}$ is due to the mass constraint in \eqref{2}.
Then the second equation of \eqref{2} can be rewritten as a nonlocal problem as follows:
\begin{equation}\label{4}
	\begin{cases}
		\varepsilon \Delta {W}=\displaystyle \frac{m}{\int_\Omega W^\chi dx} {W}^{\chi+1}\quad &\mbox{in}\quad \Omega,\\[3mm]
		{W}=b>0\quad &\mbox{on}\quad \partial \Omega.
	\end{cases}
\end{equation}
 As such the steady-state problem \eqref{2} is reduced to a scalar nonlocal problem \eqref{4} with \eqref{3}. That is, the existence of solutions to \eqref{2} is equivalent to the existence of solutions to \eqref{3}-\eqref{4}. The Keller-Segel system \eqref{KS} was considered in a one-dimensional half space $\Omega=[0,\infty)$ in our previous work \cite{Carrillo} and an explicit expression of the unique solution $W$ of \eqref{4} was found. By taking the advantage of this explicit formula, the steady-state solution $(U,W)$ was shown to possess a boundary spike-layer profile (i.e. $U$ is a Dirac mass concentrated at the boundary $x=0$ and $W$ is a boundary-layer profile) as $\e \to 0$ and be nonlinearly asymptotically stable as time tends to infinity. However, the explicit solution of \eqref{4} can not be obtained in a multi-dimensional domain.  As far as we know,  the multi-dimensional problem of \eqref{KS} or \eqref{4} remains open. This paper aims to study the nonlocal elliptic problem \eqref{4} and the time-dependent problem \eqref{KS} in multi-dimensions. Before proceeding, we recall a related nonlocal problem considered in \cite{Lee} given by
 \begin{equation}\label{4n}
	\begin{cases}
		\varepsilon \Delta {W}=\displaystyle \frac{m}{\int_\Omega \mathrm{e}^W dx} W \mathrm{e}^W\quad &\mbox{in}\quad \Omega,\\[3mm]
		{W}=b>0\quad &\mbox{on}\quad \partial \Omega,
	\end{cases}
\end{equation}
which results from the steady-state problem of the Keller-Segel model with linear chemotactic sensitivity (i.e. replacing $\log w$ in \eqref{KS} by $w$ with $p=1$). It was shown in \cite{Lee} that the nonlocal problem \eqref{4n} admits a unique solution which is a boundary-layer profile as $\e \to 0$. When $\Omega$ is radially symmetric, how the boundary curvature influences the boundary-layer profile and thickness was studied in \cite{Lee}.

Compared to \eqref{4n}, the nonlocal problem \eqref{4} is significantly more difficult to study.  A major difference in the analysis lies in the nonlocal term in \eqref{4} which may become singular (i.e. $\int_\Omega W^p dx \to 0$) as $\varepsilon \to 0$ and hence brings considerable difficulties to study the $\e$-vanishing limit of solutions, while the nonlocal term $\int_\Omega \mathrm{e}^W dx$ in \eqref{4n} has an inherent  positive lower bound.  Nevertheless, some ideas developed in \cite{Lee} like the barrier method and the use of the Beltrami Laplacian representing the Euclidean Laplacian strongly inspire a part of the present analysis. This paper has multiple goals as outlined below.
 \begin{enumerate}
 \item[(G1)] Establish the existence and uniqueness of the solution to \eqref{4} in any dimension $n\geq 1$, denoted by $W_\e(x)$, which along with \eqref{3} yields a unique steady-state solution $(U_\e, W_\e)$ of \eqref{KS} satisfying \eqref{2} (see Theorem \ref{th1.1});
 \item[(G2)] Prove that the solution $(U_\e, W_\e)$ is a boundary spike-layer profile as $\e \to 0$, namely $U_\e$ converges to a Dirac mass concentrate at the boundary and $W_\e$ to a boundary layer profile with boundary-layer thickness at the order of $\e$ as $\e \to 0$ (see Theorem \ref{th-existence});
 \item[(G3)] Explore how the boundary curvature affects the boundary-layer profile and thickness when the domain is radially symmetric (see Theorem \ref{radial});
 \item[(G4)] Prove the nonlinear stability of the unique boundary-layer profile $(U_\e, W_\e)$ of Keller-Segel system \eqref{KS} (see Theorem \ref{th-stability}) in the radially symmetric domain.
 \end{enumerate}
\vspace{0.2cm}

\noindent {\bf Strategy of the proofs and Structure of the paper}. As already mentioned, our current polynomial nonlinearity in \eqref{4} is much more difficult to deal with compared to the exponential nonlinearity, since estimating the nonlocal term $\int_\Omega W^p dx$ is significantly more challenging as it is not bounded from below by a positive constant. To overcome this difficulty, we first relegate the nonlocal problem to a local problem and compute the integral $\int_\Omega W^pdx$ based on the constructed sub- and super-solutions for the local problem. It is crucial to observe that as the diffusion coefficient $\e$ tends to zero, the integrals of the sub- and super-solutions share the same leading-order term. This allows us to determine the leading-order behavior of the nonlocal integral for the small $\e>0$. Once this relationship is established, we can effectively treat the nonlocal problem as a local one and extract further quantitative estimates on the solution needed for our purpose. We briefly further elaborate these ideas below.

For our goal (G1), inspired by the idea of \cite{Lee}, we first consider the nonlocal problem \eqref{4} by replacing the term ${m}/{\int_\Omega W^p dx}$ with a constant $\lambda$ and reducing \eqref{4} to a local problem parameterized by $\lambda$. Then we establish the existence and uniqueness of the solution to the local problems by the standard monotone iteration scheme and comparison arguments, a similar strategy is used in \cite{Ca98}. By studying the continuity and the asymptotic limits of $\int_\Omega W^p dx$ with respect to the parameter $\lambda$, we finally obtain the unique solution for the nonlocal elliptic problem \eqref{4}, and as a consequence of \eqref{2} too, via a fixed point argument in the parameter $\lambda$ (see Theorem \ref{th1.1} and its proof in Section 3).

To study the asymptotic behavior of the solution to \eqref{4} near the boundary as $\e \to 0$, i.e. our goal (G2), we treat the product of the diffusion coefficient $\e$ and the integral $\int_\Omega W^pdx$ as a unified diffusion coefficient and reformulate the nonlocal problem as a local one. Then we derive some quantitative results that describe the asymptotic behavior of the reformulated local problem near the boundary by constructing sharp sub-solutions and super-solutions based on the barrier method (cf. \cite{sattinger2006topics}) and a representation of Laplacian operator in terms of Fermi coordinates proved in our previous work \cite{Lee}. Then based on these results, we can find how the integral $\int_\Omega W^p dx$ in the nonlocal problem \eqref{4} depends on $\e$ and finally identify the asymptotic profile of $W$ as $\e \to 0$, which is shown to be a boundary-layer profile with boundary-layer thickness of the order of $\e$. Further with \eqref{3}, we show that $U_\e$ is a Dirac mass concentration at the boundary as $\e \to 0$ (see Theorem \ref{th-existence} with its proof in Section 4 and Remark \ref{bsl}).

In (G3), we are devoted to investigating how the boundary curvature affects the boundary-layer profile and thickness. This turns out to be a very difficult problem due to the nonlocal singular term. Hence we compromise to consider a radially symmetrical domain where the boundary curvature is the reciprocal of the radius.  For the radially symmetric domain, we can gain more detailed estimates on the integral $\int_\Omega W^p dx$ and find the expansion for the slope of the solution on the boundary in terms of the radius as $\e>0$ is small, and further identify how the boundary-layer thickness depends on the radius. It turns out such expansion is not a regular one as expected, but a non-regular one  (see Theorem \ref{radial}-(i) with its proof in Section 5). This is due to the singularity of the nonlocal term in \eqref{4}. In the radially symmetric case, we further can find the explicit asymptotic behavior of the solution (see Theorem \ref{radial}-(ii)) as $p \to \infty$ (strong chemotaxis) which was not found for the general domain. Based on our findings for the radially symmetrical domain, we conjecture that the boundary-layer thickness for the general domain is proportional to the boundary curvature times the volume of the domain, which remains unproved in our present work but has been verified by numerical simulations (see Figure \ref{fig1} and Figure \ref{fig2}).

Finally, we consider our goal (G4) concerning the stability of the unique boundary spike-layer profile $(U_\e, W_\e)$ for the time-dependent problem \eqref{KS}. There are two challenging issues to study the stability. The first one is the singularity at $w=0$, and the second one is that the Dirichet boundary condition on $w$ is insufficient to gain the regularity of gradient of $w$ required by the first equation of \eqref{KS}. To overcome the first barrier, we use a change of variable to transform the problem \eqref{KS} into a new one \eqref{trans} without singularity. To  overcome the second one, we consider the radially symmetrical domain to employ the technique of taking anti-derivative or the radial mass distribution function reducing the order of the first equation of \eqref{trans} so that the Dirichet boundary condition can be fully used. Then we perform sophisticated time-weighted energy estimates to obtain the nonlinear and exponential stability of $(U_\e, W_\e)$ for the radially symmetric domain (see Theorem \ref{th-stability} and its proof in Section 6).
\vspace{0.2cm}

\section{Main results and conjectures}
The first result concerning the existence and uniqueness of solutions is the following.

\begin{theorem}
\label{th1.1}
Let $\Omega$ be a bounded smooth domain in $\mathbb{R}^n~(n\ge1)$ with smooth boundary and let $m$ and $b$ be given positive constants. Then for any $\e>0$, the nonlocal problem \eqref{4} admits a unique positive classical solution $W_\e\in C^1(\overline{\Omega})\cap C^\infty(\Omega)$ and hence the Keller-Segel system \eqref{KS} admits a unique positive classical steady state $(U_\e, W_\e)\in [C^1(\overline{\Omega})\cap C^\infty(\Omega)]^2$ satisfying \eqref{3}.
\end{theorem}

Next we shall characterize the asymptotic profile of the steady-state solution $(U_\e, W_\e)$ as $\e \to 0$, which is a tricky problem since $\varepsilon$ is a singular parameter. We shall show that $U_\e$ is a Dirac measure while $W_\e$ is a boundary-layer profile near the boundary as $\e \to 0$. To state our results, for any constant $\rho>0$, we define $\Omega_\rho$ as (see an illustration in Figure \ref{fig}-(a))
\begin{equation}\label{Omega}
\Omega_\rho=\{x\in\Omega\mid 0<\mathrm{dist}(x,\partial\Omega)<\rho\},
\end{equation}
and we denote by $\Omega_\rho^c$ its complement in $\Omega$.

\begin{figure}[h]
\centering

\includegraphics[width=5.5cm]{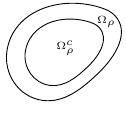} \hspace{1cm}
\includegraphics[width=6cm]{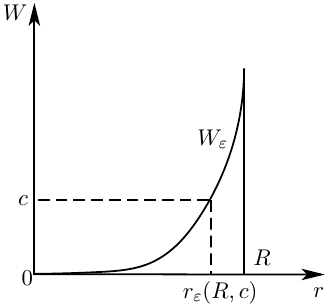}

(a) \hspace{6.5cm} (b)
\caption{(a) A schematic of domain $\Omega_\rho$. (b) Illustration of radial boundary-layer thickness}
\label{fig}
\end{figure}

We will say that two functions of one real variable are similar $f \sim g$ if and only if $r_1 f<g<r_2g$  for all values of the variable with some constants $r_1,r_2>0$. We now give the definition of boundary-layer thickness as a function of the positive parameter $\e$.

\begin{definition}
	\label{defi-bl}
	Let $\mu(\e)$ be a non-negative function satisfying $\mu(\e)\to0$ as $\e\to0$ and
 ${W}_\e(x)$ be the solution of \eqref{4}. Denoting $\ell_\e={\rm dist}(x_{in},\partial\Omega)$ for any interior point of $\Omega$,  we say the boundary-layer thickness of $W_\e(x)$ is the same order of $\e$ as $\mu(\e)$ if the the following conditions are fulfilled:
	\begin{enumerate}
		\item  If $\lim\limits_{\e\to0}\frac{\ell_\e}{\mu(\e)}=0$, then $\lim_{\e\to0}{W}_\e(x_{in})=b;$
		
		\item  If $\lim\limits_{\e\to0}\frac{\ell_\e}{\mu(\e)}=L\in(0,\infty)$, then $\lim_{\e\to0}{W}_\e(x_{in})\in(0,b);$
		
		\item  If $\lim\limits_{\e\to0}\frac{\ell_\e}{\mu(\e)}=\infty$, then $\lim_{\e\to0}{W}_\e(x_{in})=0.$
	\end{enumerate}
\end{definition}

Then our second main results are stated below.

\begin{theorem}
\label{th-existence}
Let $\Omega \subset \mathbb{R}^n~(n\geq 1)$ be a bounded domain with smooth boundary. Then for any small fixed constant $\delta>0$, the solution obtained in Theorem \ref{th1.1} satisfies
\begin{equation}
\label{eq1.3-1}
U_\e(x)\sim \frac{1}{\e\left(1+\frac{\mathrm{dist}(x,\partial\Omega)}{\e}\right)^2}, \ \ \ W_\e(x)\sim \left(1+\frac{\mathrm{dist}(x,\partial\Omega)}{\e}\right)^{-\frac2p} \quad\mbox{in}\quad\Omega_\delta,
\end{equation}
and there are some constants $c_1,c_2>0$ independent of $\e$ such that
\begin{equation}
\label{eq1.3-3}
\|U_\e(x)\|_{L^\infty} \leq c_{2}\e \quad\mbox{and}\quad \|W_\e(x)\|_{L^\infty} \leq c_1\e^{\frac2p}\quad\mbox{in}\quad\Omega_\delta^c .
\end{equation}
Moreover, the boundary-layer thickness is of order of $\e$.
\end{theorem}

\begin{remark}\label{bsl}
From \eqref{eq1.3-3}, we can see that $U_\e(x)\to0$ as $\e\to0$ provided $x$ is away from $\partial\Omega$. Combined with the fact $\int_{\Omega}U_\e dx=m$ and the value $U_\e|_{\partial \Omega}$ are the same from \eqref{3} since $W_\e |_{\partial \Omega}=b>0$, we get that in the distribution sense
$$U_\e\rightharpoonup m\delta_{\partial\Omega} \  \ \mathrm{as} \ \ \e \to 0$$
where $\delta_{\partial\Omega}$ denotes the Dirac source located in the $(n-1)$-th Hausdorff dimension set $\partial\Omega$.
From \eqref{eq1.3-3}, we also have that
\begin{equation*}
\label{th1.2-profile}
\lim_{\e\to0}\|{W}_\e\|_{L^\infty(\overline{\Omega_{\delta}^c})}=0,
\end{equation*}
but $W_\e=b>0$ on $\partial \Omega$. Hence $W_\e(x)$ is a boundary-layer profile as $\e \to 0$.
\end{remark}

In Theorem \ref{th1.1}, the existence and uniqueness of non-trivial steady state solutions of the Keller-Segel system \eqref{KS} are established. In Theorem \ref{th-existence}, we further show that the non-trivial steady state solution $(U_\e,V_\e)$ obtained in Theorem 1.1 is a boundary spike-layer profile as $\e \to 0$ and find the explicit asymptotic profile of $(U_\e,V_\e)$ near the boundary as $\e>0$ is small as given in \eqref{eq1.3-1}.  Now we proceed to investigate how the boundary curvature affects the boundary-layer profile and thickness, for which we can only give an answer when $\Omega=B_R(0)$ with radius $R>0$ where the boundary curvature is $\frac{1}{R}$. In this case, we are able to find how the slope of boundary layer solution at the boundary $r=R$ depends on the boundary curvature $1/R$ and hence quantify the boundary-layer thickness in terms of the radius. Below we sketch this idea. We first show the radial boundary layer profile $U_\e(r)$ and $W_\e(r)$ is strictly increasing with respect to $r$, and derive the expansion of their slopes at the boundary $r=R$ in terms of $R$ for small $\e>0$.  Then for given level set such that $W_\e(r_\e)=c\in(0,b)$, the distance from boundary $r=R$ to the point $r_\e$ varies with respect to $R$. To be precise, for any $c\in(0,b)$, we define
\begin{equation}\label{thick}
r_\e(R,c):=W_\e^{-1}(c)\quad\mbox{and}\quad \Gamma_\e(R,c):=\{r\in[0,R]:W_\e\in[c,b]\}=[r_{\e}(R,c),R]
\end{equation}
as functions of $R$ and $c$, where $\Gamma_\e(R,c)$ is a closed interval with width
$R-r_\e(R,c)=O(\e)$
which is nothing but the boundary layer thickness (see an illustration in Figure \ref{fig}-(b)). Then we shall quantitatively expand $R-r_\e(R,c)$ in terms of $R$ and $\e$ up to the first order expansion to see how the boundary layer thickness depends on $R$ and $\e$. Precisely we have the following results.

\begin{theorem}\label{radial}
Let $\Omega=B_R(0)$, where $B_R(0)$ denotes a ball in $\mathbb{R}^n (n\geq 1)$ with radius $R>0$. Then \eqref{2} admits a unique steady-state $(U_\e,W_\e)(r)$ with $r=|x|$, which is radially symmetric and satisfies ${U}_\e'(r)>0, {W}_\e'(r)>0$. Furthermore, the following conclusions hold.
\begin{enumerate}
\item[(i)] As $\varepsilon \to 0$, $(U_\e,W_\e)(r)$ has the following expansion at the boundary:
\begin{equation}\label{asy}
\begin{aligned}
&{U}_\e'(R)=\frac{p^4m^3}{2(2+p)^2 \omega_n^3  R^{3(n-1)}}\frac{1}{\e^2}+O\Big(\frac{\log\e}{\e}\Big), \\
&{W}_\e'(R)=\frac{p mb}{(2+p) \omega_nR^{n-1}}\frac{1}{\e}+O(\log\e),
\end{aligned}
\end{equation}
where $\omega_n$ denotes the surface area of the unit ball in $\mathbb{R}^n$.
Furthermore, the boundary layer thickness has the following expansion
\begin{equation}\label{thick2}
R-r_\e(R,c)=\Big(\big(b/c)^{\frac{p}{2}}-1\Big)\frac{2n(p+2)}{mp^2}\frac{\alpha_n(R)\e}{R}+o_\e(1)\e,
\end{equation}
where $r_\e(R,c)$ is defined in \eqref{thick} and $\alpha_n(R)=\frac{\omega_n}{n} R^n$ denotes the volume of $B_R(0)$.

\item[(ii)] As $\chi\rightarrow\infty$, it holds that for any $\e>0$
\begin{equation}\label{eqn2.15-1}
\omega_nr^{n-1}U_\e(r)\rightarrow m\delta(r-R) \text{ in the sense of distribution},\end{equation}
\begin{equation}\label{eqn2.16-1}
W_\e(r)\rightarrow b \text{ in} \ \C(\overline{B_R}),
\end{equation}
where $\delta(r-R)$ is the Dirac function centered at $r=R$.
\end{enumerate}
\end{theorem}

\begin{figure}[h]
\centering
\includegraphics[width=6.5cm]{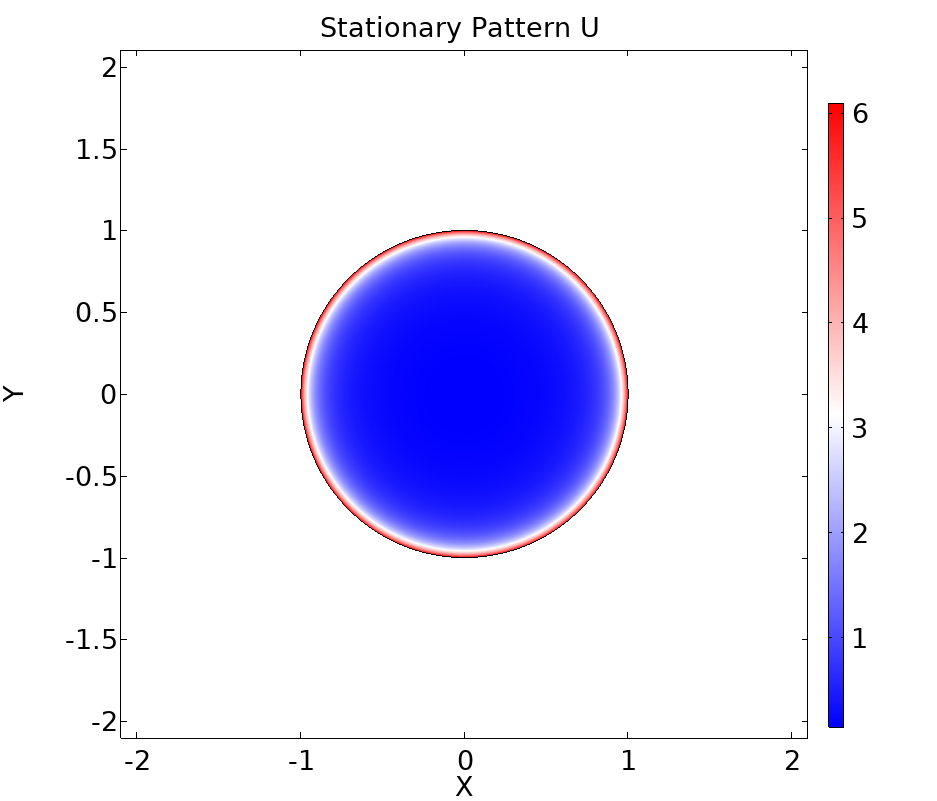}
\includegraphics[width=6.5cm]{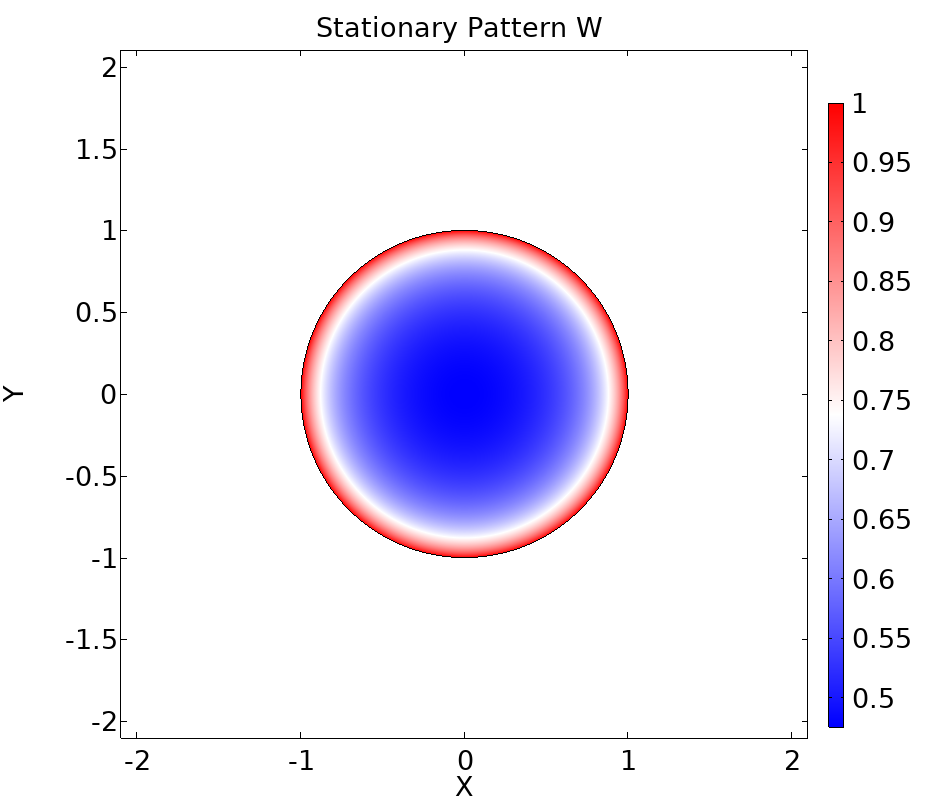}

\includegraphics[width=6.5cm]{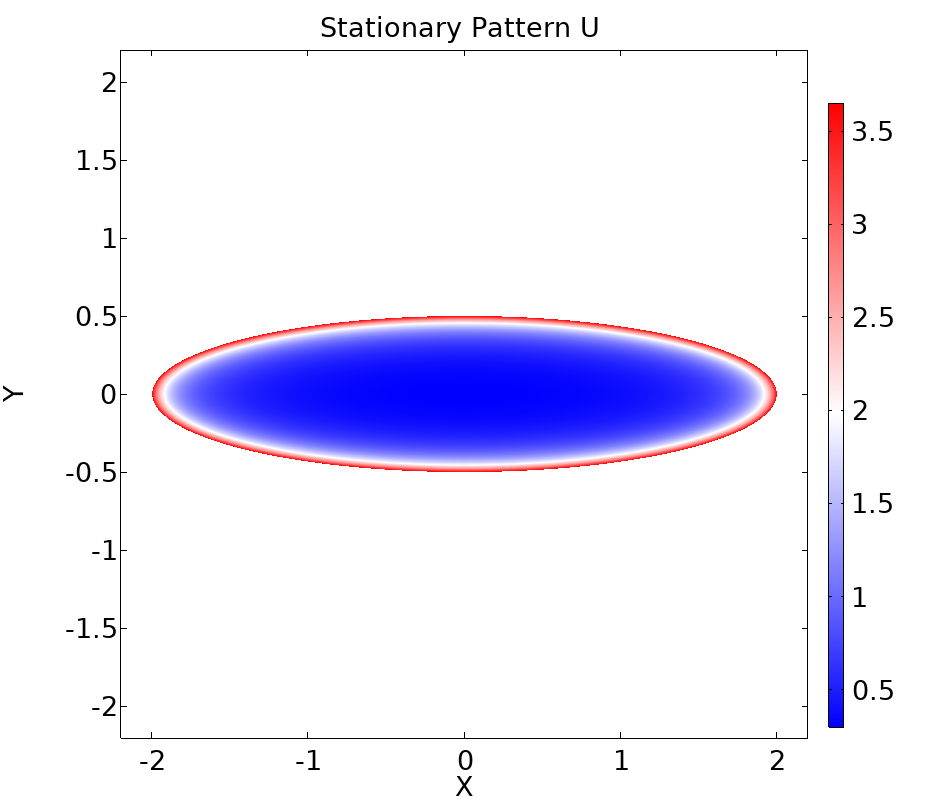}
\includegraphics[width=6.5cm]{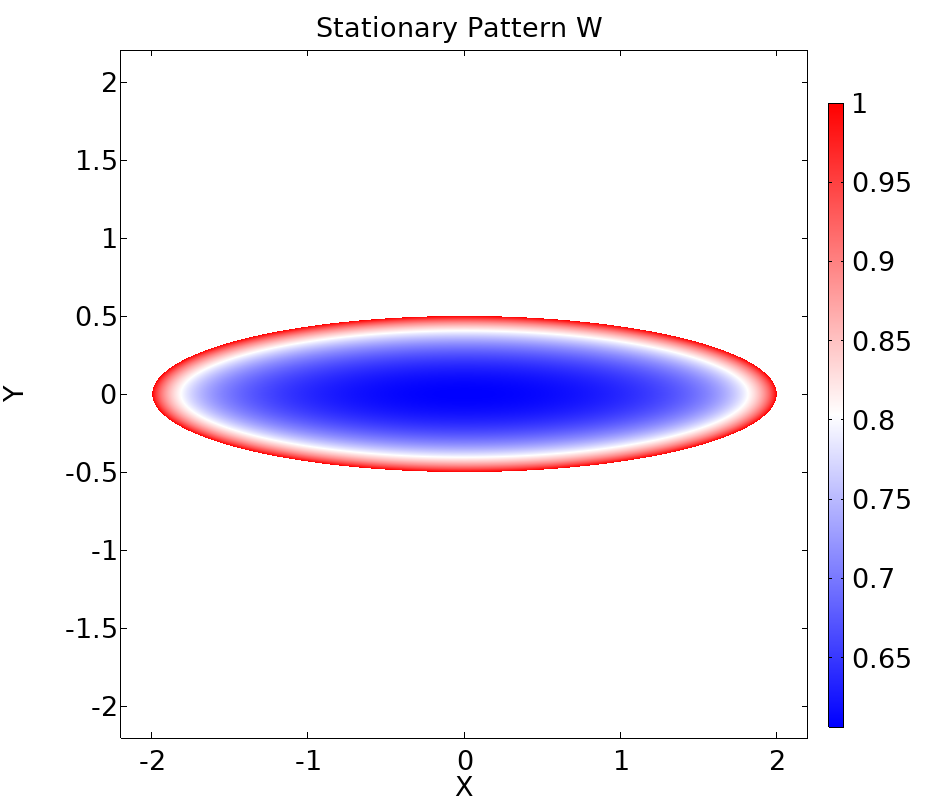}

\caption{Numerical simulations of steady-state boundary layer profiles of \eqref{KS} in a disk (first row) and in an ellipse (second row) with the same area, where the parameter values are $p=5, \e=0.1, b=1$ and initial value $(u_0, w_0)=(1,1)$.   }
\label{fig1}
\end{figure}
\begin{remark}\label{expansion}
In the expansion of ${W}_\e'(R)$ given in \eqref{asy}, one generally expects that the second term should be constant order, but surprisingly it is order of $\log\e$. Consequently, the second order term of ${U}_\e'(R)$ is $\frac{\log\e}{\e}$ instead of $\frac{1}{\e}$. These unexpected results come from the singularity in \eqref{4} (i.e. $\frac{m}{\int_\Omega W^pdx} \to \infty$) as $\e \to 0$.
\end{remark}

\begin{remark}
The expansion \eqref{thick2} asserts that the leading order term of the boundary layer thickness expansion is proportional to the product of the volume $\alpha_n(R)$ of the ball $B_R(0)$ and the boundary curvature $1/R$. We suspect that this is also true for general domains.  That is, we conjecture that the boundary layer thickness is proportional to the product of the volume of the domain and the boundary curvature. Though we are unable to prove this conjecture currently, we shall use two sets of numerical simulations to illustrate this conclusion. In the first set, we choose two domains, one is a disk and the other is an ellipse,  having different geometry but the same area. Then we numerically solve \eqref{KS} in these two domains with the same parameter values and initial data, and plot their numerical results (boundary-layer patterns) in Figure \ref{fig1}, where we see that the boundary-layer thickness is increasing with respect to the boundary curvature. In the second set, we choose one arbitrary domain with obviously different boundary curvatures at different parts of the boundary. The numerical patterns are plotted in Figure \ref{fig2}, where boundary-layer profiles are clearly observed and the boundary-layer is thicker at the boundary with larger curvature. All these numerical results are consistent with our conjecture on the influence of the boundary curvature on the boundary-layer thickness.
\end{remark}

Finally we state the nonlinear stability of the unique radial steady-state solution which is of a boundary-layer profile as $\e$ is small. This result is valid for $1\leq n\leq 3$ by using the Sobolev embedding theorem in the corresponding \emph{a priori} estimates in Proposition \ref{aprioriestimate}.

\begin{theorem}[Nonlinear stability of the radial steady state]\label{th-stability}
Let $\Omega=B_R(0)$ in $\mathbb{R}^n (1\leq n\leq 3)$ and the initial data $(u_0,w_0)$ be radially symmetric with $u_0>0$ and $w_0>0$, and that $u_0\in H^2(B_R)$, $w_0-b\in H_0^1(B_R)\cap H^2(B_R)$. Let $(U,W)$ be the unique steady state obtained in Theorem \ref{radial} with $m=\int_{B_R}u_0(x)dx.$
Then there exists a constant $\delta_1>0$ such that if the initial datum satisfies
\[\|(u_0-U,w_0-W)\|_{H^2(B_R)}\leq\delta_1,\]
the system \eqref{KS}  admits a unique global radial solution  $(u,w)\in\C([0,+\infty);H^2(B_R))$  satisfying
\begin{equation}\label{3.1-2.9}
\|(u-U,w-W)(\cdot,t)\|_{L^\infty(B_R)}\leq C e^{-\mu t},\end{equation}
where $C$ and $\mu$ are positive constants independent of $t$.
\end{theorem}

\begin{figure}[t]
\centering

\includegraphics[width=6.5cm]{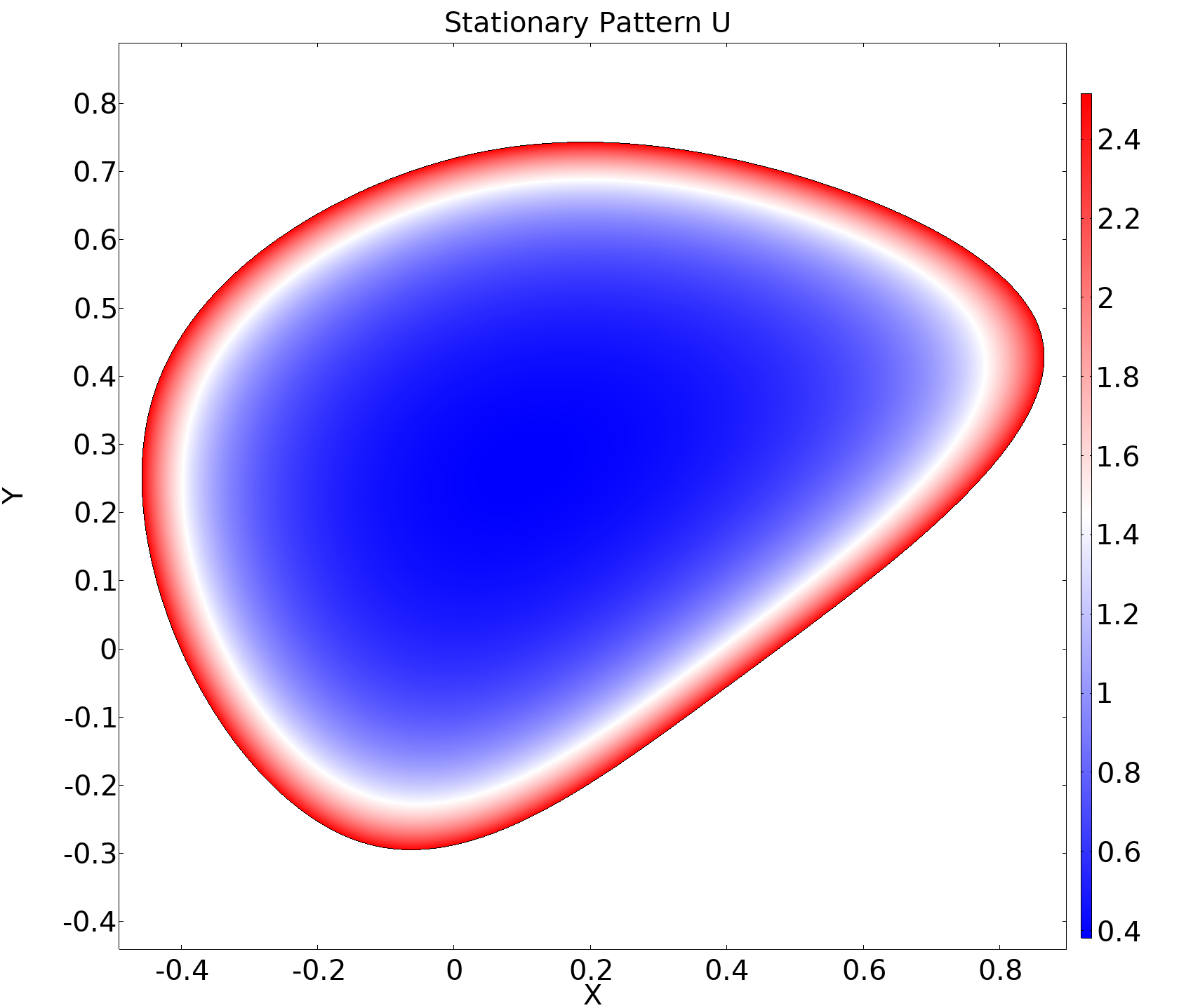}
\includegraphics[width=6.5cm]{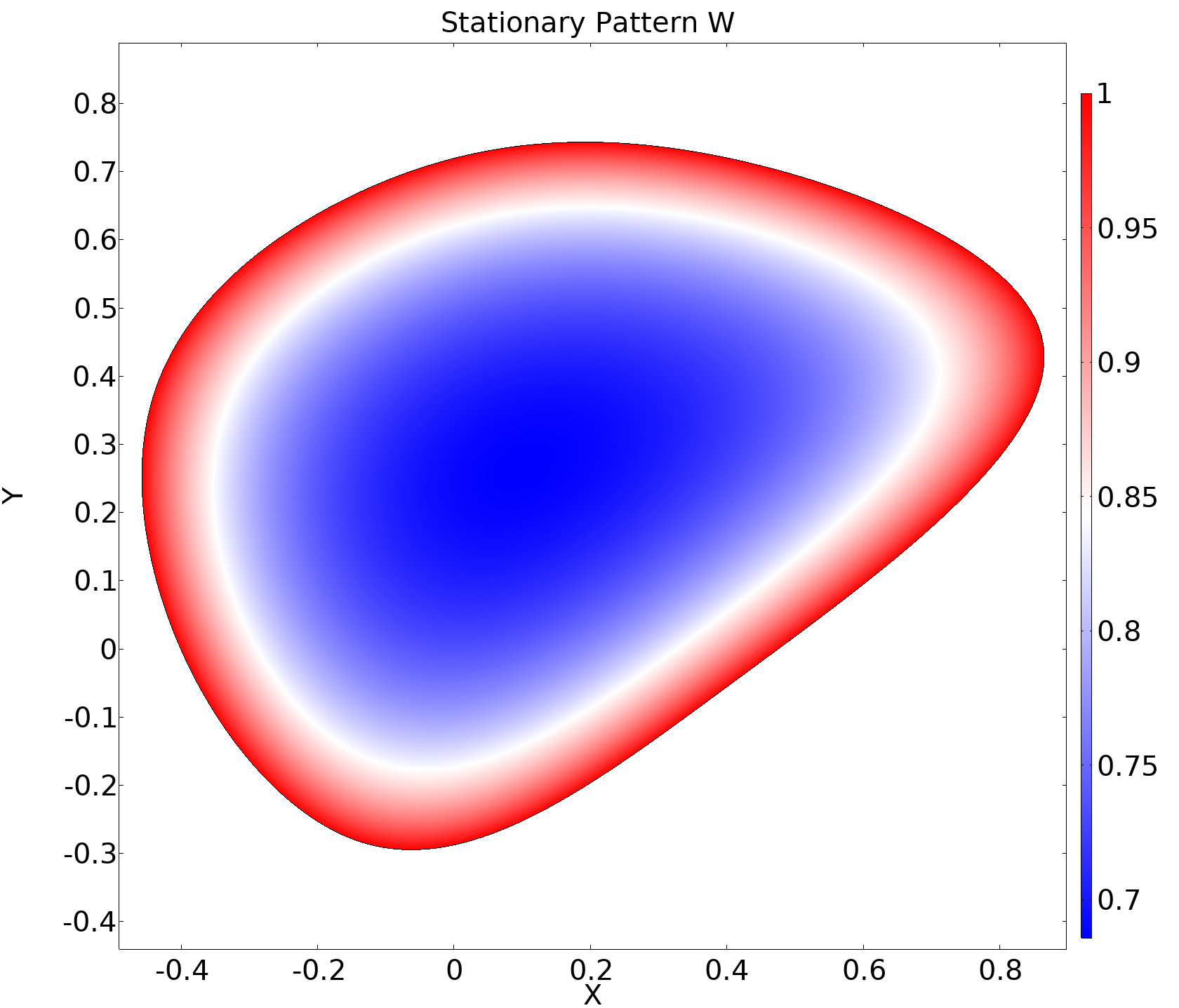}

\caption{Numerical simulations of steady-state boundary layer profiles of \eqref{KS} in a two dimensional general domain, where the parameter values are $p=5, \e=0.1, b=1$ and initial value $(u_0, w_0)=(1,1)$.     }
\label{fig2}
\end{figure}

\section{Existence and uniqueness}
This section is devoted to proving the existence and uniqueness of non-negative solutions to \eqref{2} stated in Theorem \ref{th-existence}. It suffices to prove the existence and uniqueness of solutions to the nonlocal problem \eqref{4} due to the relation given in \eqref{3}.

\subsection{Existence}\label{sec.existence}
In the sequel, without confusion, we shall denote $W_\e$ by $W$. Then the problem \eqref{4} is equivalent to the following local Dirichlet problem
\begin{equation}
\label{6}
\begin{cases}
\varepsilon \Delta {W}=\s{W}^{1+p}\quad &\mbox{in}\quad \Omega,\\
{W} = b>0,\quad &\mbox{on}\quad \partial\Omega
\end{cases}
\end{equation}
subject to the constraint
\begin{equation}\label{cons1}
\lambda \int_\Omega W^pdx=m
\end{equation}
for a given constant $m>0$. Note that we are concerned with positive solution for $W$ only (see the discussion in the Introduction).
Clearly ${W}_{\mathrm{super}}= b$ and ${W}_{\rm{sub}}=0$ are a super-solution and sub-solution of \eqref{6}, respectively. Since the function $f(W)=W^{p+1}$ is increasing, by the method of standard super-sub solutions, we immediately get that for any $\lambda>0$ equation \eqref{6} admits a unique classical positive solution depending on $\lambda$, denoted by  ${W}_\s$, which must be non-constant.  Now it remains to show that there is a unique $\lambda>0$ satisfying \eqref{cons1}, namely $\lambda \int_\Omega W_\s^pdx=m$. It turns out this is difficult since how $W_\s$ depends on $\lambda$ is unknown. Here we overcome this barrier by introducing a change of variable
$${v}_\s=\s^{\frac{1}{p}}W$$
and rewrite \eqref{6} as
\begin{equation}\label{6-u}
\begin{cases}
\e\Delta {v}_\s={v}_\s^{1+p}\quad &\mbox{in}\quad \Omega,\\
{v}_\s=\s^{\frac1p}b>0 &\mbox{on}\quad \Omega,
\end{cases}
\end{equation}
and \eqref{cons1} as
\begin{equation}\label{cons2}
\int_{\Omega}{v}_\s^pdx=m.
\end{equation}
In the new transformed problem \eqref{6-u}-\eqref{cons2}, we see that the parameter $\lambda$ only appears in the boundary condition. The existence and uniqueness of classical solutions to \eqref{6-u} for each given $\lambda>0$ is clearly obtained  by the method of super-sub solutions (or directly from the existence of $W$). The crucial point is to find a unique $\s$ such that the constraint \eqref{cons2} holds for given $m>0$. Next we prove this in the following two steps.
\medskip

\noindent {\bf Step 1} (continuity and monotonicity of $\int_{\Omega}{v}_\s^pdx$ with respect to $\lambda$).
 We denote the solution of \eqref{6-u} by ${v}_\s$ and prove that $\int_{\Omega}{v}_\s^pdx$ is continuous with respect to $\s.$ We start by showing that ${v}_\s(x)$ is non-decreasing with respect to $\s$. Indeed, for any two positive numbers $0<\s_1<\s_2,$ we claim that ${v}_{\s_1}\leq  {v}_{\s_2}$. If it is false, then  ${v}_{\s_1}-{v}_{\s_2}$ admits an interior global maximum point $q\in\Omega$ due to the boundary conditions in \eqref{6-u} and the maximal value is positive. Then we have
\begin{equation}
\label{8}
\e\Delta ({v}_{\s_1}-{v}_{\s_2})={v}_{\s_1}^{1+p}-{v}_{\s_2}^{1+p}>0
\end{equation}
in a small neighbourhood of $q$, which is a contradiction for $q$ to be an interior maximum. Thus the claim is proved.
Next, we observe again that denoting by
${v}_{\s_1,\s_2}={v}_{\s_2}-{v}_{\s_1}$,
then ${v}_{\s_1,\s_2}$ satisfies
\begin{equation}
\label{9}
\begin{cases}
\e\Delta {v}_{\s_1,\s_2}={v}_{\s_2}^{1+p}-{v}_{\s_1}^{1+p}\geq 0\quad &\mbox{in}\quad \Omega,\\[2mm]
{v}_{\s_1,\s_2}=\Big(\s_2^{\frac1p}-\s_1^{\frac1p}\Big){b} &\mbox{on}\quad \partial\Omega.
\end{cases}
\end{equation}
By the strong maximum principle either ${v}_{\s_1,\s_2}$ is constant, which is not true due to \eqref{6-u}, or
\begin{equation}
\label{10}
0\leq {v}_{\s_1,\s_2}<\Big(\s_2^{\frac1p}-\s_1^{\frac1p}\Big){b}\quad \mbox{for}\quad x\in\Omega
\end{equation}
which yields the continuity of ${v}_\s$ with respect to $\s$. As a consequence, the continuity of $\int_{\Omega}{v}_\s^pdx$ with respect to $\s$ is proved. Moreover, the function $\int_{\Omega}{v}_\s^pdx$ is increasing in $\lambda$. Assume otherwise $\int_{\Omega}{v}_{\s_1}^pdx=\int_{\Omega}{v}_{\s_2}^pdx$, since ${v}_{\s_1}\leq  {v}_{\s_2}$ for some $0<\s_1<\s_2,$ then it follows that ${v}_{\s_1}=  {v}_{\s_2}$ in $\bar{\Omega}$, which is not true in view of their continuity and the boundary conditions \eqref{6-u}. In fact, one can additionally prove that ${v}_{\s_1}<{v}_{\s_2}$, this requires an argument using the Hopf-boundary point lemma that we do not detail here for brevity.
\medskip

\noindent {\bf Step 2}. We claim that $\int_\Omega {v}_\s^pdx$ can take any value in $[0,\infty)$ as $\s$ ranges in $[0,+\infty)$. First notice that $\int_\Omega v_\s dx \to 0$ as $\s \to 0$ since $v_\s\to 0$ as $\s \to 0$. Next, we study the case that $\s>0$ is large.
We set
\begin{equation}
\label{11}
{\Theta}(y)=\s^{-\frac1p}{v}_\s(\s^{-\frac12}y).
\end{equation}
Then ${\Theta}(y)$ satisfies
\begin{equation}
\label{12}
\begin{cases}
\e\Delta {\Theta}={\Theta}^{1+p}\quad &\mbox{in}\quad \Omega^\s,\\
{\Theta}=b &\mbox{on}\quad \partial\Omega^\s,
\end{cases}
\end{equation}
where $\Omega^\s$ is defined as
\begin{equation*}
\Omega^\s=\{y\mid \s^{-\frac12}y\in\Omega\}.
\end{equation*}
By standard elliptic regularity theory, we gain that $\|{\Theta}\|_{C^1(\Omega^\s)}\leq {C_b}$ for some uniform positive constant $ {C_b}$. We choose $\ell$ such that $ {C_b}\ell<\frac{b}{2}.$ Then in the following set
\begin{equation}
\label{13}
\Omega^\s_{\ell}=\{y\mid \mbox{dist}(y,\partial\Omega^\s)<\ell\},
\end{equation}
we have
\begin{equation}
\label{14}
{\Theta}(y)\ge{b}- {C_b}\ell \ge\frac{b}{2}.
\end{equation}
It is not difficulty to check that
$$|\Omega^\s_{\ell}|=\ell C_{\Omega}\s^{\frac{n-1}{2}}~ \mbox{for some constant}~ C_\Omega>0~\mbox{depending only on}~\Omega.$$
Therefore
\begin{equation}
\label{15}
\begin{aligned}
\int_{\Omega}{v}_\s^pdx~&=\s^{1-\frac{n}{2}}\int_{\Omega^\s}{\Theta}^p(y)dy
\ge \s^{1-\frac{n}{2}}\int_{\Omega_\ell^{\s}}{\Theta}^p(y)dy\\
&\ge \left(\frac{b}{2}\right)^{p} \ell C_\Omega\s^{1-\frac{n}{2}}\s^{\frac{n-1}{2}}
=\ell C_\Omega \s^{\frac12}\left(\frac{b}{2}\right)^{p}.
\end{aligned}
\end{equation}
Therefore $\int_{\Omega}{v}_\s^pdx\to\infty$ as $\s\to+\infty$, and the claim is proved.

Combining the conclusions in Step 1 and Step 2, by the mean value theorem we can find a $\s$ such that $\int_{\Omega}{v}^pdx=m$ for the given $m$. Then we obtain a solution for \eqref{6-u}-\eqref{cons2}, which gives a solution to \eqref{6}-\eqref{cons1} and hence to \eqref{4}.

\subsection{Uniqueness}\label{sec.uniqueness}
In Section \ref{sec.existence}, the existence of solutions to the nonlocal problem \eqref{4} has been obtained. Now we prove the uniqueness of solutions to \eqref{4}.
Supposing there are two distinct solutions ${W}_1$, ${W}_2$, we shall prove ${W}_1\equiv {W}_2$ by the argument of contradiction and divide our analysis into two steps.

 {\it Step 1}. We prove that either ${W}_1\geq {W}_2$ or ${W}_1\leq {W}_2$. Without loss of generality, we may assume $\int_{\Omega}{W}_1^pdx\geq\int_{\Omega} {W}^p_2dx$. Under this assumption, we claim that ${W}_1\geq {W}_2$. If it is false, then there exists a point $q\in \Omega$, such that
$$({W}_1-{W}_2)|_{q}=\min_\Omega({W}_1-{W}_2)<0.$$
As a consequence, we have
\begin{equation*}
\left.\left(\frac{{W}_1^{1+p}}{\int_{\Omega}{W}_1^pdx}-\frac{{W}_2^{1+p}}{\int_{\Omega}{W}_2^pdx}\right)\right|_{q}<0,
\end{equation*}
which yields
\begin{equation*}
\left.\left[\varepsilon\Delta({W}_1-{W}_2)-m\left(\frac{{W}_1^{1+p}}{\int_{\Omega}{W}_1^pdx}-\frac{{W}_2^{1+p}}{\int_{\Omega}{W}_2^pdx}\right)\right]\right|_q>0.
\end{equation*}
Contradiction arises. Thus the claim holds. Therefore for any two solutions ${W}_1$ and ${W}_2$, either ${W}_1\geq {W}_2$ or ${W}_1\leq {W}_2$.
\smallskip

{\it Step 2}. Next we prove that if ${W}_1\geq {W}_2$, then ${W}_1={W}_2$. We set $Q=\frac{{W}_1}{{W}_2}$, it is obvious that
$$Q\geq 1~\mbox{in}~\Omega\quad\mbox{and}\quad Q=1~\mbox{on}~\partial\Omega.$$
Suppose that $Q\neq 1$ and
$$Q(q_0)=\max_{\Omega}Q>1.$$
Then
\begin{equation*}
\frac{{W}_1^p(q_0)}{{W}_2^p(q_0)}\geq\frac{{W}_1^p}{{W}_2^p}\quad\mbox{in}\quad \Omega.
\end{equation*}
It implies that
\begin{equation}
\label{2.positive}
\frac{{W}_1^p(q_0)}{{W}_2^p(q_0)}>\frac{\int_{\Omega}{W}_1^pdx}{\int_{\Omega}{W}_2^pdx},\quad\mbox{and}\quad
\left.\left(\frac{{W}_1^{p}}{\int_{\Omega}{W}_1^pdx}-\frac{{W}_2^{p}}{\int_{\Omega}{W}_2^pdx}\right)\right|_{q_0}>0.
\end{equation}
On the other hand, it is known that
\begin{equation}
\label{2.banlance}
\begin{aligned}
0\geq \Delta Q|_{q_0}=~&\left.\nabla\cdot\left(\frac{\nabla {W}_1\cdot {W}_2-{W}_1\cdot \nabla {W}_2}{{W}_2^2}\right)\right|_{q_0}\\
=~&\left.\left(\frac{\Delta {W}_1\cdot {W}_2-{W}_1\cdot\Delta {W}_2}{{W}_2^2}
-2\nabla Q\cdot\frac{\nabla{W}_2}{{W}_2}\right)\right|_{q_0}\\
=~&\frac{\Delta {W}_1(q_0)\cdot {W}_2(q_0)-{W}_1(q_0)\cdot\Delta {W}_2(q_0)}{{W}_2^2(q_0)}\\
=~&\frac{mW_1(q_0){W}_2(q_0)}{\varepsilon{W}_2^2(q_0)}\left.\left(\frac{{W}_1^{p}}{\int_{\Omega}W_1^pdx}-\frac{{W}_2^{p}}{\int_{\Omega}{W}_2^pdx}\right)\right|_{q_0},
\end{aligned}
\end{equation}
where we have used $\nabla Q(q_0)=0$ since $q_0$ is the maximal point of $Q$ in $\Omega.$ Using \eqref{2.positive} we see that the right hand side of \eqref{2.banlance} is positive, then contradiction arises and $Q\equiv 1$ holds. Thus, we finish the proof.

\begin{proof}[\bf Proof of Theorem \ref{th1.1}]
The existence and uniqueness of the solution to \eqref{4} has been proved in Section \ref{sec.existence} and Section \ref{sec.uniqueness}.  Since the existence of \eqref{2} and \eqref{4} has one-to-one correspondence via \eqref{3}, we obtain the existence of unique positive solution for \eqref{2} and complete the proof.
\end{proof}

\section{boundary-layer profile and thickness}
This section is devoted to the proof of Theorem \ref{th-existence}. We start with the following auxiliary problem
\begin{equation}
\label{3.bd1}
\begin{cases}
\e\Delta v=v^{1+p}\quad &\mbox{in}\quad \Omega,\\
v=b &\mbox{on}\quad \partial\Omega.
\end{cases}
\end{equation}

\begin{lemma}
\label{le4.1}
Let $v_\e\in C^\infty(\overline{\Omega})$ be the unique solution of \eqref{3.bd1}. For any compact subset $K\subset \Omega$ and sufficiently small $\e>0$, there exists a positive constant $C_K$ independent of $\e$ such that
\begin{equation}
\label{3.b-1}
\max_Kv_\e\leq C_K \e^{\frac{1}{p}}.
\end{equation}
\end{lemma}

\begin{proof}
Let us first remark that we can reduce to analyze the case $b=1$ in \eqref{3.bd1} by a simple scaling argument. In fact, take $\bar v=v/b$, it is simple to show that it satisfies
\begin{equation}
\label{3.b-1aux}
\begin{cases}
\e b^{-p}\Delta \bar v=\bar v^{1+p}\quad &\mbox{in}\quad \Omega,\\
\bar v=1 &\mbox{on}\quad \partial\Omega.
\end{cases}
\end{equation}
Then proving the estimate \eqref{3.b-1} for the problem \eqref{3.bd1} with $b>0$ is equivalent to prove
$$
\max_K \bar v_\e\leq C(K)\e^{\frac{1}{p}}b^{-1}
$$
for the solution of \eqref{3.b-1aux} or equivalently to prove \eqref{3.b-1} for the problem \eqref{3.bd1} with $b=1$. Thus, in the rest of this proof, we assume $b=1$.

When $n=1$, without loss of generality we can assume that $\Omega=[-1,1]$. It is straightforward to check that the following function
$$v_{\e,s}(x)=\left(1+\frac{x+1}{c_p\e^\frac12}\right)^{-\frac{2}{p}}+\left(1+\frac{1-x}{c_p\e^\frac12}\right)^{-\frac{2}{p}},\quad
c_p=\sqrt{\frac{2}{p}\left(\frac{2}{p}+1\right)}.$$
provides a super-solution to the above equation. Indeed, by direct computation
\begin{equation*}
	\e\Delta v_{\e,s}=\left(1+\frac{x+1}{c_p\e^\frac12}\right)^{-\frac{2+2p}{p}}
	+\left(1+\frac{1-x}{c_p\e^\frac12}\right)^{-\frac{2+2p}{p}}
	\leq v_{\e,s}^{1+p}.
\end{equation*}
Together with the trivial fact $v_{\e,s}>1$ at $x=\pm 1$, one can easily conclude that  $v_\e<v_{\e,s}$ by the strong maximal principle. Then \eqref{3.b-1} follows easily.

Now we give the proof for $n\geq2$. For any $q\in \Omega$, we select $R_q$ such that $B_{R_q}(q)\subset\Omega$ and $\partial B_{R_q}(q)\cap\partial\Omega\neq \emptyset.$ Then we consider the solution $\hat v_\e$ of the following intermediate problem
\begin{equation}
\label{3.int}
\begin{cases}
\e\Delta v=v^{1+p}\quad &\mbox{in}\quad B_{R_q}(q),\\
v=1\quad  &\mbox{on}\quad\partial B_{R_q}(q).
\end{cases}
\end{equation}
By standard comparison argument, we obtain that $v_\e\leq \hat v_\e$ in $B_{R_q}(q)$. Next, using the method of moving planes, we see that $\hat v_\e(x)$ is a radially symmetric function with respect to $q$. We write $\hat v_\e(x)=\tilde v_\e(r)$ with $r=|x-q|$. Then it is not difficult to find that $\tilde v_\e(r)$ is a non-decreasing function in $r$ and verifies that
\begin{equation}
\label{3.ode}
\begin{cases}
\e\left(\tilde v''_\e+\frac{n-1}{r}\tilde v_\e'\right)-(\tilde v_\e)^{1+p}=0, \\
\tilde v_{\e}(R_q)=1,\quad \tilde v_\e'(0)=0.
\end{cases}
\end{equation}
We claim that
\begin{equation}
\label{3.claim}
\tilde v_\e\leq
\begin{cases}
2^{\max\{n-2,\frac{n-1}{2},\frac{2}{p}\}}\left(1+\frac{R_q}{2C\e^\frac12}\right)^{-\frac{2}{p}},\quad &\mbox{for}\quad r\in[0,R_q/2],\\[2mm]
2^{\max\{n-2,\frac{n-1}{2},\frac{2}{p}\}}\left(1+\frac{R_q-r}{C \e^\frac12}\right)^{-\frac{2}{p}},\quad &\mbox{for}\quad r\in[R_q/2,R_q],\\
\end{cases}
\end{equation}
where $C$ is some positive constant independent of $\e$ to be determined later. It suffices to prove the claim for $r\in [R_q/2,R_q]$ while the rest one for $r\in [0,R_q/2]$  follows easily by the non-decreasing property of the function. We define a barrier function for $r\in(0,R_q]$ by
$$\tilde v_{\e,l}=\left(\frac{R_q}{r}\right)^{a}\left(1+\frac{R_q-r}{C\e^\frac12}\right)^{-\frac{2}{p}},$$ where $a$ is a constant determined later. By a straightforward computation, we have
\begin{equation}
\label{3.exp}
	\begin{aligned}
		\e\Delta \tilde{v}_{\e,l}-(\tilde v_{\e,l})^{1+p}&=\e\left(\tilde v''_{\e,l}+\frac{n-1}{r}\tilde v'_{\e,l}\right)-(\tilde v_{\e,l})^{1+p}\\
		&=\e a(a+2-n)\frac{R_q^a}{r^{a+2}}\left(1+\frac{R_q-r}{C\e^\frac12}\right)^{-\frac{2}{p}}\\
		&\quad+\e^\frac12\frac{2}{Cp}(n-1-2a)\frac{R_q^a}{r^{a+1}}\left(1+\frac{R_q-r}{C\e^\frac12}\right)^{-\frac{2+p}{p}}\\
		&\quad+\left(\frac{c_p^2}{C^2}-\left(\frac{R_q}{r}\right)^{ap}\right)\frac{R_q^a}{r^{a}}
		\left(1+\frac{R_q-r}{C\e^\frac12}\right)^{-\frac{2+2p}{p}}.
	\end{aligned}
\end{equation}
We choose $a=\max\{n-2,\frac{n-1}{2},\frac{2}{p}\}$. Then the second term on the right hand-side of \eqref{3.exp} is non-positive. For the first and third terms, we can rewrite them as
\begin{equation*}
\begin{aligned}
&\e a(a+2-n)\frac{R_q^a}{r^{a+2}}\left(1+\frac{R_q-r}{C\e^\frac12}\right)^{-\frac{2}{p}}
+\left(\frac{c_p^2}{C^2}-\left(\frac{R_q}{r}\right)^{ap}\right)\frac{R_q^a}{r^{a}}
\left(1+\frac{R_q-r}{C\e^\frac12}\right)^{-\frac{2+2p}{p}}\\
&=F(\e,r)\frac{R_q^a}{r^{a}}\left(1+\frac{R_q-r}{C\e^\frac12}\right)^{-\frac{2+2p}{p}},
\end{aligned}
\end{equation*}
where
$$F(\e,r)=\frac{1}{r^2}a(a+2-n)\left(\e+\frac{2\e^\frac12}{C}(R_q-r)+\frac{1}{C^2}(R_q-r)^2\right)+\frac{c_p^2}{C^2}-\left(\frac{R_q}{r}\right)^{ap}.$$
Considering the function $F(\e,r)$, in the limit case $\e=0$  we have
\begin{equation*}
\begin{aligned}
r^{ap}F(0,r)=~&\frac{1}{C^2}\left(r^{ap-2}a(a+2-n)(R_q-r)^2+c_p^2r^{ap}\right)-R_q^{ap}\\
\leq~&\frac{1}{C^2}\left(a(a+2-n)+c_p^2-C^2\right)R_q^{ap}.
\end{aligned}
\end{equation*}
Now taking $C>\sqrt{a(a+2-n)+c_p^2+1}$, we see that $F(0,r)$ is negative for $r\in(0,R_q]$. Thus if $\e$ is sufficiently small, we see that $F(\e,r)<0$ for $r\in(0,R_q]$. On the other hand, it is straightforward to check that
\begin{equation}
\label{3.boundary}
\tilde{v}_{\e,l}(R_q)=1\quad\mathrm{and}\quad \lim_{r\to0}\tilde{v}_{\e,l}(r)=\infty.
\end{equation}
By the standard comparison argument we obtain that $\tilde v_{\e}\leq\tilde v_{\e,l}$ for $r\in (0,R_q]$. Then the claim \eqref{3.claim} follows directly. As a consequence, for any point $x\in B_{{R_q}/{2}}(q)$ we have
\begin{equation}
\label{3.est-1}
v_\e(x)\leq\hat v_\e(x)\leq 2^{\max\{n-2,\frac{n-1}{2},\frac{2}{p}\}}\left(1+\frac{R_q}{2C\e^\frac12}\right)^{-\frac{2}{p}}.
\end{equation}
Since $K$ is a compact subset of $\Omega$, there exist finitely many open balls
$$B_{R_{q_j}/2}(q_j)\subsetneq B_{R_{q_j}}(q_j)\subset\Omega~\mbox{with}~q_j\in K,\quad j=1,\cdots,m,$$
such that $K\subset\bigcup_{j=1}^mB_{R_{q_j}/2}(q_j)$. Let $R_0=\min_{1\leq j\leq m}R_{q_j}$. Then by \eqref{3.est-1}, we have
\begin{equation*}
v_\e(x)\leq 2^{\max\{n-2,\frac{n-1}{2},\frac{2}{p}\}}\left(1+\frac{R_0}{2C\e^\frac12}\right)^{-\frac{2}{p}},\quad\forall x\in K,
\end{equation*}
which implies \eqref{3.b-1}. Hence we finish the proof.
\end{proof}

From \eqref{3.b-1}, we shall deduce that $W_\e \to 0$ as $\e \to 0$  for any fixed compact subset $K$ of $\Omega$ (see the proof of Theorem \ref{th-existence} later).  To capture the behavior of $W_\e$ near $\partial\Omega,$ we introduce the Fermi coordinates for any $x\in\Omega_\delta$, that is
	$$X:(y,z)\in\partial\Omega\times\mathbb{R}^+\longmapsto x=X(y,z)=y+z\nu(y)\in\Omega_\delta,$$
	where $\nu$ is the unit normal vector on $\partial\Omega$, and $\Omega_\delta$ is defined in \eqref{Omega}. 
	There is a number $\delta_0>0$ such that for any $\delta\in(0,\delta_0)$, the map $X$ is {from $\Omega_{\delta}$ to a subset of $\mathcal{O}$ (cf. \cite[Remark 8.1]{kowalczyk2011giorgi}, where}
	\begin{equation*}
	\mathcal{O}=\{(y,z)\in\partial\Omega\times(0,2\delta)\}.
	\end{equation*}
	It follows that $X$ is actually a diffeomorphism onto its image $\mathcal{N}=X(\mathcal{O})$.  For any fixed $z$, we set
	\begin{equation*}
	\Gamma_z(y)=\{p\in\Omega\mid p=y+z\nu(y)\}.
	\end{equation*}
	It is straightforward to check that the distance between any point of $\Gamma_z(y)$ and $\partial\Omega$ is $|z|$. Hence we have the following results for the Laplacian operator in terms of Fermi coordinate shown in \cite[Lemma 6.1]{Lee} motivated by \cite[Lemma 10.5]{pacard2003constant}.
\begin{lemma}\label{Fermi}
		The Euclidean Laplacian $\Delta$ can be computed by a formula in terms of the coordinate $(y,z)\in\mathcal{O}$ as
		\begin{equation*}
		\Delta_x=\partial_z^2-H_{\Gamma_z(y)}\partial_z+\Delta_{\Gamma_z},\quad
		x=X(y,z),\quad (y,z)\in\mathcal{O},
		\end{equation*}
		where $\Gamma_z(y)$ is the submanifold
		\begin{equation*}
		\Gamma_z(y)=\left\{y+z\nu(y)\mid y\in\partial\Omega\right\},
		\end{equation*}
		and $H_{\Gamma_z(y)}$ is the mean curvature at the point in $\Gamma_z(y)$ and $\Delta_{\Gamma_z(y)}$ stands for the Beltrami-Laplacian operator on $\Gamma_z(y)$.
	\end{lemma}

\begin{lemma}
\label{le4.2}
Let $\Omega$ be a smooth domain in $\mathbb{R}^n~(n\geq1)$ and $v_\e\in C^{2,\alpha}(\Omega)\cap C^0(\overline\Omega)$ be the unique solution of \eqref{3.bd1}.
Then there exist positive constants $\e_0$ and $\delta_0$ such that for any $\e\in(0,\e_0)$ and $\delta\in\left(0,\min\{\frac12,\delta_0\}\right)$,
it holds that
\begin{equation}
b_1\left(1+b_2\frac{\mathrm{dist}(x,\partial\Omega)}{\e^\frac12}\right)^{-\frac{2}{p}}\leq v_\e\leq b_3\left(1+b_4\frac{\mathrm{dist}(x,\partial\Omega)}{\e^\frac12}\right)^{-\frac{2}{p}}\quad \mbox{in}\quad \Omega_\delta,
\end{equation}
where the definition of $\Omega_{\delta}$ is given in \eqref{Omega}
and $b_1,\cdots,b_4$ are positive constants independent of $\e.$
\end{lemma}

\begin{proof}
When $n=1$, we can repeat almost the same arguments of Lemma \ref{le4.1} for $b=1$ to derive the upper bound, just replacing the barrier function by the following one
$$v_{\e,s}=b\left(\left(1+\frac{b^{\frac{p}{2}}(x+1)}{c_p\e^\frac12}\right)^{-\frac{2}{p}}
+\left(1+\frac{b^{\frac{p}{2}}(1-x)}{c_p\e^\frac12}\right)^{-\frac{2}{p}}\right).$$
While for the lower bound, we set the barrier function by
$$v_{\e,l}=\frac{b}{2}\left(\left(1+\frac{b^{\frac{p}{2}}(x+1)}{c_p\e^\frac12}\right)^{-\frac{2}{p}}
+\left(1+\frac{b^{\frac{p}{2}}(1-x)}{c_p\e^\frac12}\right)^{-\frac{2}{p}}\right).$$
Then
\begin{align*}
\e\Delta v_{\e,l}-v_{\e,l}^{1+p}=\frac{b^{1+p}}{2^{1+p}}&\left(2^{p}
\left(\left(1+\frac{b^{\frac{p}{2}}(x+1)}{c_p\e^\frac12}\right)^{-\frac{2+2p}{p}}
+\left(1+\frac{b^{\frac{p}{2}}(1-x)}{c_p\e^\frac12}\right)^{-\frac{2+2p}{p}}\right)\right.\\
&\left.-\left(\left(1+\frac{b^{\frac{p}{2}}(x+1)}{c_p\e^\frac12}\right)^{-\frac{2}{p}}
+\left(1+\frac{b^{\frac{p}{2}}(1-x)}{c_p\e^\frac12}\right)^{-\frac{2}{p}}\right)^{1+p}\right).
\end{align*}
Using the classical inequality $(a_1+a_2)^{1+p}\leq 2^{p}(a_1^{1+p}+a_2^{1+p})$ with $a_1,a_2>0$, we directly see that the right hand side of the above equation is positive. Then by the fact that $v_{\e,l}\leq b$ at $x=\pm 1$ and comparison argument, it follows  that $v_{\e,l}$ is a sub-solution, and the lower bound for $v_\e$ is provided.

Now we give the proof for $n\geq2$. Without loss of generality, we may assume that $\Omega$ is a simply connected domain for simplicity, while the case for multiply connected domain can be proved similarly. We first derive the lower bound for $v_\e$. Since $\Omega$ is simply connected, $\partial\Omega$ is a smooth connected manifold of dimension $n-1$.  We set $v_{\e,l}$ by
$$v_{\e,l}(x)=(2\delta-{\rm dist}(x,\partial\Omega))\left(1+\frac{{\rm dist}(x,\partial\Omega)}{c_p\e^\frac12}\right)^{-\frac2p}\quad\mbox{for}\quad x\in\Omega_{2\delta}.$$
It is easy to see that
\begin{equation}
v_{\e,l}(x)=\begin{cases}
2\delta, \quad &\mbox{on}\quad \partial\Omega,\\
0, &\mbox{on}\quad\partial\Omega_{2\delta}\setminus\partial\Omega.	
\end{cases}
\end{equation}
A straightforward computation based on Lemma \ref{Fermi} gives
\begin{equation}
\begin{aligned}
\e\Delta v_{\e,l}-v_{\e,l}^{p+1}=~&(\e\partial_z^2-\e H_{\Gamma_z(y)}\partial_z+\e\Delta_{\Gamma_z(y)})
(2\delta-z)\left(1+\frac{z}{c_p\e^\frac12}\right)^{-\frac2p}\\
&-(2\delta-z)^{p+1}\left(1+\frac{z}{c_p\e^\frac12}\right)^{-2-\frac2p}\\
=~&\e^\frac12 H_{\Gamma_z(y)}\left(\frac{4}{p c_p H_{\Gamma_z(y)}}+\frac{2(2\delta-z)}{p c_p}+\e^\frac12 + \frac{z}{c_p}\right)\left(1+\frac{z}{c_p\e^\frac12}\right)^{-1-\frac2p}\\
&+(2\delta-z)(1-(2\delta-z)^{p})\left(1+\frac{z}{c_p\e^\frac12}\right)^{-2-\frac2p}.
\end{aligned}
\end{equation}
We can choose $\delta$ and $\e$ small enough such that
$$\frac{4}{p c_p}+\frac{2(2\delta-z)}{p c_p} H_{\Gamma_z(y)}+\e^\frac12 H_{\Gamma_z(y)} + \frac{z}{c_p} H_{\Gamma_z(y)}\geq 0$$
and $1-(2\delta-z)^p\geq0$ for $z\in(0,2\delta)$. Then
\begin{equation*}
\e\Delta v_{\e,l}-v_{\e,l}^{p+1}\geq0\quad\mbox{in}\quad\Omega_{2\delta}.
\end{equation*}
Together with that $v_\e\geq v_{\e,l}$ on $\partial\Omega_{2\delta}^c$ and the classical comparison argument we have
$$v_\e\geq v_{\e,l} \quad\mbox{in}\quad\overline{\Omega_{2\delta}},$$
which implies that
$$v_\e\geq\delta\left(1+\frac{z}{c_p\e}\right)^{-\frac2p}\quad\mbox{in}\quad\overline{\Omega_{\delta}}.$$
While for the upper bound, we set
$$v_{\e,b}=b\left(1+\frac{b^{\frac{p}2}{\rm dist}(x,\partial\Omega)}{\Lambda c_p\e^{\frac12}}\right)^{-\frac2p},$$
where $\Lambda $ is a large constant to be determined later. A direct computation yields that
\begin{equation}
\label{3.upper-b}
\begin{aligned}
&\e \Delta v_{\e,b}-v_{\e,b}^{p+1}\\
&=b(\e\partial_z^2-\e H_{\Gamma_z(y)}\partial_z+\e\Delta_{\Gamma_z(y)})
\left(1+\frac{b^{\frac{p}2}z}{\Lambda c_p\e^\frac12}\right)^{-\frac2p}
-b^{1+p}\left(1+\frac{b^{\frac{p}2}z}{\Lambda c_p\e^\frac12}\right)^{-\frac{2+2p}p}\\
&=b^{1+p}\left(\frac{1}{\Lambda ^2}-1\right)\left(1+\frac{b^{\frac{p}2}z}{\Lambda c_p\e^\frac12}\right)^{-\frac{2+2p}p}+\e^\frac12 b^{1+\frac12p}\frac{2}{p \Lambda c_p}H_{\Gamma_z(y)}\left(1+\frac{b^{\frac{p}2}z}{\Lambda c_p\e^\frac12}\right)^{-\frac{2+p}p}\\
&=\left(\frac{1}{\Lambda ^2}+\frac{2\e^\frac12}{b^{\frac12p}p \Lambda c_p}H_{\Gamma_z(y)}+\frac{p zH_{\Gamma_z(y)}}{(p+2)\Lambda^2}-1\right)
b^{1+p}\left(1+\frac{b^{\frac{p}2}z}{\Lambda c_p\e^\frac12}\right)^{-\frac{2+2p}p}.
\end{aligned}
\end{equation}
Thus, for any $\e\in(0,1)$ and $z\in(0,2\delta)$ we can always choose $C_{\delta,1}$ sufficiently large such that for $\Lambda\in (C_{\delta,1},+\infty)$, the term in the bracket of the right hand-side is negative. While on $\partial\Omega_{\delta}$, we have $v_{\e,b}=v_{\e}$ on $\partial\Omega$. On $\partial\Omega_{\delta}\setminus\partial\Omega$, we have $v_{\e}\leq C_{\overline{\Omega_{\delta}^c}} \e^{1/p}$ due to Lemma \ref{le4.1}. Then we choose $C_{\delta,2}$ large enough such that for any $\Lambda>C_{\delta,2}$, it holds that
\begin{equation}
\label{3.boundary-b}
b\left(1+\frac{b^{\frac{p}2}\delta}{\Lambda c_p\e^\frac12}\right)^{-\frac2p}\geq C_{\overline{\Omega_{\delta}^c}}\e^{\frac1p}.
\end{equation}
Then we choose $\Lambda=\max\{C_{\delta,1},~C_{\delta,2}\}$ in the definition of $w_{\e,b}$ and using \eqref{3.upper-b}-\eqref{3.boundary-b} we see that $w_{\e,b}$ provides a super-solution to \eqref{3.bd1} in $\Omega_{\delta}$. We finally choose
$$
b_1=\delta,\quad b_2=c_p,\quad b_3=b,\quad b_4=\frac{b^{\frac{p}2}}{\Lambda c_p}
$$
to finish the proof.
\end{proof}

Returning the original nonlocal problem \eqref{4} which can be written as
\begin{eqnarray}\label{4n2}
\left\{
\begin{array}{lll}
 \e \lambda_\e\Delta W_\e= W_\e^{\chi+1},& x\in  \Omega,\\[1mm]
W_\e=b>0,&x\in \partial \Omega,
\end{array}
\right.
\end{eqnarray}
where $\s_\e=\frac{1}{m}\int_{\Omega} W_\e^\chi dx>0$ is a constant. Then
we have the following result.

\begin{lemma}
\label{le3.thick}
Let $W_\e$ be the solution to \eqref{4n2}. Then there exist two positive constants $d_1,~d_2$ such that
\begin{equation*}
d_1\e\leq  \int_\Omega {W_\e}^{p}dx\leq d_2\e.
\end{equation*}
Moreover, for any $\delta>0$ as in Lemma \ref{le4.2}, we obtain a precise behavior at the boundary given by
\begin{equation}\label{profile-1aux}
b_5\left(1+b_6\frac{\mathrm{dist}(x,\partial\Omega)}{\e}\right)^{-\frac{2}{p}}\leq {W_\e}\leq b_7\left(1+b_8\frac{\mathrm{dist}(x,\partial\Omega)}{\e}\right)^{-\frac{2}{p}}\quad \mbox{in}\quad \Omega_\delta,
\end{equation}
for some positive constants $b_5,b_6,b_7,b_8$ independent of $\e$.
\end{lemma}

\begin{proof}
By Lemma \ref{le4.2},  we can find four positive constants $b_5,b_6,b_7,b_8$ which are independent of $\e$, such that
\begin{equation}\label{profile-1}
b_5\left(1+b_6\frac{\mathrm{dist}(x,\partial\Omega)}{\e^\frac12\s_\e^\frac12}\right)^{-\frac{2}{p}}\leq {W_\e}\leq b_7\left(1+b_8\frac{\mathrm{dist}(x,\partial\Omega)}{\e^\frac12\s_\e^\frac12}\right)^{-\frac{2}{p}}\quad \mbox{in}\quad \Omega_\delta,
\end{equation}
while in $\overline{\Omega^c_\delta}$, by the equation \eqref{3.b-1} we can find a positive constant $C_{\overline{\Omega_{\delta}^c}}$ such that
\begin{equation}\label{4.16}
\max_{\overline{\Omega_\delta^c}}{W_\e}(x)\leq C_{\overline{\Omega_{\delta}^c}}\e^\frac1p\s_\e^\frac1p.
\end{equation}
By \eqref{profile-1}-\eqref{4.16}, we can get a lower and upper bound for the term $\int_\Omega {W_\e}^{p}dx$, i.e.,
\begin{equation}\label{4.17}
\begin{aligned}
m \s_\e=\int_\Omega {W_\e}^{p}dx\geq~&\int_{\Omega_\delta}b_5^p\left(1+b_6\frac{\mathrm{dist}(x,\partial\Omega)}{\e^\frac12\s_\e^\frac12}\right)^{-2}dx\\
\geq~&\int_0^\delta\int_{\Gamma_z(y)}b_5^p\left(1+b_6\frac{z}{\e^\frac12\s_\e^\frac12}\right)^{-2}dydz\\
\geq~& b_5^p\e^\frac12\s_\e^\frac12\min_{z\in(0,~\delta)}|\Gamma_z(y)|+O(\e \s_\e ),
\end{aligned}
\end{equation}
and
\begin{equation}\label{4.18}
\begin{aligned}
m \s_\e=\int_\Omega {W_\e}^{p}dx\leq~&\int_{\Omega_\delta}b_7^p\left(1+b_8\frac{\mathrm{dist}(x,\partial\Omega)}{\e^\frac12\s_\e^\frac12}\right)^{-2}dx+
\int_{\Omega_\delta^c}C_{\overline{\Omega_\delta^c}}^p\e\s_\e\\
\leq~&\int_0^\delta\int_{\Gamma_z(y)}b_7^p\left(1+b_8\frac{z}{\e^\frac12\s_\e^\frac12}\right)^{-2}dydz+O(\e \s_\e )\\
\leq~& b_7^p\e^\frac12\s_\e^\frac12\max_{z\in(0,~\delta)}|\Gamma_z(y)|+O(\e \s_\e).
\end{aligned}
\end{equation}
As a consequence of \eqref{4.17}-\eqref{4.18}, we derive that $\e\sim\s_\e$ since $m\s_\e\leq b^p |\Omega|$ by maximum principle comparing to the constant supersolution $b>0$. Then the conclusion of the lemma follows directly.
\end{proof}

With the above preparation, we give the proof of Theorem \ref{th-existence}.
\begin{proof}[\bf Proof of Theorem \ref{th-existence}.]
By \eqref{profile-1aux} in Lemma \ref{le3.thick}, we get the profile for $W_\e$ given in \eqref{eq1.3-1} for any small constant $\delta>0$. From \eqref{3}, we know that
\begin{equation}\label{4.19}
U_\e=m\frac{W_\e^p}{\int_{\Omega}W_\e^pdx}.
\end{equation}
Using Lemma \ref{le3.thick} we obtain the profile for $U_\e$ given in \eqref{eq1.3-1}. Applying Lemma \ref{le4.1} to \eqref{4n2},  we immediately observe that
\begin{equation*}
\|W_\e\|_{L^\infty(\Omega_{\delta}^c)} \leq C\e^{\frac2p}.
\end{equation*}
Then using \eqref{4.19} and Lemma \ref{le3.thick} again, we derive that
\begin{equation*}
U_\e\leq C\e\quad \mbox{for}\quad x\in \Omega_{\delta}^c.
\end{equation*}
Now it remains to show that the boundary-layer thickness is the order $O(\e)$ to finish the proof. Let us denote $\ell_\e={\rm dist}(x_{in},\partial\Omega)$ for any interior point $x_{in}$ of $\Omega$. We just need to check the conditions of Definition \ref{defi-bl} with $\mu(\e)\sim O(\e)$.

{\it Case 1}: If $\lim_{\e\to0}\frac{\ell_\e}{\e}=0$, we set $w^\e(y)={W}_\e(\e y)$. Then $w^\e$ satisfies
\begin{equation}\label{eqn1}
\Delta_y w^\e(y)=\frac{m\e}{\int_{\Omega}W_\e^p(x)dx}({w}^\e(y))^{p+1},
\end{equation}
in $\Omega^\e=\frac1\e \Omega$. Recall that, by maximum principle, we have
$$
\|w^\e(y)\|_{L^\infty(\Omega^\e)}=\|{W}_\e(x)\|_{L^\infty(\Omega)}\leq b.
$$
Following the standard elliptic estimate and the fact that the right-hand side of \eqref{eqn1} is uniformly bounded in $\Omega$ by Lemma \ref{le3.thick}, we get
\begin{equation}
|w^\e(y)|_{L^\infty(\Omega^\e)}+|D_yw^\e(y)|_{L^\infty(\Omega^\e)}\leq C,
\end{equation}
where $C>0$ is a universal constant independent of $\e$. This implies that
$$|D_x{W}_\e(x)|\leq C\e^{-1}.$$
Let $x_0\in\partial\Omega$ be the boundary point such that
$|x_0-x_{\rm in}|={\rm dist}(x_{\rm in},\partial\Omega).$
We get that $|x_0-x_{\rm in}|=o(\e)$ from $\lim_{\e\to0}\frac{\ell_\e}{\e}=0$, then
\begin{equation}
|{W}_\e(x_0)-{W}_\e(x_{\rm in})|\leq C|D_x{W}_\e||x_0-x_{\rm in}|\leq C\e^{-1}|x_0-x_{\rm in}|=o_\e(1).
\end{equation}
This implies that $\lim_{\e\to0}{W}_\e(x_{\rm in})=b$, verifying the statement in Definition \ref{defi-bl}-(1).
\medskip

{\it Case 2}: $\lim_{\e\to0}\frac{\ell_\e}{\e}=L$. In this case, we first show that $\lim\limits_{\e\to 0}{W}_\e(x_{\rm in})>0$. Indeed, by Lemma \ref{le3.thick} and $\lim_{\e\to0}\ell_\e/\e =L,$ we have
$$\lim_{\e\to0}{W}_\e(x_{\rm in})\geq b_5\left(1+b_6\frac{m}{d_1}L\right)^{-\frac{2}{p}}>0.$$
To show $\lim\limits_{\e\to0}{W}_\e(x_{\rm in})<b$, we claim that
\begin{equation}\label{eq0}
{W}_{\e}(x)\leq b\left(1+C_0\frac{\mathrm{dist}(x,\partial\Omega)}{\e}\right)^{-\frac{2}{p}}\quad\mbox{in}\quad \Omega_\delta
\end{equation}
for some suitable positive constant $C_0$. Let $d_2$ be defined in Lemma \ref{le3.thick} and ${W}_{\e,b}$ be the solution of the following equation
\begin{equation}\label{eqnn}
\begin{cases}
\e^2\Delta {W}=\frac{m}{d_2}{W}^{1+p}\quad &\mbox{in}\quad\Omega,\\
{W}=b &\mbox{on}\quad \partial\Omega.
\end{cases}
\end{equation}
By the maximum principle, we get that
${W}_\e\leq {W}_{\e,b}.$
Now we shall prove that
$${W}_{\e,b}\leq {W}_\delta:=b\left(1+b_{9}\frac{\mathrm{dist}(x,\partial\Omega)}{\e}\right)^{-\frac{2}{p}} \quad\mbox{in}\quad \Omega_\delta$$
for some suitable positive constant $b_{9}$. Indeed first we can always choose $b_{9}$ small enough such that
\begin{equation}\label{eq1}
{W}_{\e,b}\leq b\left(1+b_{9}\frac{\mathrm{dist}(x,\partial\Omega)}{\e}\right)^{-\frac{2}{p}}\quad\mbox{on}\quad \partial\Omega_\delta \setminus \partial\Omega.
\end{equation}
To see this, we denote $\Omega_*=\Omega\setminus\Omega_\delta$. Then by Lemma \ref{le4.1} applied to \eqref{eqnn}, we get
\begin{equation}\label{eq2}
\max_{\overline{\Omega_*}}W_{\e,b}\leq C_{\overline{\Omega_*}} \e^{\frac{2}{p}}.
\end{equation}
Then for each $x\in \partial\Omega_\delta \setminus \partial\Omega$, it has $\delta=\mathrm{dist}(x,\partial\Omega)$. By choosing $q_9$ sufficiently small, we will have
\begin{equation}\label{eq3}
 C_{\overline{\Omega_*}} \e^{\frac{2}{p}} \leq b\left(1+b_{9}\frac{\mathrm{dist}(x,\partial\Omega)}{\e}\right)^{-\frac{2}{p}}.
\end{equation}
Then \eqref{eq1} follows from \eqref{eq2} and \eqref{eq3}.  Since ${W}_{\e,b}={W}_\delta=b$ at $\partial\Omega$, we conclude that
$${W}_{\e,b}\leq W_\delta:=b\left(1+b_{9}\frac{\mathrm{dist}(x,\partial\Omega)}{\e}\right)^{-\frac{2}{p}}\mbox{on}\quad \partial\Omega_\delta.$$
By a direct computation, we have
\begin{equation}\label{blabla}
\begin{aligned}
\e^2\Delta {W}_\delta-\frac{m}{d_2}{W}_\delta^{1+p}=~&\left(c_p^2b_{9}^2b+\frac2p\e b_9 H_{\Gamma_z(y)}\left(1+b_{9}\frac{\mathrm{dist}(x,\partial\Omega)}{\e}\right)-\frac{m}{d_2}b^{1+p}\right)\\
&\times\left(1+b_{9}\frac{\mathrm{dist}(x,\partial\Omega)}{\e}\right)^{-2-\frac{2}{p}} \quad\mbox{in}\quad \Omega_\delta.
\end{aligned}
\end{equation}
Since $b_{9}$ is chosen to be sufficiently small, the first bracket in the right-hand side of \eqref{blabla} can be made negative.  By the comparison principle, we infer that
$$
{W}_\e\leq {W}_{\e,b}\leq {W}_\delta\quad\mbox{in }\Omega_\delta.
$$
Hence the claim \eqref{eq0} is proved by identifying $C_0$ with $q_9$ and we get that
$$\lim_{\e\to0}{W}_{\e}(x_{\rm in})\leq b\left(1+C_0L\right)^{-\frac{2}{p}}<b.$$
Thus Definition \ref{defi-bl}-(2) is verifed.
\medskip

{\it Case 3}: $\lim_{\e\to0}\frac{\ell_\e}{\e}=+\infty$. The conclusion in Definition \ref{defi-bl}-(3) is a direct consequence of Lemma \ref{le4.1} in $\Omega_\delta$ and Lemma \ref{le4.2} in $\Omega_\delta^c$.

Collecting the above three cases, we complete the proof.
\end{proof}

\section{The radial case}
In this section, we consider the special case $\Omega=B_R(0):=B_R$ in $\mathbb{R}^n(n\geq 1)$ to find the refined solution structure near the boundary as $\e \to 0$. With this, we can explore how the radius of the domain (and hence the boundary curvature) affects the boundary-layer profile and thickness, and further show the asymptotic profile of the radial steady state as $p \to \infty$ (namely the chemotactic sensitivity is very strong).

\subsection{Asymptotic profile near the boundary}
We first establish the following result.
\begin{lemma}\label{Thm-steady}
Given $b>0$, the system \eqref{2} has a unique smooth positive solution $(U,W)$ that is radially symmetric in $B_R(0)$. Moreover, $(U,W)$ satisfies $U_r>0$ and $W_r>0$ with $r:=|x|$.
\end{lemma}
\begin{proof}
The existence and uniqueness of smooth solutions come from Theorem \ref{th1.1} directly. The fact that the unique solution is radially symmetric is a consequence of Gidas-Ni-Nirenberg theorem \cite{Gidas-Ni-Nirenberg} applied to the following problem
\begin{eqnarray}\label{lambda-equation}
\left\{
\begin{array}{lll}
 \e \lambda_\e\Delta W= W^{\chi+1},& x\in  B_R(0),\\[1mm]
W=b>0,&x\in \partial B_R(0),
\end{array}
\right.
\end{eqnarray}
where $\s_\e=\frac{1}{m}\int_{B_R} W^\chi dx>0$ is a constant. Note that \eqref{lambda-equation} is the analogue of \eqref{4n2} for $\Omega=B_R(0)$. By Theorem \ref{le3.thick} we see that $\s_\e\sim\e$. Next we prove the monotonicity of $(U,W)(r)$. Indeed, the steady-state problem \eqref{2} in the ball $B_R(0)$ can be written as
\begin{eqnarray}\label{eqn2.9}
\left\{
\begin{array}{lll}
U_r=\chi U\frac{W_r}{W},& r\in (0,R),\\[1mm]
\varepsilon W_{rr}+\frac{\varepsilon(n-1)}{r}W_r=UW,& r\in (0,R),\\[1mm]
U_r(0)=W_r(0)=0,\ W(R)=b,\\[1mm]
\omega_n\int_0^Rr^{n-1}U(r)dr=m.
\end{array}
\right.
\end{eqnarray}
Write the second equation of \eqref{eqn2.9} as
\begin{equation}\label{2.13}
\varepsilon(r^{n-1}W_r)_r=r^{n-1}UW.
\end{equation}
Noting $W_r(0)=0$, we integrate \eqref{2.13} over $(0,r)$ and get
\begin{equation}\label{2.16}
W_r(r)>0, \ \forall r\in(0,R].\end{equation}
Using the first equation of \eqref{eqn2.9}, we further get $U_r(r)>0$ for any $r\in(0,R]$. Hence $U(r)$ is monotonically increasing on $[0,R]$.
\end{proof}

In this section, we shall study the boundary expansion for the problem \eqref{2} in the ball. To start with our discussion, we shall first analyze \eqref{lambda-equation}. In the sequel, we set  $\te=\e\lambda_\e$ for simplicity, we can rewrite \eqref{lambda-equation} as
\begin{equation}
\label{5.r-1}
\begin{cases}
\te\left({W}''+\frac{n-1}{r}{W}'\right)-{W}^{1+p}=0,\quad\mathrm{for}\quad r\in(0,R],\\
{W}(R)=b,\quad {W}'(0)=0.
\end{cases}
\end{equation}
We remind the reader that $\te\sim\e^2$ by Theorem \ref{le3.thick}.
Our aim is to derive sharper upper and lower bounds of the solution to \eqref{5.r-1} than in the previous section.

\begin{lemma}
\label{le5.1}
Let ${W}_\e$ be a solution of \eqref{5.r-1}. Then we have $W_{\e,l}(r) \leq {W}_\e(r) \leq W_{\e,u}(r)$ for all $r\in [0,R]$ with
\begin{equation}
\label{5.2}
W_{\e,l}(r):=
b\left(1+\frac{b^{\frac{p}2}(R-r)}{c_p\te^\frac12}\right)^{-\frac{2}{p}}
\!\!,\,
W_{\e,u}(r):=
\begin{cases}
b\left(\frac{R}{r}\right)^{\frac{n-1}{2}}\left(1+\frac{b^{\frac{p}2}(R-r)}{c_{p}\te^\frac12}\right)^{-\frac{2}{p}},~&\mbox{if}~n\geq3,\\
b\left(\frac{R}{r}\right)^{a_{p}}\left(1+\frac{b^{\frac{p}2}(R-r)}{c_{p,1}\te^\frac12}\right)^{-\frac{2}{p}},~&\mbox{if}~n=2,\\
b\left(1+\frac{b^{\frac{p}2}(R-r)}{c_{p}\te^\frac12}\right)^{-\frac{2}{p}}+\frac{c_p^\frac{2}{p}\te^\frac{1}{p}}{R^\frac{2}{p}},~&\mbox{if}~n=1,
\end{cases}
,
\end{equation}
where $c_{p,1}=c_p(1-\frac{a_p^2 \te}{b^pR^2})^{-1/2}$ and $a_p=\max\{\frac12,\frac{2}{p}\}.$
\end{lemma}

\begin{proof}
We denote the left hand-side and right hand-side functions of \eqref{5.2} by ${W}_{\e,1}$ and ${W}_{\e,2}$. First we show that ${W}_{\e,1}\leq {W}_{\e}$. By direct computation we have
\begin{equation}
\label{5.r-2}
\te\left({W}_{\e,1}''+\frac{n-1}{r}{W}'_{\e,1}\right)-{W}_{\e,1}^{1+p}=
\frac{2(n-1)\te^\frac12}{p c_pr}b^{1+\frac{p}2}\left(1+\frac{b^{\frac{p}2}(R-r)}{c_p\te^\frac12}\right)^{-\frac{2+p}{p}}
\geq0.
\end{equation}
Now we claim ${W}_{\e,1}$ is a sub-solution of \eqref{5.r-1}. Indeed, ${W}_{\e,1}(R)=b={W}_{\e}(R)$ and ${W}_{\e,1}-{W}_{\e}$ can not possess an interior local positive maximal value due to maximum principle and \eqref{5.r-2}. At the zero point, we have
$${W}'_{\e,1}(0)-{W}'_{\e}(0)={W}'_{\e,1}(0)>0$$
which entails that $0$ can not be a local maximal point. Thus, we obtain the left hand-side inequality of \eqref{5.2}.

While for the upper bound, we shall divide our discussion into three cases.
\smallskip

\noindent Case 1: $n\geq3$, by direct computation we have
\begin{equation*}
\label{5.r-3}
\begin{aligned}
&\te\left({W}_{\e,2}''+\frac{n-1}{r}{W}_{\e,2}'\right)-{W}_{\e,2}^{1+p}\\
&=\frac{\te(n-1)(3-n)}{4}bR^{\frac{n-1}{2}}r^{-\frac{n+3}{2}}\left(1+\frac{b^{\frac{p}2}(R-r)}{c_p\te^\frac12}\right)^{-\frac{2}{p}}\\&\quad
+b^{1+p}\left(1-\left(\frac{R}{r}\right)^\frac{(n-1)p}{2}\right)\left(\frac{R}{r}\right)^\frac{(n-1)}{2}
\left(1+\frac{b^{\frac{p}2}(R-r)}{c_p\te^\frac12}\right)^{-\frac{2+2p}{p}}\\
&\leq0.
\end{aligned}
\end{equation*}
Together with that ${W}_{\e,2}(R)={W}_{\e}(R)=b$ and ${W}_{\e,2}(r)\to+\infty$ as $r\to0$, we see that ${W}_{\e,2}$ provides a super-solution.
\smallskip

\noindent Case 2: $n=2$, following the computations as we did in \eqref{3.exp} and using $a_{p}=\max\{\frac12,\frac{2}{p}\}$ we have
\begin{equation*}
\begin{aligned}
&\te\left({W}_{\e,2}''+\frac{n-1}{r}{W}_{\e,2}'\right)-{W}_{\e,2}^{1+p}\\
&\leq b\left(\frac{\te a_p^2}{r^2}\left(1+\frac{b^{\frac{p}2}(R-r)}{c_{p,1}\te^\frac12}\right)^2
+\frac{c_p^2b^p}{c_{p,1}^2}-b^p\left(\frac{R}{r}\right)^{p a_p }\right)
\left(\frac{R}{r}\right)^{a_p}\left(1+\frac{b^{\frac{p}2}(R-r)}{c_{p,1}\te^\frac12}\right)^{-\frac{2+2p}{p}}.
\end{aligned}
\end{equation*}
Denoting
\begin{equation*}
F_\e(r,R):=\frac{\te a_p^2}{r^2}\left(1+\frac{b^{\frac{p}2}(R-r)}{c_{p,1}\te^\frac12}\right)^2
+\frac{c_p^2b^p}{c_{p,1}^2}-b^p\left(\frac{R}{r}\right)^{p a_p }.
\end{equation*}
Using the definition of $c_{p,1}$ we can rewrite $F_{\e}(r,R)$ as
\begin{equation*}
F_{\e}(r,R)=b^{p}\left(\frac{a_p^2(R-r)^2}{c_{p,1}^2r^2}+1-\left(\frac{R}{r}\right)^{p a_p }\right)+\frac{\te a_p^2}{r^2}-\frac{\te a_p^2}{R^2}+\frac{2\te^\frac12 a_p^2b^{\frac{p}{2}}(R-r)}{c_{p,1}r^2}.
\end{equation*}
We claim that $F_{\e}(r,R)$ is strictly negative for sufficiently small $\e$ ($\e$ small implies that $\te$ small). Using the fact $p a_p>2$ and the well-known inequality
$$(1+x)^p\geq 1+px+\frac12p(p-1)x^2~\ \mbox{for}~\ x\geq0,\quad \mbox{if}\quad p>2.$$
Then
\begin{equation*}
\begin{aligned}
F_{\e}(r,R)\leq~&b^{p}\left(\frac{a_p^2(R-r)^2}{c_{p,1}^2r^2}+1-\frac{r^2+p a_p(R-r)r
	+\frac12p a_p(p a_p-1)(R-r)^2}{r^2}\right)\\
&+\frac{\te a_p^2}{r^2}-\frac{\te a_p^2}{R^2}+\frac{2\te^\frac12 a_p^2b^{\frac{p}{2}}(R-r)}{c_{p,1}r^2}\\
=~&\frac{b^p(R-r)^2}{r^2}\left(\frac{a_p^2}{c^2_{p,1}}-\frac12p a_p(p a_p-1)\right)
-p a_p b^p\frac{R-r}{r}+\frac{2\te^\frac12 a_p^2b^{\frac{p}{2}}(R-r)}{c_{p,1}r^2}\\
&+\frac{\te a_p^2}{r^2}-\frac{\te a_p^2}{R^2}.
\end{aligned}
\end{equation*}
Consider the coefficient $\frac{a_p^2}{c^2_{p,1}}-p a_p(p a_p-1)$, one can easily check that
\begin{equation*}
\begin{aligned}
\frac{a_p^2}{c^2_{p,1}}-\frac12p a_p(p a_p-1)
=~&\frac{a_p^2}{c_p^2}-\frac1{c_p^2}\frac{a_p^4\e^2}{b^p R^2}-\frac12p a_p(p a_p-1)\\
<~&\frac{a_p^2}{c_p^2}-\frac12p a_p(p a_p-1)=a_p\left(\frac12p-\frac{p^2+p^3}{4+2p}a_p\right)\\
<~&a_p\left(\frac12p-\frac{p^2+p^3}{4+2p}\frac{2}{p}\right)=-\frac{p^2}{4+2p}a_p<-\frac{p}{2+p}.
\end{aligned}
\end{equation*}
As a consequence, we have
\begin{equation*}
F_{\e}(r,R)\leq-\frac{p}{2+p}\frac{b^p(R-r)^2}{r^2}-p a_p b^p\frac{R-r}{r}+\frac{2\te^\frac12 a_p^2b^{\frac{p}{2}}(R-r)}{c_{p,1}r^2}+\frac{\te a_p^2}{r^2}-\frac{\te a_p^2}{R^2}.
\end{equation*}
Then we choose $\te_0$ as
$$\te_0=\left(\min\left\{\dfrac{p a_p b^p}{\frac{8a_p}{R^2}+\frac{8a_p^2b^{\frac{p}{2}}}{c_{p,1} R}},~
\frac{p}{2+p}\frac{b^p R^2 c_{p,1}}{4a_p^2c_{p,1}+8a_p^2b^{\frac{p}{2}}R},~1\right\}\right)^2.$$
One can easily check that if $\te\in(0,\te_0)$ then $F_{\e}(r,R)\leq0$ for $r\in(0,R]$. Thus, we see that ${W}_{\e,2}$ provides a super-solution of \eqref{5.r-1}.
\smallskip

\noindent Case 3: $n=1$. In this case, we consider the original problem for $x\in [-R,R]$ and set
$${W}_{\e,b}:=b\left(1+\frac{b^{\frac{p}2}(R-x)}{c_{p}\te^\frac12}\right)^{-\frac{2}{p}}
+b\left(1+\frac{b^{\frac{p}2}(R+x)}{c_{p}\te^\frac12}\right)^{-\frac{2}{p}}.$$
As we have shown that in last section this provides a super-solution. Together with the trivial fact that ${W}_{\e,b}\leq {W}_{\e,2}$ for $x\in(0,R)$, we finish the proof for the right hand side of \eqref{5.2} in this case.
\end{proof}

In the following lemma, we shall derive a useful expansion of the normal derivative for the solution ${W}_{\e}$ of \eqref{5.r-1}.

\begin{lemma}\label{le5.2}
Let ${W}_{\e}$ be a solution of \eqref{5.r-1}. Then on the boundary $\partial B_R(0)$ we have for $0<\te\ll1$
\begin{equation*}
{W}_{\e}'(R) =\sqrt{\frac{2}{p+2}}b^{1+\frac{p}2}\te^{-\frac12}-\frac{2(n-1)b}{(p+4)R}+o_{\te}(1).
\end{equation*}
\end{lemma}

\begin{proof}
First multiplying both sides of \eqref{5.r-1} by ${W}_{\e}'$, and integrating the result from $0$ to $r$, we get
\begin{equation}
\label{d.1}
\frac12{W}_{\e}'(r)^2=\frac{1}{(p+2)\te}\left({W}_{\e}^{p+2}(r)-{W}_{\e}^{p+2}(0)\right)-\int_0^r\frac{n-1}{s}{W}_{\e}'(s)^2ds.	
\end{equation}
Using \eqref{5.r-1}, we have
\begin{equation}\label{est-int-1aux}
\te r^{n-1}{W}_{\e}'(r)=\int_0^rs^{n-1}{W}_{\e}^{p+1}(s)ds.
\end{equation}
So by the increasing property of ${W}_{\e}$, one has
\begin{equation}
\label{est-int-1}
\te\frac{{W}_{\e}'(r)}{r}=\frac{1}{r^n}\int_0^rs^{n-1}{W}_{\e}^{p+1}(s)ds\leq \frac1n{W}_{\e}^{p+1}(r).
\end{equation}
Then for $r\in\left(\frac{R}{2},R\right)$, it follows that
\begin{equation}
\label{est-int-2}
\begin{aligned}
\te\int_0^r\frac{n-1}{s}{W}_{\e}'(s)^2ds=~&\te\int_0^\frac{R}{4}\frac{n-1}{s}{W}_{\e}'(s)^2ds+\te\int^r_\frac{R}{4}\frac{n-1}{s}{W}_{\e}'(s)^2ds\\
\leq ~&\frac{1}{p+2}\frac{n-1}{n}{W}_{\e}^{p+2}\left(\frac{R}4\right)+\frac{4(n-1)}{R}\int_{\frac{R}{4}}^r\left(\int_0^s{W}_{\e}^{p+1}(\tau)d\tau\right){W}_{\e}'(s)ds,
\end{aligned}
\end{equation}
where \eqref{est-int-1} and \eqref{est-int-1aux} have been used in the first and second term respectively.
For the second term on the right hand side of \eqref{est-int-2}, using Lemma \ref{le4.1} and Lemma \ref{le5.1}, we have
\begin{equation}
\label{est-int-3}
\begin{aligned}
\int_0^s{W}_{\e}^{p+1}(\tau)d\tau\leq~&
\begin{cases}
C\te^{1+\frac1p},\quad &\mathrm{for}~s\in\left(0,\frac{R}{4}\right),\\
C\te^\frac12\left(1+\frac{R-s}{C\te^\frac12}\right)^{-1-\frac{2}{p}}, &\mathrm{for}~s\in\left(\frac{R}{4},R\right),
\end{cases}
\end{aligned}
\end{equation}
where $C$ is a generic positive constant. Particularly, there holds that
\begin{equation}
\label{est-int-4}
\begin{aligned}
\int_0^R{W}_{\e}^{p+1}dr\leq C\te^\frac12.
\end{aligned}
\end{equation}
As a consequence, we have from \eqref{est-int-2} and \eqref{3.b-1} that
\begin{equation}
\label{est-int-5}
\te \int_0^R\frac{n-1}{r}{W}_{\e}'(r)^2dr\leq C\te^{1+\frac{2}{p}}+C\te^\frac12\int_{\frac{R}{4}}^R{W}_{\e}'(r)dr\leq C\te^\frac12.
\end{equation}
Returning  to equation \eqref{d.1}, we can see that $\frac{1}{\te}{W}_{\e}^{p+2}(r)$ is the predominant term on the right hand side for $R-r\leq C\te^{\frac{p+4}{4(p+2)}+\gamma}$ for some sufficiently small positive number $\gamma$. Indeed, with Lemma \ref{le5.1}, we have
$$\frac{1}{\te}{W}_{\e}^{p+2}\geq \frac{1}{\te}b^{p+2}\left(1+\frac{b^{\frac{p}{2}}}{c_p}\te^{-\frac{p}{4(p+2)}+\gamma}\right)^{-\frac{2(p+2)}{p}}\gg C\te^{-\frac12}.$$
Hence for $r\in \Big(0,R-C\te^{\frac{p+4}{4(p+2)}+\gamma}\Big)$, we can apply the Taylor's expansion to rewrite \eqref{d.1} as
\begin{equation*}
\begin{aligned}
{W}_{\e}'(r)=\te^{-\frac12}\sqrt{\frac{2}{p+2}}{W}_{\e}^{\frac{p+2}{2}}(r)\left(1+O\left(\frac{{W}_{\e}^{p+2}(0)}{{W}_{\e}^{p+2}(r)}\right)
+O\left(\frac{\te\int_0^r\frac{n-1}{s}{W}_{\e}'(s)^2ds}{{W}_{\e}^{p+2}(r)}\right)\right)
\end{aligned}
\end{equation*}
where the leading coefficient in the second and third terms are negative. With this, we resort \eqref{est-int-5} to derive that
\begin{equation}
\label{grad-est-0}
\begin{aligned}
\te\int_0^R\frac{n-1}{r}{W}_{\e}'(r)^2dr\geq~&\te\int_{R-C\te^{\frac{p+4}{4(p+2)}+\gamma}}^R\frac{n-1}{r}{W}_{\e}'(r)^2dr\\
\geq~& \frac{(n-1)\te^\frac12}{R}\sqrt{\frac{2}{p+2}}\int_{R-C\te^{\frac{p+4}{4(p+2)}+\gamma}}^R{W}_{\e}^{\frac{p+2}{2}(r)}{W}_{\e}'(r)dr\\
&+O(\te^\frac12)\int_{R-C\te^{\frac{p+4}{4(p+2)}+\gamma}}^R\frac{{W}_{\e}^{p+2}(0)}{{W}_{\e}^{\frac{p+2}{2}}(r)}{W}_{\e}'(r)dr\\
&+O(\te^\frac12)\int_{R-C\te^{\frac{p+4}{4(p+2)}+\gamma}}^R\frac{\te\int_0^r\frac{n-1}{s}{W}_{\e}'(s)^2ds}{{W}_{\e}^{\frac{p+2}{2}}(r)}{W}_{\e}'(r)dr\\
=~&\frac{(n-1)\te^\frac12}{R}\sqrt{\frac{2}{p+2}}\frac{2}{p+4}b^{2+\frac{p}{2}}+O\big(\te^{\frac12+\frac{p+4}{4p+8}-\frac{p+4}{p}\gamma}\big),
\end{aligned}
\end{equation}
While on the other hand, using \eqref{d.1}, we have
\begin{equation}
\label{grad-est-1}
\begin{aligned}
\te\int_0^R\frac{n-1}{r}{W}_{\e}'(r)^2dr=~&\te\int_{R-C\e^{\frac{p+4}{4(p+2)}+\gamma}}^R\frac{n-1}{r}{W}_{\e}'(r)^2dr\\
&+\te\int_0^{R-C\e^{\frac{p+4}{4(p+2)}+\gamma}}\frac{n-1}{r}{W}_{\e}'(r)^2dr\\
\leq~&\frac{(n-1)\te^\frac12}{R-C\te^{\frac{p+4}{4(p+2)}+\gamma}}\sqrt{\frac{2}{p+2}}\frac{2}{p+4}b^{2+\frac{p}{2}}\\
&+\te \int_0^{R-C\te^{\frac{p+4}{4(p+2)}+\gamma}}\frac{n-1}{r}{W}_{\e}'(r)^2dr.
\end{aligned}
\end{equation}
For the second term on the right hand side of \eqref{grad-est-2}, using \eqref{est-int-2}-\eqref{est-int-3}, we see that
\begin{equation}
\label{grad-est-2}
\begin{aligned}
\te\int_0^{R-C\te^{\frac{p+4}{4(p+2)}+\gamma}}\frac{n-1}{r}{W}_{\e}'(r)^2dr
\leq~&\frac{4C(n-1)\te^\frac12}{R}\int_{\frac{R}{4}}^{R-C\te^{\frac{p+4}{4(p+2)}+\gamma}}\frac{{W}_{\e}'(r)}{\left(1+\frac{R-r}{C\te^\frac12}\right)^{1+\frac{2}{p}}}dr\\
& +\frac{1}{p+2}\frac{n-1}{n}{W}_{\e}^{p+2}({R}/{4})\\
\leq~&C\te^{\frac12+\frac{p+4}{4p+8}-\frac{p+4}{p}\gamma},
\end{aligned}
\end{equation}
where the lower bound of $W_\e(r)$ given in Lemma \ref{le5.1} has been used.
From \eqref{grad-est-1} and \eqref{grad-est-2} we get
\begin{equation*}
\label{grad-est-3}
\te\int_0^R\frac{n-1}{r}{W}_{\e}'(r)^2dr\leq \frac{(n-1)\te^\frac12}{R}\sqrt{\frac{2}{p+2}}\frac{2}{p+4}b^{2+\frac{p}{2}}+O\big(\te^{\frac12+\frac{p+4}{4p+8}-\frac{p+4}{p}\gamma}\big).
\end{equation*}
Combined with \eqref{grad-est-0}, we finally arrive at
\begin{equation*}
\label{grad-est-4}
\te\int_0^R\frac{n-1}{r}{W}_{\e}'(r)^2dr= \frac{(n-1)\te^\frac12}{R}\sqrt{\frac{2}{p+2}}\frac{2}{p+4}b^{2+\frac{p}{2}}+O\big(\te^{\frac12+\frac{p+4}{4p+8}-\frac{p+4}{p}\gamma}\big).
\end{equation*}
Evaluating \eqref{d.1} at $R$ and using Lemma \ref{le4.1} yields
\begin{equation}
\label{normal-est}
({W}_{\e}'(R))^2=\frac{1}{\te}\frac{2}{p+2}b^{p+2}-\te^{-\frac12}\frac{n-1}{ R}\sqrt{\frac{2}{p+2}}\frac{4}{p+4}b^{2+\frac{p}{2}}+o_{\te}\left(1\right)\te^{-\frac12}.
\end{equation}
This implies the desired conclusion.
\end{proof}

Next we shall use Lemma \ref{le5.1} and the conclusions obtained in the last section to derive a more accurate estimate on the integration $\int_{B_R(0)}{W_\e}^{p}dx$ of the original nonlocal problem.

\begin{lemma}
\label{le5.3}
Let ${W}_\e$ be a solution of the following nonlocal problem
\begin{equation}
\label{5.nonlocal-p}
\begin{cases}
\e\Delta {W}=\frac{{m}}{\int_\Omega {W}^p dx}W^{1+p}\quad &\mbox{in}\quad B_R(0),\\
{W}=b &\mbox{on}\quad \partial B_R(0),
\end{cases}
\end{equation}
and $\s_\e=\frac{\int_\Omega {W}_\e^p dx}{m}$. By $\omega_n$ we denote the surface area of the unit sphere in $\mathbb{R}^n$, then we have
$$\s_\e=\frac{\omega_n^2b^{p}c_p^2R^{2n-2}}{m^2}\e+O(\e^2\log\e).$$
\end{lemma}

\begin{proof}
In the following, we shall consider the case $n\ge2$, while the case $n=1$ can be treated similarly with simpler calculations. Using Lemma \ref{le5.1} for $n\ge2$ (with $\te$ replaced by $\e\lambda_\e$) and inequality \eqref{3.b-1} we have
\begin{equation*}
b\left(1+\frac{b^{\frac{p}2}(R-r)}{c_{p}\e^\frac12\s_\e^\frac12}\right)^{-\frac{2}{p}}
\leq {W}_\e(r)\leq \begin{cases}
b\left(\frac{R}{r}\right)^{\theta}\left(1+\frac{b^{\frac{p}2}(R-r)}{c_{p}\e^\frac12\s_\e^\frac12}\right)^{-\frac2p},~&\mbox{if}~r\geq \frac{R}{2},\\[5mm]
C_K\e^{\frac1p}\lambda_\e^{\frac1p},~&\mbox{if}~r\leq \frac{R}{2}
\end{cases}
\end{equation*}
where $\theta=\frac{n-1}{2}, \tilde{c}_p=c_p$ for $n\ge3$ and $\theta=a_p, \tilde{c}_p=c_{p,1}$ for $n=2$ with $c_{p,1}=c_p(1-\frac{a_p^2 \e \lambda_\e}{b^pR^2})^{-1/2}$ and $a_p=\max\{\frac{1}{2}, \frac{2}{p}\}$.
As a consequence, we deduce
\begin{equation}
\label{int-lower}
\begin{aligned}
\int_{B_R(0)}{W}_\e^p dx\geq~&\omega_n\int_0^Rr^{n-1}b^p\bigg(1+\frac{b^{\frac{p}2}(R-r)}{\tilde{c}_p\e^\frac12\s_\e^\frac12}\bigg)^{-2}dr\\
=~&\omega_nb^{\frac{p}2}\tilde{c}_p\e^\frac12\s_\e^\frac12\int_0^{\frac{b^{\frac{p}2}R}{\tilde{c}_p\e^\frac12\s_\e^\frac12}}(1+r)^{-2}
\bigg(R-\frac{\tilde{c}_p\e^\frac12\s_\e^\frac12}{b^{\frac{p}2}}r\bigg)^{(n-1)}dr\\
=~&\omega_nb^{\frac{p}2}\tilde{c}_p\e^\frac12\s_\e^\frac12 R^{n-1}-(n-1)\omega_n\tilde{c}_p^2\e \s_\e  \int_0^{\frac{b^{\frac{p}2}R}{c_p\e^\frac12\s_\e^\frac12}}\frac{R^{n-2}r}{(1+r)^{2}}dr+O\left(\e\s_\e\right)\\
=~&\omega_nb^{\frac{p}2}c_pR^{n-1}\e^\frac12\s_\e^\frac12
+\frac12(n-1)\omega_nc_p^2 {R^{n-2}}\e\lambda_\e\log(\e\lambda_\e)
+O\left(\e\s_\e\right)
\end{aligned}
\end{equation}
and
\begin{equation}
\label{int-upper}
\begin{aligned}
\int_{B_R(0)}{W}_\e^p dx=~&\omega_n\int^R_{R/2}r^{n-1}{W}_{\e}^{p}dr+\omega_n\int_0^{R/2}r^{n-1}{W}_{\e}^{p}dr
\\
\leq~&\omega_n\int^R_{R/2}r^{n-1}b^p\left(\frac{R}{r}\right)^{\theta p}
\bigg(1+\frac{b^{\frac{p}2}(R-r)}{\tilde{c}_p\e^\frac12\s_\e^\frac12}\bigg)^{-2}dr+C\e\lambda_\e\\
=~&\omega_nb^{\frac{p}2}\tilde{c}_p\e^\frac12\s_\e^\frac12\int_0^{\frac{b^{\frac{p}2}R}{2\tilde{c}_p\e^\frac12\s_\e^\frac12}}
\frac{\big(R-\frac{\tilde{c}_p\e^\frac12\s_\e^\frac12}{b^{\frac{p}2}}r\big)^{n-1-\frac{p(n-1)}2}R^{\theta p}}{(1+r)^2}dr+O(\e\lambda_\e)\\
=~&\omega_nb^{\frac{p}2}c_{p}\e^\frac12\s_\e^\frac12 R^{n-1}+
\frac{1}{2}(n-1-\theta p)\omega_nc_p^2 { ^{n-2}}\e\lambda_\e\log(\e\lambda_\e)
+O\left(\e \s_\e\right).
\end{aligned}
\end{equation}
Therefore we conclude that
$$
\int_{B_R(0)}{W}_\e^p dx=\omega_nb^{\frac{p}2}c_p R^{n-1}\e^\frac12\s_\e^\frac12 +O(\e\s_\e\log(\e\s_\e)).
$$
It implies that
$$m\s_\e^\frac12=\omega_nb^{\frac{p}2}c_p R^{n-1}\e^\frac12+O(\e^\frac32\log\e).$$
As a consequence, we get
\begin{equation}\label{intest}
\int_{B_R(0)}{W}_\e^p dx=\frac{\omega_n^2b^{p}c_p^2R^{2n-2}}{m}\e+O(\e^2\log\e)
\end{equation}
and the desired conclusion of Lemma \ref{le5.3}.
\end{proof}

\subsection{Asymptotic profile as $\chi \to \infty$}
We can characterize the profile of the steady state $(U,W)$ as $\chi\rightarrow\infty$.
\begin{lemma}\label{Thm-profile}
Let $(U, W)(r)$ be the unique radial solution of \eqref{2} in $B_R(0)$. Then as $\chi\rightarrow\infty$, $U$
concentrates on the boundary $\partial B_R(0)$ and $W$ converges to the boundary value $b$. That is as $\chi\rightarrow\infty$, it holds that
\begin{equation}\label{eqn2.15}
\omega_nr^{n-1}U(r)\rightarrow m\delta(r-R) \text{ in the sense of distribution},\end{equation}
\begin{equation}\label{eqn2.16}
W(r)\rightarrow b \text{ in} \ \C(\overline{B_R}),
\end{equation}
where $\delta(r-R)$ is the Dirac function centered at $r=R$.
\end{lemma}

\begin{proof} The proof consists of three steps.

\emph{Step 1.}  Since $\omega_n\int_0^Rr^{n-1}U(r)dr=m$ and $U(r)$ is monotonically increasing on $[0,R]$, by a contradiction argument, one can see that for any $0<\eta<R$, there exists a constant $M_\eta>0$ such that
\begin{equation}\label{bound U}
U(r)<M_\eta \text{ for any } r\in[0,R-\eta].
\end{equation}
Hence by the Helly's compactness theorem and the diagonal argument, after passing to a subsequence of $\chi\rightarrow\infty$, we have for any $0<\eta<R$
$$U(r)\rightarrow \text{some }U_\infty(r) \text{ pointwise on } [0,R-\eta],$$ and $U_\infty(r)\in L^1(0,R-\eta)$ is monotonically increasing.

\emph{Step 2.} Set $F:=\frac{W_r}{W}$. Then $(U,F)$ satisfies
\begin{eqnarray}\label{2.12}
\left\{
\begin{array}{lll}
U_r=\chi UF,& r\in (0,R),\\[1mm]
\varepsilon F_r+\varepsilon \cdot\frac{n-1}{r}F+\varepsilon F^2=U,& r\in (0,R).
\end{array}
\right.
\end{eqnarray}
By \eqref{2.16}, we get $F(r)>0$ for any $r\in(0,R]$.

We claim $U_\infty\equiv0$ on $[0,R)$. Otherwise, there exists $r_0\in[0,R)$ such that $U_\infty(r_0)>0$. By the monotonicity of $U_\infty$, we assume $r_0>0$. Taking $r_1\in(r_0,R)$, then
\begin{equation}\label{2.14}
U_\infty(r)>U_\infty(r_0):= d_1>0 \text{ for } r\in[r_0,r_1].\end{equation}
Write the second equation of \eqref{2.12} as
\[\varepsilon (r^{n-1}F)_r+\varepsilon r^{n-1}F^2=r^{n-1}U.\]
Integrating this equation over $(0,r)$ for $r\in(r_0,r_1)$ yields
\[\varepsilon r^{n-1}F(r)\leq \int_0^Rr^{n-1}U(r)dr= \frac{m}{\omega_n} \text{ for } r\in[r_0,r_1].\]
Thus, $F(r)\leq\frac{m}{\varepsilon r_0^{n-1}}$, $\forall r\in[r_0,r_1]$. When $\chi$ is large enough, the second equation of \eqref{2.12} further gives
\[\varepsilon F_r(r)\leq U(r)\leq 2U_\infty(r_1) \text{ for } r\in[r_0,r_1].\]
Thus, $F(r)$ is bounded in $\C^1([r_0,r_1])$ with respect to $\chi$. Thanks to the Arzel\'{a}-Ascoli theorem, there exists $F_\infty\in\C[r_0,r_1]$ such that
\[F\rightarrow F_\infty \text{ in }\C([r_0,r_1]) \text{ as } \chi\rightarrow\infty.\]

We next claim
\begin{equation}\label{2.23}
\text{ there exists } \bar{r}\in (r_0,r_1) \text{ such that } F_\infty(\bar{r})>0.
\end{equation}
Otherwise, $F_\infty\equiv0$ on $[r_0,r_1]$. Multiplying the second equation of \eqref{2.12} by a test function $\phi\in \C_0^\infty((r_0,r_1))$, then
\[\varepsilon \int_{r_0}^{r_1}F\phi_rdr+\varepsilon \int_{r_0}^{r_1}\frac{n-1}{r}F\phi dr+\varepsilon \int_{r_0}^{r_1}F^2\phi dr=\int_{r_0}^{r_1} U\phi dr.\]
Sending $\chi\rightarrow\infty$ and using the Lebesgue Dominated Convergence Theorem, we get
\[\int_{r_0}^{r_1} U_\infty\phi dr=0 \text{ for any }\phi\in \C_0^\infty((r_0,r_1)).\]
Thus, $U_\infty\equiv0$ on $(r_0,r_1)$, which contradicts  the assumption \eqref{2.14}.

By \eqref{2.23} and the continuity of the function $F_\infty(r)$ on $[r_0,r_1]$, there exists an interval $[r_2,r_3]\subset[r_0,r_1]$ such that $F_\infty(r)>d_2$ on $[r_2,r_3]$ for some constant $d_2>0$. Now integrating the first equation of \eqref{2.12}, in view of \eqref{2.14}, we have
\[U(r)=U(r_0)e^{\chi\int_{r_0}^rF(\tau)d\tau}\geq\frac{d_1}{2}e^{\chi d_2(r-r_2)} \text{ for } r>r_2.\]
It then follows that
\[m\geq\int_{r_2}^{r_3}r^{n-1}U(r)dr\geq\frac{d_1}{2}\int_{r_2}^{r_3}r^{n-1}e^{\chi d_2(r-r_2)}dr\rightarrow\infty \text{ as }\chi\rightarrow\infty,\]
which is a contradiction. Therefore, $U_\infty\equiv0$ on $[0,R)$, which implies \eqref{eqn2.15} due to $\int_{B_R}U(x)dx=m>0$.

\emph{Step 3.} It remains to show the limit of $W$. On one hand, by the maximum principle,
$0<W(r)\leq b \text{ for }r\in[0,R].$
On the other hand, since $U\left(\frac{R}{2}\right)\leq M_{\frac{R}{2}}$, integrating \eqref{2.13} gives
\begin{equation}\label{2.15}
0\leq\varepsilon  W_r(r)=\frac{1}{r^{n-1}}\int_0^rs^{n-1}UWds\leq bU\left(\frac{R}{2}\right)\frac{1}{r^{n-1}}\int_0^rs^{n-1}ds\leq C \text{ for }r\in\left[0,\frac{R}{2}\right].\end{equation}
Moreover,  for $r\in\left[\frac{R}{2},R\right]$, there holds that
\begin{equation*}\begin{split}
0\leq \varepsilon W_r(r)&=\frac{1}{r^{n-1}}\int_{\frac{R}{2}}^rs^{n-1}UWds+W_r\left(\frac{R}{2}\right)\leq2^{n-1}b\int_0^Rr^{n-1}Uds+W_r\left(\frac{R}{2}\right)\\&\leq2^{n-1}\frac{bm}{\omega_n}+W_r\left(\frac{R}{2}\right)\leq C.\end{split}\end{equation*}
Hence $\|W\|_{\C^1([0,R])}\leq C$. Thanks to the Arzel\'{a}-Ascoli theorem, there exists $W_\infty\in\C([0,R])$ such that
$W\rightarrow W_\infty \text{ in }\C([0,R]) \text{ as }\chi\rightarrow\infty,$
which also indicates that $W_\infty(R)=b.$ Integrating \eqref{2.13} twice gives
\[\varepsilon W(r)-\varepsilon W\left(\frac{R}{2}\right)=\int_{\frac{R}{2}}^r\frac{1}{s^{n-1}}\int_0^s\tau^{n-1}UWd\tau ds.\]
Now sending $\chi\rightarrow\infty$ and recalling $U\rightarrow0$ pointwise on $[0,R)$, by  the Lebesgue Dominated Convergence Theorem, we have
\[W_\infty(r)=W_\infty\left(\frac{R}{2}\right) \text{ for any } r\in[0,R].\]
Therefore, $W_\infty(r)\equiv b$, and we obtain $W\rightarrow b$ in $\C(\overline{B_R})$ as $\chi\rightarrow\infty$.
\end{proof}

\subsection{Proof of Theorem \ref{radial}} First the existence and uniqueness of radially symmetric solution with monotonicity follows from Lemma \ref{Thm-steady}. By Lemma \ref{le5.2} and Lemma \ref{le5.3}, we get
\begin{equation}
\label{nnder}
{W}_\e'(R)=\frac{p mb}{(2+p) \omega_nR^{n-1}}\frac{1}{\e}+O(\log\e)
\end{equation}
which gives the expansion for ${W}_\e'(R)$ in \eqref{asy} directly. With \eqref{3} and \eqref{intest}, the expansion for ${U}_\e'(R)$ in \eqref{asy} is obtained.  In the following, we shall use Lemma \ref{le5.1} (with $\te$ replaced by $\e\lambda_\e$ there) to derive the first order expansion of $R-r_\e(R,c)$ as $\e\to 0$ for a given $c\in(0,b).$ By Lemma \ref{le5.1}, using the sub-solution of \eqref{5.nonlocal-p} given by $W_{\e,l}(r)$, we can solve $W_{\e,l}(r)=c,$ to obtain
we have
\begin{equation*}
r_{1,\e}(c):=W_{\e,l}^{-1}(c)=R-\dfrac{b^{\frac{p}{2}}-c^{\frac{p}{2}}}{b^{\frac{p}{2}}c^{\frac{p}{2}}}c_p\e^\frac12\s_\e^\frac12.
\end{equation*}
Next, we use the supersolution of \eqref{5.nonlocal-p} given by $W_{\e,u}(r)$ in \eqref{5.2} (with $\e$ replaced by $\e\s_\e$ there) and solve $W_{\e,u}(r)=c,$ to get
\begin{equation*}
r_{2,\e}(c):=W_{\e,u}^{-1}(c)=\begin{cases}
R-\dfrac{b^{\frac{p}{2}}-c^{\frac{p}{2}}}{b^{\frac{p}{2}}c^{\frac{p}{2}}}c_p\e^\frac12\s_\e^\frac12+O(\e\s_\e),\quad &\mbox{if}~n\geq2,\\[3mm]
R-\dfrac{b^{\frac{p}{2}}-c^{\frac{p}{2}}}{b^{\frac{p}{2}}c^{\frac{p}{2}}}c_p\e^\frac12\s_\e^\frac12+O(\e^{2/p+1}),\quad &\mbox{if}~n=1.
\end{cases}
\end{equation*}
It is easy to see that $r_{2,\e}(c)\leq r_\e(R,c)\leq r_{1,\e}(c)$, then we have
\begin{equation*}
r_\e(R,c)=R-\dfrac{b^{\frac{p}{2}}-c^{\frac{p}{2}}}{b^{\frac{p}{2}}c^{\frac{p}{2}}}c_p\e^\frac12\s_\e^\frac12+O(\e^{2/p+1}).
\end{equation*}
Using Lemma \ref{le5.3}, we further get that
\begin{equation*}
r_\e(R,c)=R-\dfrac{b^{\frac{p}{2}}-c^{\frac{p}{2}}}{c^{\frac{p}{2}}}\frac{\omega_nc_p^2R^{n-1}}{m}\e+o_\e(1)\e
\end{equation*}
which yields \eqref{thick2} and completes the proof of Theorem \ref{radial}-(i).

The conclusion of Theorem \ref{radial}-(ii) comes from Lemma \ref{Thm-profile}. The proof of Theorem \ref{radial} is completed.

\begin{remark}\label{re-expansion}
The expansion \eqref{nnder} differs from other related problems as \cite{Lee} for the exponential case on the nonlinearity, as mentioned in the introduction, since the second-order term is of order one in that case.
\end{remark}

\section{Nonlinear stability of the radial boundary-layer steady state}

\subsection{A preliminary result}
The monotonicity of radially symmetric steady-state solution $(U,W)(r)$ will be essentially used later to prove the stability result. We further show the following result concerning a sign property.

\begin{lemma}\label{lemma2.6}
Set $V:=\ln W$. Then it follows that
\[\begin{split}
(r^{n-1}V_r)_r>0 \ \text{ for any} \ r\in[0,R].
\end{split}\]

\end{lemma}

\begin{proof} Since $W>0$, dividing the second equation of \eqref{eqn2.9} by $W$, we get
\begin{equation}\label{2.8}
\varepsilon\left(\frac{W_r}{W}\right)_r+\varepsilon\left(\frac{W_r}{W}\right)^2+\varepsilon\cdot\frac{n-1}{r}\cdot\frac{W_r}{W}=U.
\end{equation}
Noting $V_r=\frac{W_r}{W}$, one can see from \eqref{2.8} that $V$ satisfies
\begin{equation}\label{2.20}
\varepsilon(r^{n-1}V_r)_r=r^{n-1}(U-\varepsilon V_r^2).\end{equation}
Since $U(0)>0$ and $V_r(0)=0$, there exists $r_0>0$ such that
\[(r^{n-1}V_r)_r>0, \ \forall \ r\in(0,r_0].\]
We claim $(r^{n-1}V_r)_r>0$  for all  $r\in[0,R].$ Otherwise, denote by $r_1\in(r_0,R]$ the first number such that $(r^{n-1}V_r)_r|_{r=r_1}=0$. In other words
$(r^{n-1}V_r)_r>0$, $\forall r\in(0,r_1)$, and hence the function $r^{n-1}V_r(r)$ is monotonically increasing on $(0,r_1)$.
In view of \eqref{2.20}, we have
\begin{equation}\label{2.26}
U(r_1)-\varepsilon V_r^2(r_1)=0.\end{equation}
Moreover a simple calculation for \eqref{2.20} along with \eqref{2.26}  gives
\begin{equation}\label{2.21}
V_{rr}(r_1)r_1^{n-1}=-(n-1)r_1^{n-2}V_r(r_1).\end{equation}
Differentiating \eqref{2.20} in $r$ gives
\[\varepsilon(r^{n-1}V_r)_{rr}=r^{n-1}(U_r-2\varepsilon V_rV_{rr})+(n-1)r^{n-2}(U-\varepsilon V_r^2).\]
Now by \eqref{2.26} and \eqref{2.21}, we have
\[\begin{split}
\varepsilon(r^{n-1}V_r)_{rr}|_{r=r_1}&=r_1^{n-1}(U_r-2\varepsilon V_rV_{rr})|_{r=r_1}\\
&=r_1^{n-1}U_r(r_1)+2\varepsilon(n-1)r_1^{n-2}V_r^2(r_1)>0,
\end{split}\] where we have used the monotonicity of $U(r)$ in the last inequality (see Theorem \ref{Thm-steady}).
Thus the function $r^{n-1}V_r(r)$ takes a local minimum at $r=r_1$, which contradicts  the fact that it is monotonically increasing on $(0,r_1)$. Therefore, $(r^{n-1}V_r)_r>0$  for all  $r\in[0,R].$
\end{proof}

\subsection{Nonlinear stability of $(U,V)(r)$ on $B_R(0)$}\label{nonlinear stability}

In this section, we study the asymptotic stability of the steady state $(U,W)$ to the system \eqref{KS}  under radial perturbations. Since the $n=1$ case has been achieved in our previous work \cite{Carrillo}, we only consider the case $n=2,3$.

\begin{theorem}\label{Thm-stability} Let $n=2,3$.
Assume that the initial datum $(u_0,w_0)$ is radially symmetric with $u_0>0$ and $w_0>0$, and that $u_0\in H^2(B_R)$, $w_0-b\in H_0^1(B_R)\cap H^2(B_R)$. Let $(U,W)$ be the unique steady state obtained in Lemma \ref{Thm-steady} with $\int_{B_R}U(x)dx=\int_{B_R}u_0(x)dx.$
Then there exists a constant $\delta_1>0$ such that if the initial datum satisfies
\[\|(u_0-U,w_0-W)\|_{H^2(B_R)}\leq\delta_1,\]
then  the system \eqref{KS}  admits a unique global radial solution  $(u,w)\in\C([0,+\infty);H^2(B_R))$, which satisfies $u(x,t)>0$, $w(x,t)>0$ on $B_R\times(0,+\infty)$, and
\begin{equation}\label{3.1}
\|(u-U,w-W)(\cdot,t)\|^2_{H^2(B_R)}\leq C\|(u_0-U,w_0-W)\|^2_{H^2(B_R)}e^{-\mu t},\end{equation}
where $C$ and $\mu$ are positive constants independent of $t$.
\end{theorem}

\begin{corollary}\label{corollary}By the Sobolev imbedding theorem, we further have the following  asymptotic convergence:
\[\|(u-U,w-W)(\cdot,t)\|^2_{\C(\overline{B_R})}\leq Ce^{-\mu t}.\]
\end{corollary}

Theorem \ref{th-stability} is a direct consequence of Theorem \ref{Thm-stability} and Corollary \ref{corollary}. To show Theorem \ref{Thm-stability}, we first note that since $w|_{\partial B_R(0)}=b>0$, if $w_0(x)>0$ for all $x\in B_R(0)$, then by the maximum principle, we have $w(x,t)>0$ for all $x\in B_R(0)$ and $t>0$. Thus, one can remove the logarithmic nonlinearity of the sensitive function of \eqref{KS} by setting $v:=\ln w$. Then the system \eqref{KS}  is transformed into
\begin{eqnarray}\label{trans}
\left\{
\begin{array}{lll}
u_{t}=\Delta u-\chi\nabla\cdot(u\nabla v),& x\in  B_R(0),\ t>0, \\[1mm]
v_{t}=\varepsilon\Delta v+\varepsilon|\nabla v|^2-u,& x\in  B_R(0),\ t>0, \\[1mm]
u(x,0)=u_0(x),\ v(x,0)=v_0(x):=\ln w_0(x), & x\in  B_R(0), \\[1mm]
\frac{\partial u}{\partial \nu}-\chi u\frac{\partial v}{\partial\nu}=0, \ v=\ln b, & x\in \partial B_R(0),\ t>0.
\end{array}
\right.
\end{eqnarray}
Let $(U,W)$ be the unique positive solution of \eqref{2} with $\Omega=B_R(0)$ and $m=\int_{B_R}u_0(x)dx$ (see Lemma \ref{Thm-steady}). Set $V=\ln W$. Then it is straightforward to check that $(U, V)$  satisfies
\begin{eqnarray}\label{3.2}
\left\{
\begin{array}{lll}
\Delta U-\chi\nabla\cdot(U\nabla V)=0,& x\in  B_R(0), \\[1mm]
\varepsilon\Delta V+\varepsilon|\nabla V|^2-U=0,& x\in  B_R(0), \\[1mm]
\frac{\partial U}{\partial \nu}-\chi U\frac{\partial V}{\partial\nu}=0, \ V=\ln b, & x\in \partial B_R(0),\\[1mm]
\int_{B_R}Udx=\int_{B_R}u_0(x)dx,
\end{array}
\right.
\end{eqnarray}
which implies that $(U, V)$ is a steady state of \eqref{trans}. Using the uniqueness of positive solutions of \eqref{2}, one can immediately show that $(U, V)$ is the unique solution of \eqref{3.2}.

We next investigate the stability of the steady state to \eqref{trans}. Applying the contraction mapping theorem, one can readily derive the local wellposedness of \eqref{trans}. The routine and tedious proof details are omitted for brevity.
\begin{proposition}[Local well-posedness]\label{local}
For any $\Xi>0$, assume that the initial datum $(u_0,v_0)$ satisfies
\begin{equation*}
0<u_0\in H^2(B_R),\ v_0-\ln b\in H_0^1(B_R)\cap H^2(B_R) \text{ and } \|u_0-U\|_{H^2}+\|v_0-V\|_{H^2}\leq\Xi,
\end{equation*}where $(U,V)$ is the unique  solution of \eqref{3.2}.
Then, there exists a constant $T>0$ only depending on $\Xi$  such that the system  \eqref{trans} has a unique local solution $u\in C([0,T];H^2(B_R))$,  $v-\ln b\in H_0^1(B_R)\cap H^2(B_R)$, $(u,v)\in L^2((0,T);H^3(B_R))$, and for $t\in[0,T]$ it holds
\[\sup_{\tau\in[0,t]}\|(u-U,v-V)(\cdot,\tau)\|_{H^2}^2+\int_0^t\|(u-U,v-V)(\cdot,\tau)\|_{H^3}^2d\tau\leq 4\|(u_0-U,v_0-V)\|_{H^2}^2.\]

\end{proposition}

Noting that the system \eqref{trans} is rotationally invariant, by the uniqueness of solutions, one can see that if the initial data $(u_0,v_0)(x)$ is radially symmetric, then the solution $(u,v)(x,t)$ is also radially symmetric. In the radial setting, we have the following stability result for the reformulated system \eqref{trans}.
\begin{theorem}\label{v-stability} Let  $n=2,3$.
Assume that the initial datum $(u_0,v_0)$ is radially symmetric, and  that $0<u_0\in H^2(B_R)$, $v_0-\ln b\in H_0^1(B_R)\cap H^2(B_R)$. Let $(U,V)$ be the unique positive solution of \eqref{3.2}.
Then there exists a constant $\delta_2>0$ such that if the initial datum satisfies
$\|(u_0-U,v_0-V)\|_{H^2}\leq\delta_2,$
then  the system \eqref{trans} admits a unique global radial solution $(u,v)\in\C([0,+\infty);H^2(B_R))$, which satisfies  $u(x,t)>0$ on $B_R\times(0,+\infty)$, and
\begin{equation}\label{v-convergence}
\|(u-U,v-V)(\cdot,t)\|_{H^2}^2\leq C\|(u_0-U,v_0-V)\|_{H^2}^2e^{-\mu t},\end{equation}
where $C$ and $\mu$ are positive constants independent of $t$.
\end{theorem}

To show the global existence result claimed in Theorem \ref{v-stability}, by the local well-posedness result and the standard continuation argument, it suffices to establish the corresponding \emph{a priori} estimates of $(u-U,v-V)$. In order to achieve the proofs of Theorems \ref{v-stability} and \ref{Thm-stability} at the end of this section, we introduce the notation
\begin{equation}\label{NT}
N(t):=\sup_{\tau\in[0,t]}\big\{\|(u-U)(\cdot,\tau)\|_{H^2}+\|(v-V)(\cdot,\tau)\|_{H^2}\big\}.
\end{equation}
By the Sobolev embedding theorem, for $n=2,3$, we have
\begin{equation}\label{sobolev}
\sup_{\tau\in[0,t]}\big\{\|(u-U)(\cdot,\tau)\|_{L^\infty}+\|(v-V)(\cdot,\tau)\|_{L^\infty}\big\}\leq CN(t).
\end{equation}
Then the following \emph{a priori} estimates, proven below in a sequence of iterated lemmas, can be established.

\begin{proposition}[\emph{A priori} estimate]\label{aprioriestimate} Let  $n=2,3$. Assume that $(u,v)(x,t)$ is a radially symmetric solution of the system \eqref{trans} obtained in Proposition \ref{local} on $[0,T]$ for some $T>0$. Then there exists a constant $h>0$ independent of $T$ such that if $N(T)\leq h$,  it holds that
\begin{equation}\label{eqn3.4}
e^{\mu t}\|(u-U,v-V)(\cdot,t)\|_{H^2}^2+\int_0^te^{\mu \tau}\|(u-U,v-V)(\cdot,\tau)\|_{H^3}^2d\tau\leq C\|(u_0-U,v_0-V)\|_{H^2}^2
\end{equation}
for any $t\in[0,T]$, where $C$ and $\mu$ are positive constants independent of $t$ and $h$.
\end{proposition}

We now write the system \eqref{trans} in the radial coordinates as
\begin{eqnarray}\label{uv}
\left\{
\begin{array}{lll}
u_{t}=\frac{1}{r^{n-1}}[(r^{n-1}u_r)_r-\chi(r^{n-1}uv_r)_r],& r\in (0,R),\ t>0, \\[1mm]
v_{t}=\frac{\varepsilon}{r^{n-1}}(r^{n-1}v_r)_r+\varepsilon|v_r|^2-u,& r\in (0,R),\ t>0, \\[1mm]
u(r,0)=u_0(r),\ v(r,0)=\ln w_0(r), & r\in (0,R), \\[1mm]
u_r-\chi uv_r=0, \ v=\ln b, & r=R,\ t>0, \\[1mm]
u_r=0,\ v_r=0, \ & r=0,\ t>0.
\end{array}
\right.
\end{eqnarray}
Using the boundary condition of $U$, one can see that the steady state $(U,V)$ of \eqref{uv} satisfies
\begin{eqnarray}\label{uv-steady}
\left\{
\begin{array}{lll}
U_r=\chi UV_r,& r\in (0,R),\\[1mm]
\frac{\varepsilon}{r^{n-1}}(r^{n-1}V_r)_r+\varepsilon|V_r|^2-U=0,& r\in (0,R), \\[1mm]
 U_r(0)=0=V_r(0), \ V(R)=\ln b.
\end{array}
\right.
\end{eqnarray}
Thanks to the conservation of mass, it holds that $\int_{B_R}u(x,t)dx=\int_{B_R}u_0(x)dx$. Since $$\int_{B_R}u_0(x)dx=\int_{B_R}U(x)dx,$$ we have
$$
\int_{B_R}u(x,t)dx=\int_{B_R}U(x)dx \text{ for all } t\in[0,T],
$$
that is
$$
\omega_n\int_0^R u(r,t) r^{n-1}dr=\omega_n\int_0^R U(r) r^{n-1}dr,
$$
where we have used the fact $\int_{B_R} f(|x|,t)dx=\omega_n \int_0^R f(r,t)r^{n-1}dr$ with $\omega_n$ denoting the surface area of the unit sphere in $\R^n$.
This inspires us to take the relative difference between the radial mass distribution function of any solution $u(r,t)$ to the unique steady state $U(r)$ or \lq\lq radius-dependent anti-derivative\rq\rq to reformulate the problem. Precisely, if we set
\begin{equation}\label{anti}
\phi(r,t):=\frac{1}{r^{n-1}}\int_0^r(u(s,t)-U(s))s^{n-1}ds,\ \ \psi(r,t):= v(r,t)-V(r).
\end{equation}
Then by L'H\^{o}pital's rule, we have
\[\phi(0,t)=\underset{r\rightarrow0}{\lim}\frac{1}{r^{n-1}}\int_0^r(u(s,t)-U(s))s^{n-1}ds
=\underset{r\rightarrow0}{\lim}\frac{(u(r,t)-U(r))\cdot r}{n-1}=0,\]
and
\[\phi(R,t)=\int_0^R(u(s,t)-U(s))s^{n-1}ds=0.\]
Substituting \eqref{anti} into \eqref{uv}, one can find that $(\phi,\psi)$ satisfies
\begin{equation}\label{perturbation}
\begin{cases}
\phi_{t}=\left(\frac{(r^{n-1}\phi)_r}{r^{n-1}}\right)_r
-\chi V_r\frac{(r^{n-1}\phi)_r}{r^{n-1}}-\chi U\psi_{r}-\chi\frac{(r^{n-1}\phi)_r}{r^{n-1}}\psi_{r},& r\in(0,R),t>0,\\[1mm]
\psi_{t}=\frac{\varepsilon}{r^{n-1}}(r^{n-1}\psi_r)_r+2\varepsilon V_r\psi_{r}+\varepsilon\psi_{r}^{2}-\frac{(r^{n-1}\phi)_r}{r^{n-1}},& r\in(0,R),t>0,\\[1mm]
(\phi,\psi_r)(0,t)=(0,0), \ (\phi,\psi)(R,t)=(0,0), & t>0,
\end{cases}
\end{equation} with initial data
\begin{equation}\label{perturbation-initial}
(\phi_0,\psi_0)(r):=\left(\frac{1}{r^{n-1}}\int_0^r(u_0(s)-U(s))s^{n-1}ds,\ln w_0(r)-V(r)\right).
\end{equation}

We next establish the \emph{a priori} estimate \eqref{eqn3.4}. We begin with the basic $L^2$ estimate. In what follows, we shall abbreviate $\int_0^Rf(r)dr$ as $\int_0^R f(r)$, $\int_0^t \int_0^Rf(r,s)drds$ as $\int_0^t \int_0^Rf(r,s)$, and $B_R(0)$ as $B_R$ for the sake of notational simplicity.

\begin{lemma}\label{L2} Let the assumptions of Proposition \ref{aprioriestimate} hold.
If $N(T)\ll1$, then there exist two constants $C>0$ and $\mu>0$ independent of $t$ such that
\begin{equation}\label{eqn3.10}
\begin{split}
&e^{\mu t}(\|\phi(\cdot,t)\|_{L^2}^2+\|\psi(\cdot,t)\|_{L^2}^2)
+\int_0^te^{\mu \tau}(\|u(\cdot,\tau)-U(\cdot)\|_{L^2}^2
+\|\psi(\cdot,\tau)\|_{H^1}^2)d\tau\\& \leq C(\|\phi_0\|_{L^2}^2+\|\psi_0\|_{L^2}^2)
\end{split}\end{equation}
holds for any $t\in[0,T]$.
\end{lemma}

\begin{proof}
Integrating the sum of the first equation of $\eqref{perturbation}$ multiplied by $\frac{r^{n-1}\phi}{U}$ and the second one multiplied by $\chi r^{n-1}\psi$, we have alongside the integration by parts
\begin{equation}\label{3.4}
\begin{split}&\frac{1}{2}\frac{d}{dt}\int_0^R \left(\frac{r^{n-1}\phi^{2}}{U}+\chi r^{n-1}\psi^{2}\right)+\int_0^R \frac{|(r^{n-1}\phi)_r|^{2}}{r^{n-1}U}+\chi\varepsilon\int_0^R r^{n-1}\psi_{r}^{2}+\chi\varepsilon\int_0^R (r^{n-1}V_r)_r\psi^2\\
&\quad+\int_0^R(r^{n-1}\phi)_r\phi \left[\left(\frac{1}{U}\right)_r+\frac{\chi V_r}{U}\right]=-\chi\int_0^R\frac{\phi (r^{n-1}\phi)_r\psi_r}{U}+\chi\varepsilon\int_0^Rr^{n-1}\psi\psi_r^2.
\end{split}\end{equation}
A simple calculation from the first equation of \eqref{uv-steady} yields
\[\left(\frac{1}{U}\right)_r+\frac{\chi V_r}{U}=0.\]
From Lemma \ref{lemma2.6}, we get $(r^{n-1}V_r)_r>0.$
The last two terms of \eqref{3.4} can be estimated as follows.
By H\"{o}lder's inequality, we can derive that
\[\begin{split}
\left|\int_0^R\frac{\phi (r^{n-1}\phi)_r\psi_r}{U}\right|&\leq\|\phi\|_{L^\infty}\int_0^R
\frac{|(r^{n-1}\phi)_r|^{2}}{r^{n-1}U}+\|\phi\|_{L^\infty}\int_0^R\frac{r^{n-1}\psi_{r}^{2}}{U}\\&
\leq CN(t)\int_0^R
\left(\frac{|(r^{n-1}\phi)_r|^{2}}{r^{n-1}U}+\frac{r^{n-1}\psi_{r}^{2}}{U}\right),\end{split}\]
and
\[\begin{split}
\left|\int_0^Rr^{n-1}\psi\psi_r^2\right|\leq\|\psi\|_{L^\infty}\int_0^Rr^{n-1}\psi_{r}^{2}
\leq CN(t)\int_0^Rr^{n-1}\psi_{r}^{2},\end{split}\]
where we have used the fact based on \eqref{sobolev}
$$\|\phi(\cdot,t)\|_{L^\infty}\leq C\|u(\cdot,t)-U(\cdot)\|_{L^\infty}\leq CN(t),$$
and
$$\|\psi(\cdot,t)\|_{L^\infty}\leq CN(t).$$
Since $U$ is continuous and positive in $\overline{B_R}$, then $\max\limits_{x\in \overline{B_R}} U(x)=U_{\mathrm{max}}\geq U\geq U_{\mathrm{min}}=\min\limits_{x\in \overline{B_R}} U(x)>0$ and hence $\frac{r^{n-1}\psi_{r}^{2}}{U}\leq\frac{r^{n-1}\psi_{r}^{2}}{U_{\mathrm{min}}}$, $ \frac{|(r^{n-1}\phi)_r|^{2}}{r^{n-1}U}\geq  \frac{|(r^{n-1}\phi)_r|^{2}}{r^{n-1}U_{\mathrm{max}}}$.
Then substituting the above estimates into \eqref{3.4}, we obtain
\begin{equation}\label{3.9}
\begin{split}\frac{d}{dt}\int_0^R \left(\frac{r^{n-1}\phi^{2}}{U}+\chi r^{n-1}\psi^{2}\right)&+\frac{2}{U_{\mathrm{max}}}(1-pCN(t))\int_0^R \frac{|(r^{n-1}\phi)_r|^{2}}{r^{n-1}}\\
&+2\left[\chi\varepsilon-(p\e +p/U_{\mathrm{min}})CN(t)\right]\int_0^R r^{n-1}\psi_{r}^{2}\leq0.
\end{split}\end{equation}
Recalling the transformation \eqref{anti}, we get
\begin{equation}\label{3.12}
(r^{n-1}\phi)_r=(u-U)r^{n-1},
\end{equation}
from which we further have
\[\int_0^R \frac{|(r^{n-1}\phi)_r|^{2}}{r^{n-1}}=\int_0^Rr^{n-1}|u-U|^2dr
=\frac{1}{\omega_n}\int_{B_R}|u-U|^2dx.\]
Noticing that $\int_0^R r^{n-1}\psi_r^2dr=\frac{1}{\omega_n}\int_{B_R}|\nabla \psi|^2dx$, we have from \eqref{3.9} that
\begin{equation}\label{3.13}
\begin{split}\frac{d}{dt}\int_0^R \omega_n\left(\frac{r^{n-1}\phi^{2}}{U}+\chi r^{n-1}\psi^{2}\right)+\chi\varepsilon\int_{B_R}|\nabla\psi|^2dx+\frac{1}{U_{\mathrm{max}}}\int_{B_R}|u-U|^2dx\leq0,
\end{split}
\end{equation}
where we have used the assumption that $N(t)$ is sufficiently small such that, for instance, $1-pCN(t)>{1}/{2}$ and $\left[\chi\varepsilon-(p\e+p/U_{\mathrm{min}})CN(t)\right]>\chi\varepsilon/2$.
By  \eqref{anti}, we also have
\[\begin{split}
\int_{B_R}\phi^{2}dx=\omega_n\int_0^Rr^{n-1}\phi^{2}&=\omega_n\int_0^R\frac{1}{r^{n-1}}
\left(\int_0^r(u-U)s^{n-1}ds\right)^2\\&
\leq \omega_n\int_0^R(u-U)^2r^{n-1}dr\int_0^R\frac{1}{r^{n-1}}\int_0^rs^{n-1}dsdr\\
&=\frac{\omega_nR^2}{2n}\int_{B_R}(u-U)^2dx.
\end{split}\]
Noting that $\psi|_{\partial B_R}=0$ from the equations in \eqref{perturbation}, it follows from Poincar\'{e}'s inequality that
\[\omega_n\int_0^Rr^{n-1}\psi^{2}=\int_{B_R}\psi^2dx\leq C\int_{B_R}|\nabla\psi|^2dx.\]
Now, multiplying \eqref{3.13} by $e^{\mu t}$, where $\mu>0$ is a constant to be determined, we get
\[\begin{split}
&e^{\mu t}\left(\int_{B_R}\frac{\phi^{2}}{U}+\chi \int_{B_R}\psi^{2}\right)+\int_0^te^{\mu \tau}\left(\frac{1}{U_{\mathrm{max}}}\int_{B_R}(u-U)^2+\chi \varepsilon\int_{B_R}|\nabla\psi|^2\right)\\
&\leq \mu\int_0^te^{\mu \tau}\left(\int_{B_R}\frac{\phi^{2}}{U}+\chi \int_{B_R}\psi^{2}\right)+\int_{B_R}\frac{\phi_0^{2}}{U}+\chi \int_{B_R}\psi_0^{2}\\
&\leq \mu\int_0^te^{\mu \tau}\left(\frac{1}{U_{\mathrm{min}}}\int_{B_R}\phi^{2}+\chi \int_{B_R}\psi^{2}\right)+\frac{1}{U_{\mathrm{min}}}\int_{B_R}\phi_0^{2}+\chi \int_{B_R}\psi_0^{2}
\\&\leq C\mu\int_0^te^{\mu \tau}\left(\int_{B_R}(u-U)^2+\chi \int_{B_R}|\nabla\psi|^2\right)+C\int_{B_R}(\phi_0^{2}+\psi_0^{2}).\end{split}\]
Thus, the desired estimate \eqref{eqn3.10} holds by choosing $\mu>0$ to be suitably small.
\end{proof}

We next establish the $L^2$ estimate for $u-U$.
\begin{lemma}\label{L2-u} Let the assumptions of Proposition \ref{aprioriestimate} hold.
If $N(T)\ll1$, then it holds that
\begin{equation}\label{3.16}
\begin{split}
&e^{\mu t}\|u(\cdot,t)-U(\cdot)\|_{L^2}^2
+\int_0^te^{\mu \tau}\|\nabla(u(\cdot,\tau)-U(\cdot))\|_{L^2}^2\\&\leq C(\|u_0-U\|_{L^2}^2+\|\phi_0\|_{L^2}^2+\|\psi_0\|_{L^2}^2),\end{split}\end{equation}
for any $t\in[0,T]$, where $C>0$ is a constant  independent of $t$.

\end{lemma}

\begin{proof}
Using \eqref{3.12}, we write the first equation of \eqref{perturbation} as
\begin{equation}\label{phi}
\begin{cases}
\phi_{t}=(u-U)_r
-\chi V_r(u-U)-\chi U\psi_{r}-\chi(u-U)\psi_{r},& r\in(0,R),t>0,\\
\phi(0,t)=\phi(R,t)=0, & t>0,
\end{cases}
\end{equation}
with the following identity
\begin{eqnarray*}
\begin{aligned}
\phi_tr^{n-1}(u-U)_r &=(r^{n-1}\phi_t(u-U))_r-(u-U)(r^{n-1}\phi)_{tr}\\
&=(r^{n-1}(u-U)\phi_t)_r-\frac{1}{2}(r^{n-1}(u-U)^2)_t.
\end{aligned}
\end{eqnarray*}
Then multiplying $\eqref{phi}$ by $r^{n-1}(u-U)_r$ and integrating the result by parts along with Young's inequality,
we have
\begin{equation}\label{3.6}
\begin{split}&\frac{1}{2}\frac{d}{dt}\int_0^R r^{n-1}(u-U)^2+\int_0^Rr^{n-1} \left|(u-U)_r\right|^{2}\\
&=\chi\int_0^R r^{n-1}V_r(u-U)(u-U)_r
+\chi\int_0^R r^{n-1}U\psi_r(u-U)_r\\&\quad+\chi\int_0^Rr^{n-1}\psi_r(u-U)(u-U)_r
\\&\leq\frac{1}{2}\int_0^Rr^{n-1} \left|(u-U)_r\right|^{2}+C\int_0^R r^{n-1}[V_r^2(u-U)^2+U^2 \psi_r^2]\\&
\quad+C\int_0^R r^{n-1}(u-U)^2|\psi_r|^2.
\end{split}\end{equation}
Using \eqref{sobolev}, we have
\begin{equation}\label{eqn3.22}
\|u(\cdot,t)-U(\cdot)\|_{L^\infty}\leq CN(t).
\end{equation}
and hence get
\begin{equation}\label{eqn3.21}
\int_0^Rr^{n-1}(u-U)^2|\psi_r|^2\leq \|u-U\|^2_{L^\infty}\int_0^Rr^{n-1}|\psi_r|^2 \leq CN(t)^2\int_0^Rr^{n-1}|\psi_r|^2.
\end{equation}
Substituting \eqref{eqn3.21} into \eqref{3.6}, since $\|\nabla V\|_{L^\infty}=\|V_r\|_{L^\infty}\leq C$ and $\|U\|_{L^\infty}\leq C$, we  have
\begin{equation}\label{eqn3.15}
\begin{split}\frac{d}{dt}\|u(\cdot,t)-U(\cdot)\|_{L^2}^2
+\|\nabla(u(\cdot,t)-U(\cdot))\|_{L^2}^2\leq C\|u(\cdot,t)-U(\cdot)\|_{L^2}^2
+C\|\psi(\cdot,t)\|_{H^1}^2.
\end{split}\end{equation}
Now, multiplying \eqref{eqn3.15} by $e^{\mu t}$ and integrating the resulting inequality in $t$, from Lemma \ref{L2}, we obtain the desired estimate \eqref{3.16}.
\end{proof}

We proceed to derive the $H^1$ estimate for $\psi$.
\begin{lemma}\label{H1-psi}Let the assumptions of Proposition \ref{aprioriestimate} hold.
If $N(T)\ll1$, then it holds
\begin{equation}\label{3.8}
e^{\mu t}\|\nabla\psi(\cdot,t)\|_{L^2}^2+\int_0^te^{\mu \tau}\|\psi(\cdot,\tau)\|_{H^2}^2\leq C(\|\phi_0\|_{L^2}^2+\|\psi_0\|_{H^1}^2),\end{equation} for any $t\in[0,T]$, where $C>0$ is a constant  independent of $t$.

\end{lemma}

\begin{proof}  We write the second equation of \eqref{perturbation}, using the Cartesian coordinates,  as
\begin{equation}\label{3.10}
\begin{cases}
\psi_{t}=\varepsilon\Delta\psi+2 \varepsilon\nabla V\nabla\psi+\varepsilon|\nabla\psi|^{2}-(u-U),& x\in B_R,\ t>0\\
\psi=0,& x\in \partial B_R,\ t>0.
\end{cases}
\end{equation}
Multiplying $\eqref{3.10}$ by $\Delta\psi$ and integrating the resultant equation by parts alongside Young's inequality, we have
\begin{equation}\label{3.5}
\begin{split}
&\frac{1}{2}\frac{d}{dt}\int_{B_R}|\nabla\psi|^{2}+\varepsilon\int_{B_R} |\Delta\psi|^2\\&
=-2\varepsilon\int_{B_R} \nabla V\nabla\psi\Delta\psi-\varepsilon\int_{B_R}|\nabla\psi|^2\Delta\psi
+\int_{B_R}(u-U)\Delta\psi\\
&\leq\frac{3\varepsilon}{4}\int_{B_R}|\Delta\psi|^2+4\varepsilon\int_{B_R}|\nabla V|^2|\nabla\psi|^2
+\frac{1}{\varepsilon}\int_{B_R}|u-U|^2+\varepsilon\int_{B_R}|\nabla\psi|^4.
\end{split}
\end{equation} By the Sobolev imbedding theorem for $n=2,3$, we get
\begin{equation}\label{3.7}
\int_{B_R}|\nabla\psi|^4\leq C\|\psi\|_{H^2}^4\leq CN(t)^2\|\psi\|_{H^2}^2,\end{equation}
where we have used the fact $\|\psi(\cdot,t)\|_{H^2}\leq N(t)$ from \eqref{NT}. Employing the standard regularity theory for the Poisson equation:
\[\begin{cases}
\Delta\psi=f,& x\in B_R, \\
\psi=0,& x\in \partial B_R,
\end{cases}
\]
we have
\begin{equation}\label{ell-regular}
\|\psi\|_{H^2(B_R)}\leq C_1\|f\|_{L^2(B_R)}=C_1\|\Delta\psi\|_{L^2(B_R)}.
\end{equation}
It then follows from \eqref{3.5}-\eqref{ell-regular} that
\begin{equation*}
\begin{split}
\frac{1}{2}\frac{d}{dt}\int_{B_R}|\nabla\psi|^2+\left(\frac{\varepsilon}{4C_1^2}-CN(t)^2\right)\|\psi(\cdot,t)\|_{H^2}^2\leq C\varepsilon\int_{B_R}|\nabla\psi|^2
+\frac{1}{\varepsilon}\int_{B_R}|u-U|^2.
\end{split}
\end{equation*}
Thus, if $N(t)\ll1$ such that $\frac{\varepsilon}{4C_1^2}-CN(t)^2>\frac{\varepsilon}{8C_1^2}$, we have
\begin{equation*}
\begin{split}
\frac{d}{dt}\int_{B_R}|\nabla\psi|^2+\frac{\varepsilon}{4C_1^2}\|\psi(\cdot,t)\|_{H^2}^2\leq C\int_{B_R}|\nabla\psi|^2
+C\int_{B_R}|u-U|^2.
\end{split}
\end{equation*}
Now, multiplying this inequality by $e^{\mu t}$, integrating the resulting equation in $t$, and using Lemma \ref{L2}, we obtain \eqref{3.8}.
\end{proof}

We also need the $H^1$ estimate for $(u-U)$.
\begin{lemma}\label{H1-u}Let the assumptions of Proposition \ref{aprioriestimate} hold.
Assume that $N(T)\ll1$. Then it holds that
\[\begin{split}
e^{\mu t}\|\nabla(u(\cdot,t)-U(\cdot))\|_{L^2}^2+
\int_0^te^{\mu \tau}\|(u-U)_\tau(\cdot,\tau)\|_{L^2}^2\leq C(\|\phi_0\|_{L^2}^2+
\|u_0-U\|_{H^1}^2+\|\psi_0\|_{H^1}^2),\end{split}\]
for any $t\in[0,T]$, where $C>0$ is a constant  independent of $t$.
\end{lemma}

\begin{proof} First the equation \eqref{phi} gives rise to
\begin{equation}\label{eqn3.20}
r^{n-1}\phi_{t}^2\leq Cr^{n-1}(|(u-U)_r|^2+V_r^2(u-U)^2+U^2\psi_{r}^2+(u-U)^2\psi_{r}^2).\end{equation}
Integrating this inequality multiplied by $e^{\mu t}$ and using the boundedness of $\nabla V$ and $U$, we obtain
\begin{equation}\label{3.20}
\begin{split}
\int_0^te^{\mu \tau}\|\phi_{t}\|^2_{L^2}
&\leq C\int_0^te^{\mu \tau}
(\|\nabla(u-U)\|^2_{L^2}+\|u-U\|^2_{L^2}+\|\nabla\psi\|^2_{L^2}+\|u-U\|_{L^\infty}^2\|\nabla\psi\|^2_{L^2})\\
&\leq C(\|u_0-U\|_{L^2}^2+\|\phi_0\|_{L^2}^2+\|\psi_0\|_{L^2}^2),
\end{split}
\end{equation}
where in the second inequality we have used \eqref{eqn3.22} and Lemmas \ref{L2}-\ref{L2-u}.
Similarly, using the equation \eqref{3.10}, the inequality \eqref{3.7}, Lemma \ref{L2} and Lemma \ref{H1-psi}, we get
\begin{equation}\label{3.21}
\begin{split}
\int_0^te^{\mu \tau}\|\psi_{t}\|_{L^2}^2&\leq C\int_0^te^{\mu \tau}
(\|\psi\|_{H^2}^2+\|\nabla\psi\|_{L^4}^4+\|u-U\|^2_{L^2})\leq C( \|\phi_0\|_{L^2}^2+\|\psi_0\|_{H^1}^2).
\end{split}
\end{equation}

We next estimate $\|\phi_t(\cdot,t)\|_{L^2}$. Differentiating the equation \eqref{phi} with respect to $t$ gives
\begin{equation}\label{3.22}
\begin{cases}
\phi_{tt}=(u-U)_{tr}
-\chi V_r(u-U)_t-\chi(u-U)_t\psi_{r}-\chi [U+(u-U)]\psi_{tr},& r\in(0,R),\\
\phi_t(0,t)=\phi_t(R,t)=0.
\end{cases}
\end{equation}
Note that
\[\begin{split}(u-U)_{tr}r^{n-1}\phi_t&=((u-U)_tr^{n-1}\phi_t)_r-(u-U)_t(r^{n-1}\phi)_{rt}\\&
=((u-U)_tr^{n-1}\phi_t)_r-r^{n-1}|(u-U)_t|^2,\end{split}\]
and
\[\begin{split}-\chi[U+(u-U)]\psi_{tr}r^{n-1}\phi_t=&-\chi([U+(u-U)]\psi_tr^{n-1}\phi_t)_r
+\chi[U+(u-U)]_r\psi_tr^{n-1}\phi_t\\&+\chi[U+(u-U)]\psi_tr^{n-1}(u-U)_t,\end{split}\]
where \eqref{3.12} has been used.
Multiplying the equation \eqref{3.22} by $r^{n-1}\phi_t$ and integrating  by parts along with the boundary conditions in \eqref{3.22},
we have
\begin{equation}\label{3.23}
\begin{split}&\frac{1}{2}\frac{d}{dt}\int_0^R r^{n-1}\phi_t^2+\int_0^Rr^{n-1} \left|(u-U)_t\right|^{2}\\
&=-\chi\int_0^R V_r\phi_tr^{n-1}(u-U)_t-\chi\int_0^R \psi_r\phi_tr^{n-1}(u-U)_t
\\&\quad+\chi\int_0^R [U+(u-U)]_r\psi_tr^{n-1}\phi_t+\chi\int_0^R[U+(u-U)]\psi_tr^{n-1}(u-U)_t
\\&\leq\frac{1}{2}\int_0^Rr^{n-1} \left|(u-U)_t\right|^{2}+C\int_0^R r^{n-1}[\phi_t^2+\psi_r^2\phi_t^2+\psi_t^2+|(u-U)_r|^2\phi_t^2+(u-U)^2\psi_t^2].
\end{split}\end{equation}
By H\"{o}lder's inequality and the Sobolev inequality for $n=2,3$ (cf. \cite[Theorem 7.10]{gilbarg1977elliptic}), we get
\begin{equation}\label{3.24}
\begin{split}\omega_n\int_0^R r^{n-1}\psi_r^2\phi_t^2
=\int_{B_R}|\nabla\psi|^2\phi_t^2&\leq\left(\int_{B_R}|\nabla\psi|^4\right)^{\frac{1}{2}}
\left(\int_{B_R}|\phi_t|^4\right)^{\frac{1}{2}}\\&
\leq C \|\psi\|_{H^2}^2\int_{B_R}|\nabla\phi_t|^2\\&\leq CN(t)^2\int_{B_R}|\nabla\phi_t|^2,
\end{split}\end{equation}
where we have used $\|\psi(\cdot,t)\|_{H^2}\leq N(t)$ again from \eqref{NT}. Similarly, noting $\|(u-U)(\cdot,t)\|_{H^2}\leq N(t)$, it holds
\[\omega_n\int_0^Rr^{n-1}|(u-U)_r|^2\phi_t^2=\int_{B_R}|\nabla(u-U)|^2\phi_t^2\leq C\|u-U\|_{H^2}^2\int_{B_R}|\nabla\phi_t|^2\leq CN(t)^2\int_{B_R}|\nabla\phi_t|^2.\]
By \eqref{eqn3.22}, we get
\[\int_0^R r^{n-1}(u-U)^2\psi_t^2\leq CN(t)^2\int_0^R r^{n-1}\psi_t^2.\]
Observe that the second term of \eqref{3.23} satisfies
\begin{equation}\label{3.26}
\begin{split}
\int_0^Rr^{n-1} \left|(u-U)_t\right|^{2}&=\int_0^Rr^{n-1}\left|\frac{(r^{n-1}\phi_t)_r}{r^{n-1}}\right|^{2}\\
&=\int_0^R\left(r^{n-1}\phi_{tr}^2+2(n-1)\phi_t\phi_{tr}r^{n-2}+(n-1)^2r^{n-3}\phi_t^2\right)\\
&=\int_0^Rr^{n-1}\phi_{tr}^2+(n-1)\int_0^Rr^{n-3}\phi_t^2+(n-1)\phi_t^2(R,t)R^{n-2}\\
&\geq\frac{1}{\omega_n}\int_{B_R}|\nabla\phi_t|^2.
\end{split}\end{equation} It then follows from \eqref{3.23} that
\begin{equation*}
\begin{split}\frac{d}{dt}\int_0^R r^{n-1}\phi_t^2+\frac{1}{2}\int_0^Rr^{n-1} \left|(u-U)_t\right|^{2}+\left(\frac{1}{2\omega_n}-CN(t)^2\right)\int_{B_R}|\nabla\phi_t|^2\leq C\int_0^R r^{n-1}(\phi_t^2+\psi_t^2).
\end{split}\end{equation*}
If $N(t)\ll1$  such that $\frac{1}{2\omega_n}>CN(t)^2$, then
\begin{equation}\label{3.27}
\begin{split}\frac{d}{dt}\int_0^R r^{n-1}\phi_t^2+\frac{1}{2}\int_0^Rr^{n-1} \left|(u-U)_t\right|^{2}\leq C\int_0^R r^{n-1}(\phi_t^2+\psi_t^2).
\end{split}\end{equation}
Now multiplying \eqref{3.27} by $e^{\mu t}$ and integrating the resulting equation in $t$, by the estimates \eqref{3.20}-\eqref{3.21}, we have
\begin{equation}\label{3.29}
e^{\mu t}\int_{B_R}\phi_t^2+\int_0^te^{\mu \tau}\int_{B_R}\left|(u-U)_t\right|^{2}\leq C(\|\phi_0\|_{L^2}^2+
\|u_0-U\|_{H^1}^2+\|\psi_0\|_{H^1}^2),
\end{equation}
where we have used $\phi_t^2(\cdot,0)\leq C(|(u_0-U)_r|^2+(u_0-U)^2+\psi_{0r}^2+\psi_{0r}^2)$ owing to \eqref{eqn3.20}.
Using the equation \eqref{phi} again, by \eqref{3.29} and Lemmas \ref{L2-u}-\ref{H1-psi}, we have
\begin{equation}\label{3.30}
\begin{split}
e^{\mu t}\int_{B_R}|\nabla(u-U)|^2&\leq Ce^{\mu t}\int_{B_R}(\phi_t^2+(u-U)^2+|\nabla\psi|^2)\\&\leq C(\|\phi_0\|_{L^2}^2+
\|u_0-U\|_{H^1}^2+\|\psi_0\|_{H^1}^2).
\end{split}\end{equation}
The desired estimate follows from \eqref{3.29} and \eqref{3.30}.
\end{proof}

The $H^2$ estimate for $\psi$ is as follows.

\begin{lemma}\label{H2-psi}Let the assumptions of Proposition \ref{aprioriestimate} hold.
Assume that $N(T)\ll1$, then it holds
\begin{equation}\label{eqn3.38}
\begin{split}
e^{\mu t}\|\psi(\cdot,t)\|_{H^2}^2+
\int_0^te^{\mu \tau}\|\psi(\cdot,\tau)\|_{H^3}^2\leq C(\|\phi_0\|_{L^2}^2+
\|u_0-U\|_{H^1}^2+\|\psi_0\|_{H^2}^2),\end{split}\end{equation}
for any $t\in[0,T]$, where $C>0$ is a constant  independent of $t$.
\end{lemma}

\begin{proof}Differentiating the equation \eqref{3.10} in $t$ yields
\begin{equation}\label{3.32}
\begin{cases}
\psi_{tt}=\varepsilon\Delta\psi_t+2\varepsilon \nabla V\nabla\psi_t+2\varepsilon\nabla\psi\nabla\psi_t-(u-U)_t,& x\in B_R,\ t>0\\
\psi_t=0,& x\in \partial B_R,\ t>0.
\end{cases}
\end{equation}
Multiplying \eqref{3.32} by $\psi_t$ and integrating by parts, we get
\begin{equation}\label{3.33}
\begin{split}
&\frac{1}{2}\frac{d}{dt}\int_{B_R}\psi_t^2+\varepsilon\int_{B_R}|\nabla\psi_t|^2\\&
=2\varepsilon\int_{B_R} \nabla V\nabla\psi_t\psi_t+2\varepsilon\int_{B_R}\nabla\psi\nabla\psi_t\psi_t-\int_{B_R}(u-U)_t\psi_t \\
&\leq\frac{\varepsilon}{4}\int_{B_R}|\nabla\psi_t|^2+C\int_{B_R}(|\nabla V|^2|\psi_t|^2+|(u-U)_t|^2+|\psi_t|^2)
+C\int_{B_R}|\nabla\psi|^2|\psi_t|^2.
\end{split}
\end{equation}
As in \eqref{3.24}, by the H\"{o}lder's inequality and the Sobolev inequality, we get
\begin{equation*}%
\begin{split}
\int_{B_R}|\nabla\psi|^2|\psi_t|^2
\leq\left(\int_{B_R}|\nabla\psi|^4\right)^{\frac{1}{2}}
\left(\int_{B_R}|\psi_t|^4\right)^{\frac{1}{2}}
\leq C\|\psi\|_{H^2}^2\int_{B_R}|\nabla\psi_t|^2\leq CN(t)^2\int_{B_R}|\nabla\psi_t|^2.
\end{split}\end{equation*}
If $N(t)\ll1$, it then follows from \eqref{3.33} that
\begin{equation*}
\frac{d}{dt}\int_{B_R}\psi_t^2+\varepsilon\int_{B_R}|\nabla\psi_t|^2\leq C\int_{B_R}(|\psi_t|^2+|(u-U)_t|^2).
\end{equation*}
Multiplying this inequality by $e^{\mu t}$ and integrating in $t$, by \eqref{3.21} and Lemma \ref{H1-u}, we have
\begin{equation}\label{3.35}\begin{split}
e^{\mu t}\int_{B_R}\psi_t^2+
\varepsilon\int_0^te^{\mu \tau}\int_{B_R}|\nabla\psi_t|^2\leq C(\|\phi_0\|_{L^2}^2+
\|u_0-U\|_{H^1}^2+\|\psi_0\|_{H^2}^2),\end{split}\end{equation}
where we have used $\|\psi_t(\cdot,0)\|_{L^2}^2\leq C(\|\psi_0\|_{H^2}^2+\|u_0-U\|_{L^2}^2)$.
Write \eqref{3.10} as an elliptic equation,
\begin{equation}\label{3.47}
\begin{cases}
-\varepsilon\Delta\psi=-\psi_{t}+2 \varepsilon\nabla V\nabla\psi+\varepsilon|\nabla\psi|^{2}-(u-U),& x\in B_R,\\
\psi=0,& x\in \partial B_R.
\end{cases}
\end{equation}
Using \eqref{3.7} and \eqref{ell-regular}, we have
\begin{equation}\label{3.36}\begin{split}
e^{\mu t}\|\psi\|_{H^2}^2&\leq Ce^{\mu t}(\|\psi_t\|_{L^2}^2+\|\psi\|_{H^1}^2+\|u-U\|_{L^2}^2) \\&\leq C(\|\phi_0\|_{L^2}^2+
\|u_0-U\|_{H^1}^2+\|\psi_0\|_{H^2}^2),\end{split}\end{equation} where we have used \eqref{3.35} and Lemmas \ref{L2-u}-\ref{H1-psi}.
Applying the regularity theory of elliptic equation for \eqref{3.47}, (cf. \cite[Theorem 8.13]{gilbarg1977elliptic}), we get
\begin{equation}\label{new-3.49}
\|\psi\|_{H^3}^2\leq C(\|\psi_t\|_{H^1}^2+\|\psi\|_{H^2}^2+\||\nabla\psi|^2\|_{H^1}^2+\|u-U\|_{H^1}^2).\end{equation}
By \eqref{3.7} and the Sobolev imbedding theorem,
\begin{equation*}\begin{split}
\||\nabla\psi|^2\|_{H^1}^2&\leq C\int_{B_R}|\nabla\psi|^4+C\sum_{i,j=1}^n\int_{B_R}|\nabla\psi|^2|\partial_{x_ix_j}^2\psi|^2\\
&\leq CN(t)^2\|\psi\|_{H^2}^2+C\sum_{i,j=1}^n\left(\int_{B_R}|\nabla\psi|^4\right)^{\frac{1}{2}}
\left(\int_{B_R}|\partial_{x_ix_j}^2\psi|^4\right)^{\frac{1}{2}}\\&
\leq C\|\psi\|_{H^2}^2+C\|\psi\|_{H^2}^2\|\psi\|_{H^3}^2\\&\leq C\|\psi\|_{H^2}^2+CN(t)^2\|\psi\|_{H^3}^2.\end{split}\end{equation*}
Substituting this inequality into \eqref{new-3.49}, when $N(t)\ll1$, we get
\begin{equation}\label{3.49}
\|\psi\|_{H^3}^2\leq C(\|\psi_t\|_{H^1}^2+\|\psi\|_{H^2}^2+\|u-U\|_{H^1}^2).\end{equation}
Multiplying \eqref{3.49} by $e^{\mu t}$ and integrating in $t$,  by \eqref{3.21},
\eqref{3.35} and Lemmas \ref{L2}-\ref{H1-psi}, we obtain
\begin{equation}\label{3.44}
\begin{split}
\int_0^te^{\mu \tau}\|\psi(\cdot,\tau)\|_{H^3}^2\leq C(\|\phi_0\|_{L^2}^2+
\|u_0-U\|_{H^1}^2+\|\psi_0\|_{H^2}^2).\end{split}\end{equation}
The desired estimate \eqref{eqn3.38} follows from \eqref{3.36} and \eqref{3.44}.
\end{proof}

Finally, we establish the $H^2$ estimate for $(u-U)$.

\begin{lemma}\label{H2-u} Let the assumptions of Proposition \ref{aprioriestimate} hold.
Assume that $N(T)\ll1$, then it holds
\begin{equation}\label{3.45}
\begin{split}
e^{\mu t}\|(u(\cdot,t)-U(\cdot))\|_{H^2}^2+
\int_0^te^{\mu \tau}\|(u-U)(\cdot,\tau)\|_{H^3}^2\leq C(\|u_0-U\|_{H^2}^2+\|\psi_0\|_{H^2}^2),\end{split}\end{equation}
for any $t\in[0,T]$, where $C>0$ is a constant  independent of $t$.
\end{lemma}

\begin{proof} To save notation, we set $j:=\nabla u-\chi u\nabla v$ and $J:=\nabla U-\chi U\nabla V$, then $(u-U)$ satisfies
\begin{equation*}
\begin{cases}
(u-U)_{t}=\nabla\cdot(j-J),& x\in B_R,t>0,\\
(j-J)\cdot\nu=0,& x\in \partial B_R,t>0.
\end{cases}
\end{equation*}
Differentiating this equation in $t$ yields
\begin{equation}\label{3.37}
\begin{cases}
(u-U)_{tt}=\nabla\cdot(j-J)_t,& x\in B_R,t>0,\\
(j-J)_t\cdot\nu=0,& x\in \partial B_R,t>0.
\end{cases}
\end{equation}
Multiplying \eqref{3.37} by $(u-U)_t$ and integrating by parts, we have
\begin{equation}\label{3.46}
\begin{split}
\frac{1}{2}\frac{d}{dt}\int_{B_R}(u-U)_t^2=-\int_{B_R}(j-J)_t\nabla(u-U)_t.
\end{split}
\end{equation}
A direct calculation gives
\begin{equation*}
(j-J)_t=\nabla(u-U)_t-\chi (u-U)_t\nabla V-\chi (u-U)_t\nabla\psi-\chi (u-U)\nabla\psi_t-\chi U\nabla\psi_t.
\end{equation*}
It then follows from \eqref{3.46} that
\begin{equation}\label{3.38}
\begin{split}
&\frac{d}{dt}\int_{B_R}(u-U)_t^2+2\int_{B_R}|\nabla(u-U)_t|^2\\
&\leq C\int_{B_R}(|(u-U)_t|^2|\nabla V|^2
+|u-U|^2|\nabla\psi_t|^2+U^2|\nabla\psi_t|^2)\\&
\quad+C\int_{B_R}|\nabla\psi|^2|(u-U)_t|^2+\frac{1}{2}\int_{B_R}|\nabla(u-U)_t|^2.
\end{split}
\end{equation}
Noting $\int_{B_R}(u-U)_tdx=\frac{d}{dt}\int_{B_R}(u-U)dx=0$, when $n=2,3$, it follows from Poincar\'{e}'s inequality that
\[\|(u-U)_t\|_{L^4}\leq C\|\nabla(u-U)_t\|_{L^2},\]
which implies
\[\int_{B_R}|\nabla\psi|^2|(u-U)_t|^2\leq\left(\int_{B_R}|\nabla\psi|^4\right)^{\frac{1}{2}}
\left(\int_{B_R}|(u-U)_t|^4\right)^{\frac{1}{2}}\leq CN(t)^2\int_{B_R}|\nabla(u-U)_t|^2,\]
where we have used $\|\nabla\psi(\cdot,t)\|_{L^4}\leq C\|\psi(\cdot,t)\|_{H^2}\leq CN(t)$. Thus if $N(t)\ll1$,  using the boundedness of $\nabla V$ and $U$, by \eqref{sobolev}, we get from \eqref{3.38} that
\begin{equation*}
\begin{split}
\frac{d}{dt}\int_{B_R}(u-U)_t^2+\int_{B_R}|\nabla(u-U)_t|^2
\leq &C\int_{B_R}(|(u-U)_t|^2
+|\nabla\psi_t|^2).
\end{split}
\end{equation*}
Multiplying this inequality by $e^{\mu t}$ and integrating in $t$, by \eqref{3.35} and Lemma \ref{H1-u}, we have
\begin{equation}\label{3.40}\begin{split}
&e^{\mu t}\int_{B_R}(u-U)_t^2+
\int_0^te^{\mu \tau}\int_{B_R}|\nabla(u-U)_t|^2\\&\leq C(\|\phi_0\|_{L^2}^2+
\|u_0-U\|_{H^1}^2+\|\psi_0\|_{H^2}^2+\|(u-U)_t(\cdot,0)\|_{L^2}^2)\\&\leq C(
\|u_0-U\|_{H^2}^2+\|\psi_0\|_{H^2}^2),\end{split}\end{equation}
where we have used $$\|\phi_0\|_{L^2}^2\leq C\|u_0-U\|_{L^\infty}^2\leq C\|u_0-U\|_{H^2}^2,$$ and
$$\|(u-U)_t(\cdot,0)\|_{L^2}^2\leq C(\|u_0-U\|_{H^2}^2+\|\psi_0\|_{H^2}^2).$$

We next estimate $\|(u-U)(\cdot,t)\|_{H^2}$. Set $h:= ue^{-\chi v}$ and $H:= Ue^{-\chi V}$, then by the first equation of \eqref{trans}, $h$ satisfies
\begin{eqnarray*}
\left\{
\begin{array}{lll}
h_{t}=\Delta h-\chi hv_t+\chi\nabla h\nabla v,& x\in  B_R,\ t>0, \\[1mm]
\frac{\partial h}{\partial \nu}=0,& x\in \partial B_R,\ t>0,\\[1mm]
h(x,0)=u_0(x)e^{-\chi v_0(x)}, & x\in  B_R,
\end{array}
\right.
\end{eqnarray*}
and $(h-H)$ satisfies
\begin{eqnarray}\label{3.54}
\left\{
\begin{array}{lll}
-\Delta (h-H)=-h_{t}-\chi h\psi_t+\chi\nabla (h-H)\nabla V+\chi\nabla (h-H)\nabla\psi+\chi\nabla H\nabla\psi,& x\in  B_R,\\[1mm]
\frac{\partial (h-H)}{\partial \nu}=0,& x\in \partial B_R.
\end{array}
\right.
\end{eqnarray}
Thus,
\begin{equation}\label{h-H}
\int_{B_R}|\Delta(h-H)|^2\leq C\int_{B_R}(h_t^2+h^2\psi_t^2+(|\nabla V|^2+|\nabla \psi|^2)|\nabla (h-H)|^2+|\nabla H|^2|\nabla \psi|^2).
\end{equation}
A direct calculation gives
\begin{equation}\label{3.42}
h-H=(u-U)e^{-\chi(V+\psi)}+Ue^{-\chi V}(e^{-\chi\psi}-1),
\end{equation}
which implies
\[\|h\|_{L^\infty}\leq \|h-H\|_{L^\infty}+\|H\|_{L^\infty}\leq C,\ \|\nabla(h-H)\|_{L^2}^2\leq C(\|u-U\|_{H^1}^2+\|\psi\|_{H^1}^2),\]
and
\[\|h_t\|_{L^2}^2\leq C(\|(u-U)_t\|_{L^2}^2+\|\psi_t\|_{L^2}^2).\]
By H\"{o}lder's inequality and the Sobolev imbedding theorem for $n=2,3$, we have
\[\int_{B_R}|\nabla \psi|^2|\nabla (h-H)|^2\leq \|\nabla\psi\|_{L^4}^2\|\nabla(h-H)\|_{L^4}^2\leq C\|\psi\|_{H^2}^2\|h-H\|_{H^2}^2\leq CN(t)^2\|h-H\|_{H^2}^2.\]
It then follows from \eqref{h-H} that
\begin{equation}\label{3.48}
\|\Delta(h-H)\|_{L^2}^2\leq C(\|(u-U)_t\|_{L^2}^2+\|\psi_t\|_{L^2}^2+\|u-U\|_{H^1}^2+\|\psi\|_{H^1}^2)+CN(t)^2\|h-H\|_{H^2}^2.\end{equation}
As in \eqref{ell-regular}, by the standard regularity theory for the Poisson equation, we get from \eqref{3.48} and \eqref{3.42} that
\[\begin{split}
\|h-H\|_{H^2}^2&\leq C(\|\Delta(h-H)\|_{L^2}^2+\|h-H\|_{L^2}^2)\\&\leq C(\|(u-U)_t\|_{L^2}^2+\|\psi_t\|_{L^2}^2+\|u-U\|_{H^1}^2+\|\psi\|_{H^1}^2)\\&\quad+CN(t)^2\|h-H\|_{H^2}^2+C\|h-H\|_{L^2}^2\\
&\leq C(\|(u-U)_t\|_{L^2}^2+\|\psi_t\|_{L^2}^2+\|u-U\|_{H^1}^2+\|\psi\|_{H^1}^2+N(t)^2\|h-H\|_{H^2}^2).\end{split}\]
Thus, if $N(t)\ll1$, then we have
\begin{equation}\label{3.43}
\|h-H\|_{H^2}^2\leq C(\|(u-U)_t\|_{L^2}^2+\|\psi_t\|_{L^2}^2+\|u-U\|_{H^1}^2+\|\psi\|_{H^1}^2).\end{equation}
By \eqref{3.42}, we also get
\begin{equation}\label{3.59}
u-U=e^{\chi(V+\psi)}(h-H)-Ue^{\chi \psi}(e^{-\chi\psi}-1),
\end{equation}
which in combination with \eqref{3.43} leads to
\[\|u-U\|_{H^2}^2\leq C(\|h-H\|_{H^2}^2+\|\psi\|_{H^2}^2)\leq C(\|(u-U)_t\|_{L^2}^2+\|\psi_t\|_{L^2}^2+\|u-U\|_{H^1}^2+\|\psi\|_{H^2}^2).\]
It then follows from \eqref{3.40}, \eqref{3.35} and Lemmas \ref{L2-u} and  \ref{H1-u}-\ref{H2-psi} that
\begin{equation}\label{3.50}
e^{\mu t}\|u-U\|_{H^2}^2\leq C(\|u_0-U\|_{H^2}^2+\|\psi_0\|_{H^2}^2).\end{equation}
As in \eqref{3.49}, by the regularity theory of elliptic equation for \eqref{3.54}, we are led to
\begin{equation*}
\int_0^te^{\mu \tau}\|(h-H)(\cdot,\tau)\|_{H^3}^2\leq C\int_0^te^{\mu \tau}(\|h_t\|_{H^1}^2+\|\psi_t\|_{H^1}^2+\|h-H\|_{H^2}^2+\|\psi\|_{H^3}^2).\end{equation*}
Thanks to \eqref{3.59}, we further have
\begin{equation}\label{3.60}
\begin{split}
\int_0^te^{\mu \tau}\|(u-U)(\cdot,\tau)\|_{H^3}^2&\leq C\int_0^te^{\mu \tau}(\|(u-U)_t\|_{H^1}^2+\|\psi_t\|_{H^1}^2+\|u-U\|_{H^2}^2+\|\psi\|_{H^3}^2)\\
&\leq C(\|u_0-U\|_{H^2}^2+\|\psi_0\|_{H^2}^2),\end{split}\end{equation}
where we have used \eqref{3.21}, \eqref{3.35}, \eqref{3.40} and Lemmas \ref{L2}-\ref{H2-psi}. The desired estimate \eqref{3.45} follows from \eqref{3.50} and \eqref{3.60}.
\end{proof}


\begin{proof}[Proof of Theorem \ref{v-stability}.]
First note that Proposition \ref{aprioriestimate} is a direct consequence of Lemmas \ref{H2-psi}-\ref{H2-u}. The \emph{a priori} estimate \eqref{eqn3.4} guarantees that if $\|(u_0-U,v_0-V)\|_{H^2}$ is small enough, then $N(t)$ is small for all $t>0$. Therefore, applying the standard extension argument, we obtain the global well-posedness of the system \eqref{trans} in $\C([0,+\infty);H^2(B_R))$. Moreover, the estimate \eqref{eqn3.4} also implies the exponential stability of the steady state $(U,V)$ in $H^2(B_R)$. By the Sobolev imbedding theorem, we further have
\[\|(u-U)(\cdot,t)\|_{L^\infty}\leq C\|(u-U)(\cdot,t)\|_{H^2}\leq C\|u_0-U\|_{H^2}.\] Thus,
\[u(x,t)\geq U(x)-C\|u_0-U\|_{H^2}\geq U_{\min}-C\|u_0-U\|_{H^2}>0,\] provided that $\|u_0-U\|_{H^2}<\delta_2:=\frac{U_{\min}}{C}$. This completes the proof.
\end{proof}
\medskip

\begin{proof}[Proof of Theorem \ref{Thm-stability}.] We complete the proof in two steps.

\emph{Step 1.} We first show the uniqueness of solutions for the system \eqref{KS} in the space $Y_T$ defined by
\[Y_T:=\{(u,w)(x,t)\ | \ u(x,t)>0, w(x,t)>0, (u,w)\in \C([0,T];H^2(B_R))\}.\]
Without loss of generality, we only consider the $n=3$ case. Suppose that in $Y_T$ the system \eqref{KS} has two solutions $(u_1,w_1)$ and $(u_2,w_2)$  with the same initial data $(u_0,w_0)(x)$. Set $\Phi:=u_1-u_2$ and $\Psi:=w_1-w_2$. Then a direct calculation yields the equations of $(\Phi,\Psi)$
\begin{eqnarray}\label{unique}
\left\{
\begin{array}{lll}
\Phi_{t}=\Delta\Phi-p\nabla\cdot\left(\frac{u_2}{w_1}\nabla \Psi-\frac{u_2\nabla w_2}{w_1w_2}\Psi+\frac{\nabla w_1}{w_1}\Phi\right),& x\in  B_R,\ t>0, \\[1mm]
\Psi_{t}=\varepsilon\Delta \Psi-w_1\Phi-u_2\Psi,& x\in  B_R,\ t>0, \\[1mm]
\Phi(x,0)=\Psi(x,0)=0, & x\in  B_R, \\[1mm]
\frac{\partial \Phi}{\partial \nu}-p (\frac{u_2}{w_1}\frac{\partial \Psi}{\partial\nu}-\frac{u_2\Psi}{w_1w_2}\frac{\partial w_2}{\partial\nu}+\frac{\Phi}{w_1}\frac{\partial w_1}{\partial\nu})=0, \ \Psi=0, & x\in \partial B_R,\ t>0.
\end{array}
\right.
\end{eqnarray}
Multiplying the first equation of \eqref{unique} by $\Phi$ and integrating the equation on $B_R$, we get
\begin{equation}\label{F1}
\begin{split}
&\frac{1}{2}\frac{d}{dt}\int_{B_R}|\Phi|^{2}+\int_{B_R} |\nabla\Phi|^2\\&
=p\int_{B_R} \frac{u_2}{w_1}\nabla \Psi\nabla\Phi-p\int_{B_R}\frac{u_2\nabla w_2}{w_1w_2}\Psi\nabla\Phi+p
\int_{B_R}\frac{\nabla w_1}{w_1}\Phi\nabla\Phi\\
&\leq\frac{1}{8}\int_{B_R} |\nabla\Phi|^2+Cp^2\int_{B_R} \frac{u_2^2}{w_1^2}|\nabla \Psi|^2+Cp^2\int_{B_R}\frac{u_2^2|\nabla w_2|^2}{w_1^2w_2^2}\Psi^2+Cp^2\int_{B_R}\frac{|\nabla w_1|^2}{w_1^2}\Phi^2.
\end{split}
\end{equation} For convenience, set $Q_T:=B_R\times(0,T)$, $A_i(T):=\underset{(x,t)\in Q_T}{\min}w_i^2$  and $C_i(T):=\|(u_i,w_i)\|_{C([0,T];H^2)}^2$ for $i=1,2$.
By the H\"{o}lder's inequality and the Sobolev inequality, we get
\[\int_{B_R} \frac{u_2^2}{w_1^2}|\nabla \Psi|^2\leq\frac{\|u_2\|_{L^\infty(Q_T)}^2}{\underset{(x,t)\in Q_T}{\min}w_1^2}\int_{B_R} |\nabla \Psi|^2\leq\frac{CC_2(T)}{A_1(T)}\int_{B_R} |\nabla \Psi|^2.\]
Similarly,
\[\begin{split}
\int_{B_R}\frac{u_2^2|\nabla w_2|^2}{w_1^2w_2^2}\Psi^2&\leq\frac{CC_2(T)}{A_1(T)A_2(T)}\left(\int_{B_R}|\nabla w_2|^4\right)^{\frac{1}{2}}
\left(\int_{B_R}|\Psi|^4\right)^{\frac{1}{2}}\\
&\leq\frac{CC_2(T)}{A_1(T)A_2(T)}\|w_2\|_{H^2}^2\int_{B_R} |\nabla\Psi|^2\\
&\leq\frac{CC_2^2(T)}{A_1(T)A_2(T)}\int_{B_R} |\nabla\Psi|^2.\end{split}\]
By the interpolation inequality, one has
\[\begin{split}
Cp^2\int_{B_R}\frac{|\nabla w_1|^2}{w_1^2}\Phi^2&\leq\frac{Cp^2}{A_1(T)}\left(\int_{B_R}|\nabla w_1|^4\right)^{\frac{1}{2}}
\left(\int_{B_R}|\Phi|^4\right)^{\frac{1}{2}}\\
&\leq\frac{Cp^2}{A_1(T)}\|w_1\|_{H^2}^2\|\Phi\|_{L^2}^{\frac{1}{2}}\|\Phi\|_{L^6}^{\frac{3}{2}}\\
&\leq\frac{Cp^2C_1(T)}{A_1(T)}\|\Phi\|_{L^2}^{\frac{1}{2}}\|\Phi\|_{H^1}^{\frac{3}{2}}\\
&\leq\frac{3}{4}\int_{B_R} |\nabla\Phi|^2+\left(\frac{Cp^8C_1^4(T)}{A_1^4(T)}+\frac{Cp^2C_1(T)}{A_1(T)}\right)\int_{B_R}|\Phi|^{2}.\end{split}\]
It hence follows from \eqref{F1} that
\begin{equation}\label{F2}
\begin{split}
\frac{d}{dt}\int_{B_R}|\Phi|^{2}+\frac{1}{4}\int_{B_R} |\nabla\Phi|^2\leq C(T)\int_{B_R} |\nabla\Psi|^2+C(T)\int_{B_R}\Phi^2.
\end{split}
\end{equation}
Multiplying the second equation of \eqref{unique} by $2\Psi$ and integrating the equation on $B_R$, we get
\begin{equation}\label{F3}
\begin{split}
\frac{d}{dt}\int_{B_R}|\Psi|^{2}+2\varepsilon\int_{B_R} |\nabla\Psi|^2&=-2\int_{B_R}w_1\Phi\Psi-2\int_{B_R}u_2\Psi^2\\&
\leq\|w_1\|_{L^\infty(Q_T)}\int_{B_R} |\Phi|^2+\|w_1\|_{L^\infty(Q_T)}\int_{B_R} |\Psi|^2\\
&\leq CC_1(T)\int_{B_R} (|\Phi|^2+ |\Psi|^2).
\end{split}
\end{equation}
Multiplying \eqref{F3} by a large constant $K$ and combining the equation with \eqref{F2}, we obtain
\[\frac{d}{dt}\int_{B_R}(|\Phi|^2+ K|\Psi|^2)\leq C(T)\int_{B_R}(|\Phi|^2+ K|\Psi|^2).\]
Noting $(\Phi,\Psi)(x,0)=(0,0)$, it follows from the Gronwall's inequality that
\[\int_{B_R}(|\Phi|^2+ K|\Psi|^2)=0 \text{ for any } t\in[0,T],\] which implies
\[(\Phi,\Psi)(x,t)\equiv(0,0) \text{ on } Q_T,\]
and hence $(u_1,w_1)(x,t)\equiv(u_2,w_2)(x,t)$ on $Q_T$.

\emph{Step 2.} We next construct a global solution for the system \eqref{KS} in $Y_T$. Recall that $v_0$ and $w_0$ satisfy
\begin{equation}\label{F4}
v_0-V=\ln\left(1+\frac{w_0-W}{W}\right).\end{equation}
By the Sobolev imbedding theorem, it holds that $\|w_0-W\|_{L^\infty}\leq C\|w_0-W\|_{H^2}$. Then with $\underset{x\in \overline{B_R}}{\min}W(x)>0$, when $\|w_0-W\|_{H^2}\ll1$, by Taylor's formula, we get from \eqref{F4} that
\[\|v_0-V\|_{H^2}\leq C_0\|w_0-W\|_{H^2}.\]
Now we take $\delta_1=\frac{\delta_2}{1+C_0}$, where $\delta_2$ is the constant obtained in Theorem \ref{v-stability}. If $(u_0,w_0)$ satisfies
\[\|(u_0-U,w_0-W)\|_{H^2}\leq\delta_1,\]
then the initial value of the system \eqref{trans}  satisfies
\[\|(u_0-U,v_0-V)\|_{H^2}\leq(1+C_0)\|(u_0-U,w_0-W)\|_{H^2}\leq (1+C_0)\delta_1=\delta_2.\]
Thus, by Theorem \ref{v-stability}, the system \eqref{trans} has a unique global radial solution $(u,v)$ satisfying $u(x,t)>0$, $(u,v)\in\C([0,+\infty);H^2(B_R))$ and the exponential convergence \eqref{v-convergence}. Now we define $w(x,t)=e^{v(x,t)}$, then it is easy to verify that $(u,w)\in\C([0,+\infty);H^2(B_R))$ is a solution of
the original system \eqref{KS}. Moreover, since $w-W=e^v-e^V=e^V(e^\psi-1)$ due to \eqref{anti}, we further have by the Taylor's formula again
\[\|w-W\|_{H^2}^2\leq C\|\psi\|_{H^2}^2\leq Ce^{-\mu t}(\|u_0-U\|_{H^2}^2+\|\psi_0\|_{H^2}^2),\]
which, along with the convergence of $(u-U)$ in Theorem \ref{v-stability}, gives the estimate \eqref{3.1}. We complete the proof.
\end{proof}

\vspace{4mm}

 \section*{Acknowledgement} 
JAC was supported by the Advanced Grant Nonlocal-CPD (Nonlocal PDEs for Complex Particle Dynamics: Phase Transitions, Patterns and Synchronization) of the European Research Council Executive Agency (ERC) under the European Union’s Horizon 2020 research and innovation programme (grant agreement No. 883363).
JAC was also partially supported by the EPSRC grant number EP/V051121/1 and by the “Maria de Maeztu” Excellence Unit IMAG, reference CEX2020-001105-M, funded by MCIN/AEI/10.13039/501100011033/.
The research of J. Li was supported by the National Science Foundation of China (No. 12371216) and the Natural Science Foundation of Jilin Province (No. 20210101144JC).  The research of Z.A. Wang was supported in part by the Hong Kong RGC GRF grant No. 15306121 and an internal grant ZZPY.  The research of W. Yang was supported by National Key R\&D Program of China 2022YFA1006800, NSFC No.12171456, NSFC No.12271369, FDCT No.0070/2024/RIA1, No. MYRG-GRG2024-00082-FST and No. SRG2023-00067-FST.

\bibliographystyle{plain}

\end{document}